\documentclass[11pt, a4paper, headings = small]{scrartcl}
\usepackage{RGDM_preamble}
\usepackage{cancel}

\makeatletter
\DeclareRobustCommand{\cev}[1]{%
  \mathpalette\do@cev{#1}%
}
\newcommand{\do@cev}[2]{%
  \fix@cev{#1}{+}%
  \reflectbox{$\m@th#1\vec{\reflectbox{$\fix@cev{#1}{-}\m@th#1#2\fix@cev{#1}{+}$}}$}%
  \fix@cev{#1}{-}%
}
\newcommand{\fix@cev}[2]{%
  \ifx#1\displaystyle
    \mkern#23mu
  \else
    \ifx#1\textstyle
      \mkern#23mu
    \else
      \ifx#1\scriptstyle
        \mkern#22mu
      \else
        \mkern#22mu
      \fi
    \fi
  \fi
}

\title{\fontsize{16}{19} \selectfont Statistical Convergence of Spherical First Hitting Diffusion Models}
\author{Simon Bienewald\thanks{University of Stuttgart, Institute for Stochastics and Applications, Wankelstraße 5, 70563 Stuttgart, Germany. \newline Email: \href{mailto:simon.bienewald@isa.uni-stuttgart.de}{simon.bienewald@isa.uni-stuttgart.de}}
 \and Lukas Trottner\thanks{University of Stuttgart, Institute for Stcohastics and Applications, Wankelstraße 5, 70563 Stuttgart, Germany. \newline Email: \href{mailto:lukas.trottner@isa.uni-stuttgart.de}{lukas.trottner@isa.uni-stuttgart.de}}}
\begin{document}

\maketitle

\begin{abstract}
    Denoising diffusion models have evolved into a state-of-the-art method for tasks in various fields, such as denoising and generation of images, text generation, or generation of synthetic data for training of other machine learning models. First hitting diffusion models (FHDM) are a particular class of denoising diffusion models with \textit{random} adaptive generation time tailored to generate data on a known manifold. Building on the conditioning framework of Doob's $h$-transform these models leverage the given information on the target data manifold to demonstrate strong performance across tasks while offering distinct features such as time-homogeneous dynamics of the generating process and a reduced average simulation time. Even though the theoretical investigation of standard forward-backward diffusion models has attracted much attention in the recent past, the statistical convergence properties of FHDMs are not yet understood. In this work, we show that, up to logarithmic factors, FHDMs achieve the minimax optimal convergence rate in total variation for spherically supported Sobolev smooth data distributions. In particular, this is the first statistical optimality result for denoising diffusion modelling with random generation time.
\end{abstract}

\section{Introduction}
\label{sec:introduction}

Generative modelling has rapidly gained in importance in recent years, with applications including the generation of hyperrealistic images and videos, text and also synthetic data for training of large machine learning models. Many state of the art generative models are variants of score-based denoising diffusion models (DDMs) \cite{song21}, which train neural networks to learn the drift of a stochastic diffusion process that is initialised in some easy-to-sample-from distribution and terminates in the targeted data distribution. The most widely used procedure simulates forward SDEs with Gaussian transition kernels initalised in the training data until approximate convergence to a prescribed Gaussian distribution to create a training data set of paths. Based on a denoising score matching procedure this data set is then used to learn the \textit{score} of the forward process (the $\log$-gradient of the forward marginal densities), which characterises the dynamics of the reverse (generating) process and contains the statistical information on the targeted unknown data distribution. Finally, new samples are generated as the terminal value of the approximated reverse SDE initialised in the approximate terminal distribution of the forward process.

A drawback of denoising diffusion models compared to one-step sampling procedures such as GANs or variational autoencoders is the numerically expensive and highly sensitive generative procedure that requires many iterative sampling steps to approximate the reverse SDE on a fine discretised time-grid. Moreover, because of the time-inhomogeneous nature of the reverse process, the algorithm does not adapt to the noise level along generated paths since it relies on a fixed deterministic simulation time that needs to be sufficiently large for the forward process to terminate in a distribution that is approximately independent of the unknown data distribution. 

To address such issues when some prior geometric knowledge on the data support is given,  \cite{ye2022first} introduced First Hitting Diffusion Models (FHDMs). While FHDMs, similarly to DDMs, use an iterative SDE based sampling procedure, they do not rely on a time-reversal mechanism but aim to  condition a simple SDE, such as a Brownian motion, to terminate in the target distribution $\Pi^\ast$ at the stochastic first hitting time of the known data manifold. This is achieved by learning the log-gradient of a Doob $h$-transform \cite{Chung_Walsh_2005} (the analogoue of the score in DDMs) that is fitted to $\Pi^\ast$ via a denoising score matching procedure that involves drawing first hitting bridges to the given data at the training stage. For instance, FHDMs have been shown in \cite{ye2022first} to produce high quality point clouds in $\R^3$, climate data on the sphere, unweighted graphs, and image segmentations. Importantly, the authors demonstrate that leveraging the geometric information on the data in the design of the  model leads to better generation performance compared to conventional DDMs, while the average (here \textit{random} first-hitting) generation time can be significantly decreased.

\subsection*{Our contribution}
To the best of our knoweldge, no theoretical analysis for FHDMs is available that can explain their impressive algorithmic performance. In this work, we therefore provide the first statistical convergence analysis of FHDMs for the particular case of data supported on the unit sphere $\partial B \subset \R^d$ in dimension $d \geq 3$. Specifically, given i.i.d.\ data $\{X_1,\ldots,X_n\}$ with $\alpha$-Sobolev smooth density ($\alpha \in \N \cap ((d-1)/2,\infty)$) on the sphere that is uniformly bounded away from zero, we are interested in the total variation convergence rate  of FHDMs in terms of the size $n$, dimension $d$ and smoothness $\alpha$ of the data. Approximating the score $\nabla \log h$ with  sparse ReLU-neural network functions, our main result, Theorem \ref{theo:main}, demonstrates that the output distribution $\hat{\Pi}_n$ of the learned model with an appropriate choice of $n$-dependent network sizes, converges at rate 
\[\E[\mathrm{TV}(\Pi^\ast, \hat{\Pi}_n)] \lesssim n^{-\frac{\alpha}{2\alpha + d-1}} (\log n)^{7/2}.\]
Since the sphere is $(d-1)$-dimensional, this matches, up to small $\log$-factors, the minimax optimal nonparametric rate for the data distributions under consideration. We thus prove statistical optimality of FHDMs, providing a theoretical explanation of their strong experimental performance. Moreover, this is  the first statistical optimality result on generative modelling with adaptive random generation time.

\subsection*{Related work}
Statistical theory for neural network based DDMs has made tremendous advances in recent years since the seminal work of \cite{oko23}, who were the first to prove statistical optimality in total variation of DDMs with Ornstein--Uhlenbeck forward dynamics for Besov-smooth and lower bounded data densities  on the cube $[-1,1]^d$. They also show that DDMs can implicitly adapt to the data geometry by proving almost optimal 1-Wasserstein convergence rates for data supported in a lower-dimensional linear subspace, see also \cite{chen23b}. The latter result was further refined in \cite{tang24} and \cite{Rousseau2025}, where adaptivity of DDMs to more general  low-dimensional manifold structures with  almost optimal convergence rate in 1-Wasserstein distance is demonstrated, thus providing statistical evidence for the ability of DDMs to overcome the curse of dimensionality in light of the \textit{manifold hypothesis} \cite{fefferman16}. For other classes of data distributions with intrinsic dimensionality $d^\ast$ smaller than the ambient dimension $d$, \cite{Yakovlev2025,fan25,kwon26} also demonstrate that the convergence rate of standard DDMs is only $d^\ast$-dependent. Notably, the results of \cite{Yakovlev2025} do not require lower-bound assumptions on the data distribution, which is an important feature that has also been demonstrated for full-dimensional supports in \cite{zhang24} for kernel-based score approximations and more recently in \cite{Stephanovitch2025} for neural network approximations with $\tanh$-activations. While all of the works above consider DDMs in their standard formulation with Gaussian transitions, some recent progress has also been made for statistical convergence of iterative generative models, which in a similar spirit to FHDMs, are designed to explicitly enforce geometric constraints on the data distribution such as \cite{Asbjorn_2025,holk26} for \textit{reflected} diffusion models \cite{lou23} and \cite{wakasugi26} for discrete diffusion models \cite{lou24}.

\subsection*{Notation}
For $\alpha\in\N$ and any open domain $D\in\R^d$, we let $H^\alpha(D)$ be the Sobolev spaces of $\alpha$-times weakly differentiable functions with square-integrable weak derivatives. We denote by $\lVert\cdot\rVert_{L^p(D)}$, $p \in [1,\infty]$ the usual functional $L^p$-norms and let $\langle\cdot,\cdot\rangle$ be the standard scalar product in $L^2(D)$. For $z\in\R^d$, we write $\lVert\cdot\rVert$ for the Euclidean norm, $\lVert\cdot\rVert_\infty$ for the maximum norm, and $\lVert\cdot\rVert_0$ for the number of non-zero values. For $f\in H^\alpha(D)$ we write  $\partial_i f$ for the weak derivative of $f$ wrt\ the $i$\textsuperscript{th} component. Finally,  $\mathrm{TV}(\mu,\nu) \coloneq \sup_{A \in \mathcal{F}} \lvert \mu(A) - \nu(A) \rvert$ denotes the total variation distance between two probability measures $\mu,\nu$ on a measurable space $(\Omega,\cF)$.

\subsection*{Structure of the paper}

In Section~\ref{sec:main}, we first introduce the necessary mathematical background on FHDMs, before introducing our assumptions on the target distribution in Section~\ref{subsec:ass}. After explaining and theoretically justifying the training and generation procedure of the model in Section~\ref{subsec:generation_and_estimation}, we present our formal main result in Section~\ref{subsec:main}. Section~\ref{sec:proofs} gives an overview of the central aspects of the proof of the main result and Section~\ref{sec:discussion} discusses our results and their limitations. The appendix contains additional technical results as well as all proofs, with explicit pointers to their locations given in the main part.

\section{Setting and main result}\label{sec:main}

We first introduce the technical $h$-transform setup of \cite{ye2022first} and provide some additional mathematical details that are needed for the statistical analysis. 
Let $\Pi^\ast$ be the target distribution, supported on the boundary of the (open) unit ball $B\subset \R^d$, where  $d\in\N \cap [3,\infty)$. Let also $W$ be a Brownian motion started in $z \in \overline{B}$ under $\PP^z$ and $Z$ be the Brownian motion absorbed in $\partial B$, that is 
\begin{equation}
    \diff Z_t=\one_B(Z_t)\diff W_t, \quad Z_0 = W_0 \in \overline{B}.
\end{equation}
We let $\mathbb{F} = (\mathcal{F}_t)_{t \geq 0}$ be the $(\PP^z)_{z \in \overline{B}}$-completed, and hence right-continuous, natural filtration of the continuous $\overline{B}$-valued Markov process $Z$. Furthermore,  denote by $\Q^z$ the law of $Z$ started in a fixed location $z \in \overline{B}$. Thus, if we let $\xi = (\xi)_{t \geq 0}$ be the canonical process on $C([0,\infty); \R^d)$ 
and define $\mathbb{G} = (\mathcal{G}_t)_{t \geq 0}$ to be the $(\mathbb{Q}^z)_{z \in \overline{B}}$-completion of the natural filtration of $\xi$ and $\mathcal{G}_\infty \coloneqq \sigma(\bigcup_{t \geq 0} \mathcal{G}_t)$, then $\xi$ is a Brownian motion stopped in $\tau(\xi) \coloneq \inf\{t \geq 0: \xi_t \in \partial B\}$ and started in $z$ on $(C([0,\infty); \R^d), \mathcal{G}_\infty, \mathbb{G}, \mathbb{Q}^z)$. 

The idea of \cite{ye2022first} is to transform this process into a process $Z^{h}$ using Doob's $h$-transform to achieve $Z^{h}_{\tau^h}\sim\Pi^\ast$, where $\tau^h\coloneqq \inf\{t\geq0: Z_t^{h}\in \partial B \} = \tau(Z^h)$ is the first hitting time of the unit sphere. To this end, we abbreviate $\tau \coloneq \tau(W) = \tau(Z)$ for the first hitting time of $\partial B$ (or, equivalently, the first exit time of $B$) by the Brownian motion $W$, and define
\[\Q_{\partial B}(z,\diff{x}) \coloneq \PP^z(W_\tau \in \diff{x}), \quad z \in B, x \in \partial B, \] 
to be the Poisson kernel of the Brownian motion.  The Poisson kernel can be decomposed as $\Q_{\partial B}(z,\diff{x}) = q(x \mid z) \, \sigma(\diff{x})$, where $\sigma$ denotes the $(d-1)$-dimensional Hausdorff measure and $q$ is given by \cite{Gilbarg_Trudinger2001}
\begin{equation}\label{Poisson_kernel}
    q(x \mid z) \coloneqq \frac{1}{\sigma(\partial B )}\frac{\lVert x\rVert^2- \lVert z \rVert^2}{\lVert x\rVert\,\lVert z - x\rVert^d} = \frac{1}{\sigma(\partial B )}\frac{1 - \lVert z \rVert^2}{\lVert z - x\rVert^d}, \quad z \in B, x \in \partial B.
\end{equation}
We assume that the target measure is absolutely continuous wrt\ the Poisson kernel with a fixed initialisation $z_0 \in B$ and let $\pi^\ast \coloneq \diff{\Pi^\ast}/\diff{\Q_{\partial B}(z_0,\cdot)}$ be its density. We then choose the function $h$ that induces the desired $h$-transform as 
\[h(z) \coloneq \E^z[\pi^\ast(W_\tau)] = \E^z[\pi^\ast(Z_\tau)], \quad z \in \overline{B},\]
which can be further represented by
\[
    h(z) = \begin{cases} \int_{\partial B} \pi^\ast(x) \, \Q_{\partial B}(z,\diff{x}) = \int_{\partial B} \pi^\ast(x) q(x \mid z) \, \sigma(\diff{x}) = \int_{\partial B} \frac{q(x \mid z)}{q(x \mid z_0)}\, \Pi^\ast(\diff{x}),  &z \in B \\ \pi^\ast(z), &z \in \partial B .\end{cases}
\]
Furthermore, we will assume that $\sup_{y \in \partial B} \pi^\ast(y) < \infty$ and $\inf_{y \in \partial B} \pi^\ast(y) > 0$, which implies that $h$ is bounded and that $h > 0$ on $\overline{B}$. The function $h$ is harmonic for $Z$ in the sense that for any $z \in \overline{B}$ and $t \geq 0$ it holds that $\E^z[h(Z_t)] = h(z)$, which follows from the following calculation using the strong Markov property of $W$ and that $\tau$ is an $\F$-stopping time,
\begin{align*}
\E^z[h(Z_t)] &= \E^z\big[\E^{W_t}[\pi^\ast(W_\tau)] \one_{\{t < \tau\}}\big] + \E^z[\pi^\ast(W_\tau) \one_{\{\tau \leq t\}}]\\ 
&= \E^z\big[\E^z[\pi^\ast(W_\tau) \mid \mathcal{F}_t] \one_{\{t < \tau\}}\big] + \E^z[\pi^\ast(W_\tau) \one_{\{\tau \leq t\}}]\\ 
&= \E^z[\pi^\ast(W_\tau)] = h(z), \quad z \in \overline{B}.
\end{align*}
In particular the Markov property of $Z$ implies that $(h(Z_t))_{t \geq 0}$ is an $\F$-martingale under any $\PP^z$. 
This allows us to use Doob's $h$-transform technique \cite[Chapter 11]{Chung_Walsh_2005} to define  Markov transition kernels 
\[P_t^h(z,A) \coloneq \frac{1}{h(z)} \E^z[h(Z_t) \one_A(Z_t) ], \quad z \in \overline{B}, A \in \mathcal{B}(\overline{B}),\]
with the associated Markov process $Z^h$ being again of diffusion type with a version on $(\Omega,\mathcal{F}, \mathbb{F},\mathbb{P})$, which is the continuous solution of the absorbed SDE
\begin{equation}\label{Z_pi}
    \diff Z^{h}_t=\one_B(Z_t^{h})(\nabla\log h(Z_t^{h})\diff t + \diff W_t), \quad Z^h_0 = Z_0.
\end{equation}
In particular, $Z^h_0 =z$, $\PP^z$-a.s.\ for all $z \in \overline{B}$. In analogy to denoising diffusion models, we call the additional term $\nabla\log h$ arising in \eqref{Z_pi} the  \emph{score}. We denote by $\Q^{h,z}$ the law of $Z^h$ under $\PP^z$ and set $\Q^{\Pi^\ast} \equiv \Q^h \equiv \Q^{h,z_0}$ to indicate that $Z^h$ has terminal distribution $\Pi^\ast$ when started in $z_0$. \cite[Theorem 11.9]{Chung_Walsh_2005} now yields for any $\mathbb{G}$-stopping time $T$ that
\begin{equation} \label{eq:h_path}
    \Q^{h,z}(\Lambda) = \frac{1}{h(z)} \int_{\Lambda} h(\xi_T) \, \diff \mathbb{Q}^z, \quad \Lambda \in \mathcal{G}_T, z \in \overline{B}.
\end{equation}
Noting that $h(z_0) = 1$, this yields 
\begin{equation}\label{eq:h_path_2}
    \Q^{\Pi^\ast}(\Lambda)=\int_\Lambda h(\xi_T) \diff{\Q^{z_0}},  \quad  \Lambda \in\mathcal G_T,
\end{equation}
for any $\mathbb{G}$-stopping time $T$ and since $\xi_t = \xi_\tau$ and $\tau < \infty$, $\Q^{z_0}$-a.s., we obtain from this that generally
\begin{equation}\label{eq:h_path_3}
    \Q^{\Pi^\ast}(\Lambda)=\int_\Lambda \pi^\ast(\xi_\tau) \diff{\Q^{z_0}},  \quad  \Lambda \in \mathcal{G}_\infty.
\end{equation}
A \textit{formal} calculation via disintegration then gives
\begin{equation}\label{eq:disintegration}
\begin{split}
\Q^{\Pi^\ast}(\diff{\omega}) = \pi^\ast(\xi_\tau(\omega)) \, \Q^{z_0}(\diff{\omega}) &= \int_{\partial B} \frac{\diff{\Pi^\ast}}{\diff \Q^{z_0}(\xi_{\tau} \in \cdot)}(x)\,  \Q^{z_0}(\xi_{\tau} \in \diff{x})\, \Q^{z_0}(\diff{\omega} \mid \xi_{\tau} = x)\\ 
&= \int_{\partial B} \Q^{z_0}(\diff \omega \mid \xi_{\tau} = x) \, \Pi^\ast(\diff{x}) \\ 
&= \int_{\partial B} \Q^{q(x \mid \cdot)}(\diff{\omega}) \, \Pi^\ast(\diff{x}), \quad \omega \in C([0,\infty); \R^d).
\end{split}
\end{equation}
where $\Q^{q(x \mid \cdot)}$ denotes the law of the $q(x \mid \cdot)$-transform of the Brownian motion $W$ killed on first hitting $\partial B$, started in $z_0$; see also \cite[equation (8)]{ye2022first}. For $x \in \partial B$ and $Z^x_0 = z_0$, this transform can be realised by the SDE 
\begin{equation}\label{eq:bridge}
\diff{Z^x_t} = \nabla_2 \log q(x \mid Z^x_t) \diff{t} + \diff{W_t}, \quad Z^x_0 = z_0, \, t < \tau^x \coloneq \inf\{t \geq 0: Z^x \in \partial B\},
\end{equation}
where $\nabla_2 \log q(x \mid z)$ points with increasing force towards $x$ as $z \to \partial B$, thus conditioning the SDE to hit the unit sphere in $x$ just before being killed. Making \eqref{eq:disintegration}  precise is technically delicate because of the pole of $q(x \mid \cdot)$ at $x$. We will not need or attempt to formally prove  this statement thanks to an early stopping procedure that we introduce for the generation algorithm and the statistical analysis. For our purposes it therefore suffices that \eqref{eq:bridge} is a well-behaved SDE up to the first hitting-time of the boundary of any slightly smaller ball $B_{1-\varepsilon}$ for some $\varepsilon \in (0,1)$.

The above discussion motivates a  natural sampling procedure for the $h$-transform \eqref{Z_pi}: we first sample $x$ from $\Pi^\ast$ on $\partial B$ and then simulate a first hitting bridge $Z^x$ according to \eqref{eq:bridge}. This procedure is of central importance for score-based generative modelling because the true score $\nabla \log h$ is implicitly defined via the unknown data distribution $\Pi^\ast$ and is thus inaccesible in practice. The score therefore needs  to be approximated by a learnable function $s$ based on paths of $Z^h$, which given the data set $X_1,\ldots, X_n \overset{\text{iid}}{\sim} \Pi^\ast$ can be generated according to the procedure above, where  sampling from $\Pi^\ast$ in the first step is replaced by sampling from the empirical data distribution. The approximate posterior process $Z^s$ is then generated  according to
\begin{equation*}
    \diff Z_t^s= \one_B(Z_t^s)(s(Z_t^s)\diff t + \diff W_t), \quad Z_0^s=z_0\in B \label{Z_theta}.
\end{equation*}

We will first show that the score can be expressed as a conditional expectation, which will prove to be a useful characterisation for later purposes. To this end, we have the following formula for the distribution of $Z^h$ at its terminal time that further motivates the choice of $h$.

\begin{lemma} \label{lem:terminal}
    It holds that 
    \begin{equation*}
            \PP^z(Z^h_{\tau^h} \in \diff{x}) = \frac{\pi^\ast(x)}{h(z)} \, \Q_{\partial B}(z,\diff{x}) = \frac{\pi^\ast(x) q(x \mid z)}{h(z)} \, \sigma(\diff{x}), \quad  \quad z \in B, x \in \partial B.
        \end{equation*}
    In particular, $\PP^{z_0}(Z^h_{\tau^h} \in \diff{x}) = \Pi^\ast(\diff{x})$.
\end{lemma}

 Lemma \ref{lem:terminal}  immediately yields the following denoising score representation. 

\begin{lemma}\label{lem:score}
\[\nabla \log h(z) = \E^z\big[ \nabla_2 \log q(Z^h_{\tau^h} \mid Z^h_0)\big], \quad z \in B.\]
\end{lemma}
The proofs of both lemmata are given in Appendix \ref{app:main}.

\subsection{Assumptions on the target distribution $\Pi^\ast$} \label{subsec:ass}

Our goal in this paper is to prove optimal nonparametric convergence rates of spherical first hitting diffusion models under Sobolev regularity assumptions on the target distribution $\Pi^\ast$. To this end, we first introduce the necessary background on (Riemannian) analysis on the sphere and then state our regularity assumptions on the data density $\pi^\ast$.

The space of square-integrable functions on $\partial B$ wrt\ the surface measure $\sigma$ is defined by 
\begin{equation*}
    L^2(\partial B) \coloneqq \left\{u \colon \partial B\to \R: \int_{\partial B} \lvert u \rvert^2 \diff\sigma<\infty \right\},
\end{equation*}
which we equip with the inner product
\begin{equation*}
    \langle u,v\rangle_{L^2(\partial B)} \coloneqq\int_{\partial B} uv\diff\sigma, \qquad u,v\in L^2(\partial B).
\end{equation*}
For notational simplicity we will omit the subscript of the scalar product as long as there is no risk of confusion with the usual scalar product on $L^2(B)$.

\begin{definition}[Laplace--Beltrami operator]
    For a differentiable function $u$, the \emph{Laplace--Beltrami operator} on the unit sphere can be defined by \cite{Shubin1987}
    \begin{equation*}
        -\Delta_{\partial B} u(x)\coloneqq -\Delta_z u(z/\lVert z\rVert)|_{z=x}, \qquad x\in\partial B,
    \end{equation*}
    for 
    \begin{align*}
        u\in\mathrm{Dom}(\Delta_{\partial B}) \coloneqq \{u\in L^2(\partial B): \R^d\ni z\mapsto u(z/\lVert z\rVert) \text{ is twice weakly differentiable} \}.
    \end{align*}
\end{definition}

The Laplace--Beltrami operator is a self-adjoint operator wrt\ the  inner product on the sphere, with eigenvalues $l(l+d-2)$, $l\in\N$, and corresponding orthonormal eigenfunctions $Y_{lm}$, $m=1,\ldots,M_l$, where $M_l$ denotes the dimension of the eigenspace to the $l$\textsuperscript{th} eigenvalue. The $Y_{lm}$ are the so-called \emph{spherical harmonics} and are well-known from harmonic analysis, Riemannian geometry and representation theory. Due to the self-adjointness of $\Delta_{\partial B}$, they build a complete basis of $L^2(\partial B)$. We are now ready to introduce the space of Sobolev functions on the sphere. In our context, it will be most convenient to use the following characterisation.

\begin{definition}[Sobolev space on the sphere]
    For $\alpha\in\N$, the \emph{Sobolev space on the (unit) sphere} is given by
    \begin{equation*}
        H^\alpha(\partial B)\coloneqq \left\{u\in \mathrm{Dom}(\Delta_{\partial B}): \lVert u\rVert_{H^\alpha(\partial B)} <\infty \right\},
    \end{equation*}
    with the norm
    \begin{equation*}
        \lVert u\rVert_{H^\alpha(\partial B)} \coloneqq \big(\lVert u\rVert_{L^2(\partial B)}^2 + \langle u, (-\Delta_{\partial B})^\alpha u\rangle_{L^2(\partial B)}\big)^{1/2} = \left(\sum_{l=0}^\infty\sum_{m=1}^{M_l} [1+(l(l+d-2))^\alpha]\, \lvert \langle u,Y_{lm} \rangle\rvert^2\right)^{1/2}.
    \end{equation*}
\end{definition}

This will allow us to easily relate the Sobolev smoothness of the score function $\nabla\log h$ to the Sobolev smoothness of the data density $\pi^\ast$, thus determining the speed of convergence of the neural network approximation. With this preparation, we can now introduce the following assumptions on the target distribution, which we assume to hold for the remainder of the paper without further mention.

\begin{enumerate}[label = ($\mathcal{H}$\arabic*), ref = ($\mathcal{H}$\arabic*)]
    \item\label{ass:smoothness_1} $\pi^\ast\in H^\alpha(\partial B)$ for $\alpha\in\N$ and $\alpha>(d-1)/2\geq 1$.
    \item\label{ass:smoothness_2} $\pi^\ast$ is uniformly bounded away from zero, that is, there exists a constant $\pi_{\min} > 0$ such that $ \pi_{\mathrm{min}}\leq \pi^\ast$.
\end{enumerate}

According to the Sobolev embedding theorem for manifolds \cite[Theorem~2.20]{Aubin_1998}, the Sobolev-smoothness assumption \ref{ass:smoothness_1} implies that $\pi^\ast$ is continuous and thus bounded from above as well. The lower bound\-ed\-ness assumption \ref{ass:smoothness_2} on the data distribution is typical in the literature on statistical convergence rates of diffusion models \cite{oko23, tang24, Rousseau2025, Asbjorn_2025} and simplifies the approximation analysis of the score function $\nabla \log h = \nabla h/h$ considerably since it implies that $h = \int_{\partial B} \pi^\ast(x) \, \Q_{\partial B}(\cdot, \diff{x}) \geq \pi_{\min} > 0$. Consequently, we only consider target distributions $\Pi^\ast$ with full support on the unit sphere $\partial B$.

\subsection{Generation and estimation strategy}\label{subsec:generation_and_estimation}
Similarly to forward-backward diffusion models with fixed time horizon, the main idea to approximate the unknown score $\nabla \log h$ is to use Girsanov's theorem \cite[Theorem 5.22]{LeGall_2016} to first express the KL divergence between the path measures induced by the approximating drift $s$ and the ``true'' drift $\nabla \log h$ in terms of an $L^2$-loss along the paths $t \mapsto Z^h_t$. By establishing an equivalence of this \textit{explicit} score loss to a  \textit{denoising} score loss that does not involve the unknown score $\nabla \log h$ but only known transition kernels (in our case simply the Poisson kernel $q(x \mid z)$ thanks to the time-homogeneous nature of the generating  process), a training objective is obtained, which (with a further Monte-Carlo approximation step) can be optimised, usually over a class of neural networks.

However, $\lVert \nabla \log h \rVert$ cannot be guaranteed to be bounded in a neighbourhood of $\partial B$, making the applicability of Girsanov's theorem and finiteness of the explicit score matching loss a delicate issue. To circumvent such problems, we do not aim at simulating $Z^h$ until its terminal time $\tau^h$, but stop a little early at first hitting of the $(1-\varepsilon)$-sphere $\partial B_{1-\varepsilon}$ for small $\varepsilon > 0$. To generate a data distribution with the correct support, the simulated value of the early stopped process on $\partial B_{1-\varepsilon}$  is then simply projected onto $\partial B$.

Let us therefore introduce $\tau_{1-\varepsilon}(\omega) \coloneq \inf\{t \geq 0: \omega_t \in \partial B_{1-\varepsilon}\}$ for $\omega \in \R^{[0,\infty)}$ and set $\tau^h_{1-\varepsilon} \coloneq \tau_{1-\varepsilon}(Z^h)$, $\tau^s_{1-\varepsilon} \coloneq \tau_{1-\varepsilon}(Z^s)$ as the first hitting times of $\partial B_{1-\varepsilon}$ by $Z^h$ and $Z^s$, respectively. If not said otherwise, we will always assume that $\varepsilon < (1- \lVert z_0 \rVert)/2$ so that $\partial B_{1-\varepsilon}$ is well separated from the initialisation $z_0$. Furthermore, let $\Q^h_{\partial B_{1-\varepsilon}}$ and $\P^s_{\partial B_{1-\varepsilon}}$ denote the laws of the processes $Z^h_{\tau_{1-\varepsilon}^h}$ and $Z^s_{\tau_{1-\varepsilon}^s}$ started in $z_0$, where $Z^s$ solves the absorbed SDE
\[\diff{Z^s_t} = \one_B(s(Z^s_t) \diff{t} + \diff{W}_t), \quad Z^s_0 = z_0.\]
for some locally Lipschitz approximating function $s$. Furthermore, let 
\begin{equation*}
    P_{\partial B}\colon B\setminus\{0\}\to\partial B, \quad  x\mapsto \frac{x}{\lVert x \rVert},
\end{equation*}
be the projection onto the unit sphere. We can now decompose the total variation distance between our target distribution $\Pi^\ast$ and the distribution of the projection of the simulated $Z^s_{\tau^s_{1-\varepsilon}}$ according to
\begin{equation}\label{eq:TV_decomposition}
\begin{split}
    \mathrm{TV}(\Pi^\ast, P_{\partial B}\sharp \P^s_{\partial B_{1-\varepsilon}}) &\leq \mathrm{TV}(\Pi^\ast,P_{\partial B}\sharp  \Q^h_{\partial B_{1-\varepsilon}}) + \mathrm{TV}(P_{\partial B}\sharp  \Q^h_{\partial B_{1-\varepsilon}},P_{\partial B}\sharp  \P^s_{\partial B_{1-\varepsilon}})\\ 
    &= \mathrm{TV}(\Pi^\ast,P_{\partial B}\sharp  \Q^h_{\partial B_{1-\varepsilon}}) + \mathrm{TV}(\Q^h_{\partial B_{1-\varepsilon}},\P^s_{\partial B_{1-\varepsilon}})
    \end{split}
\end{equation}
where the second line uses that the restriction of $P_{\partial B}$ to $\partial B_{1-\varepsilon}$ is a bijection between $\partial B_{1-\varepsilon}$ and $\partial B$. Note here that projecting $Z^s_{\tau^s_{1-\varepsilon}}$ onto $\partial B$ is necessary to obtain meaningful total variation bounds, since $\mathrm{TV}(\Pi^\ast, \PP^s_{\partial B_{1-\varepsilon}}) = 1$ because of disjoint supports. To minimise the rhs of \eqref{eq:TV_decomposition} for fixed $\varepsilon > 0$ and a given approximation class $\mathcal{S} \ni s$, we need to optimise $\mathrm{TV}(\Q^h_{\partial B_{1-\varepsilon}},\P^s_{\partial B_{1-\varepsilon}})$ or a suitable upper bound thereof. By Pinsker's inequality, we have
\[\mathrm{TV}(\Q^{\Pi^*}_{\partial B_{1-\varepsilon}},\P^s_{\partial B_{1-\varepsilon}})^2 \leq \frac{1}{2}\text{KL}(\Q^{\Pi^\ast}_\varepsilon \, \Vert \, \P^s_\varepsilon),\] 
and the KL-divergence has an explicit expression in terms of the drifts $s$ and $\nabla \log h$ according to the following proposition, whose proof can be found in Appendix \ref{app:main}.

\begin{proposition}\label{prop:bound_TV_Girsanov}
    Suppose that $s$ is locally Lipschitz. Then, $\Q_\varepsilon^h\approx \P_\varepsilon^s$ and
    \begin{equation}
        \mathrm{KL}(\Q^h_\varepsilon \, \Vert \, \P^s_\varepsilon) = \frac{1}{2}\E^{z_0}\Big[\int_0^{\tau_{1-\varepsilon}^h} \lVert s(Z_t^h)-\nabla \log h(Z_t^h) \rVert^2 \d t\Big]. \label{score_matching}
    \end{equation}
\end{proposition}

Since the explicit score matching loss on the rhs of \eqref{score_matching} depends explicitly on the unknown score $\nabla \log h$, we next derive a denoising score matching identity that allows us to replace $\nabla \log h(x)$ by the Poisson kernel $q(x \mid \cdot)$ for estimation purposes.

\begin{proposition}\label{prop:denoising_score_equivalence}
    Let $z \in B$, $\varepsilon \in (0,1)$ and suppose that $s$ is bounded on $B_{1-\varepsilon}$. For any $\varepsilon > 0$ and $\mathbb{F}$-stopping times $\underline{\tau} \leq \overline\tau \leq \tau^h_{1-\varepsilon}$, it holds that 
     \begin{equation}\label{eq:denoising_score}
        \E^z\Big[ \int_{\underline{\tau}}^{\overline\tau} \lVert \nabla \log h(Z^h_t) - s(Z^h_t)\Vert^2 \diff{t}\Big] = \E^z\Big[ \int_{\underline{\tau}}^{\overline{\tau}} \lVert \nabla_2 \log q(Z^h_{\tau^h} \mid Z_t^h) - s(Z^h_t)\Vert^2 \diff{t}\Big] + C,
    \end{equation}
    where the constant $C$ given by
    \begin{align*}
        C &\coloneq \E^z\Big[ \int_{\underline{\tau}}^{\overline{\tau}} \lVert \nabla \log h(Z^h_t)\Vert^2 \diff{t}]\Big] - \E^z\Big[ \int_{\underline{\tau}}^{\overline{\tau}} \lVert \nabla_2 \log q(Z^h_{\tau^h} \mid Z_t^h)\Vert^2 \diff{t}\Big]\\ 
        &= -\E^z\Big[ \int_{\underline{\tau}}^{\overline{\tau}} \lVert \nabla_2 \log q(Z^h_{\tau^h} \mid Z_t^h) - \nabla \log h(Z^h_t)\Vert^2 \diff{t}\Big],
    \end{align*}
    is independent of $s$. Furthermore, for $Z^x$ denoting the $q(x \mid \cdot)$-transform of $Z$ given by \eqref{eq:bridge} and $\tau^x_{r} \coloneq \tau_{r}(Z^x)$ for $r \in [0,1]$, we have the representation 
    \begin{equation}\label{eq:denoising_2}
    \E^{z_0}\Big[ \int_{0}^{\tau^h_{1-\varepsilon}} \lVert \nabla_2 \log q(Z^h_{\tau^h} \mid Z_t^h) - s(Z^h_t) \rVert^2 \diff{t}\Big] = \int_{\partial B} \E^{z_0}\Big[ \int_{0}^{\tau^x_{1-\varepsilon}} \lVert \nabla_2 \log q(x \mid Z_t^x) - s(Z^x_t) \rVert^2 \diff{t}\Big] \, \Pi^\ast(\diff{x}).
    \end{equation}
\end{proposition}

The proof is given in Appendix \ref{app:main}. Given the data sample $X_1,\dots,X_n \overset{\text{iid}}{\sim} \Pi^\ast$, this result suggests to approximate $\nabla \log h$ via the empirical risk minimiser $\hat s$ given by
\begin{equation}\label{eq:score_estimator}
    \hat{s} \coloneq \argmin_{s\in\mathcal S} \frac{1}{n}\sum_{i=1}^n L_s(X_i).
\end{equation}
for the denoising score loss function 
\begin{equation}
    L_s(x) = L_s^\varepsilon(x)\coloneqq \E^{z_0}\Big[\int_0^{\tau_{1-\varepsilon}^x} \lVert s(Z_t^x)-\nabla_2\log q(x\mid Z_t^x) \rVert^2 \diff t\Big], \quad x\in\partial B. \label{Ls}
\end{equation}
As approximation class $\mathcal{S}$ we choose sparse neural networks with ReLU activation function, which we now briefly introduce. For some $m\in\N$ and $v,b\in\R^m$, let the shifted activation function be
\begin{equation*}
    \sigma^b(v)\coloneqq \begin{pmatrix}
        \mathrm{ReLU}(v_1-b_1) \\
        \vdots \\
        \mathrm{ReLU}(v_m-b_m)
    \end{pmatrix}.
\end{equation*}
We then define the set of ReLU neural networks with number of layer $\mathsf L\in\N$, maximal width $\mathsf W\in \R_+$, sparsity contraint $\mathsf S\in\N$ and norm constraint $\mathsf B\in\R_+$, as
\begin{equation}\label{eq:neural_nets}
    \mathrm{NN}(\mathsf{L,W,S,B})\coloneqq \left\{\begin{array}{ll}
        A^{\mathsf L}\sigma^{b^{\mathsf L}}\cdots A^1\sigma^{b^1} A^0 : W\in\N^{\mathsf L+2},A^i\in\R^{W_{i+1}\times W_i}, b^i\in\R^{W_{i+1}}, \lVert W\rVert_\infty \leq \mathsf W, \\
        \sum_{i=1}^{\mathsf L} (\lVert A^i\rVert_0+\lVert b^i\rVert_0) + \lVert A^0\rVert_0\leq \mathsf S, \quad \max_{i=1,\dots,\mathsf L} (\lVert A^i\rVert_\infty\lor \lVert b^i\Vert_\infty)\lor\lVert A^0\rVert_\infty \leq \mathsf B
    \end{array}
    \right\}.
\end{equation}
We summarise some (by now) standard results on sparse neural network approximations from the literature in Appendix~\ref{app:neural}. 
Based on such neural networks, we define our approximation class $\mathcal{S}$ by
\begin{equation}\label{eq:approx_class}
        \mathcal S=\left\{\phi\in \mathrm{NN}(\mathsf{L,W,S,B}) : \lVert\phi(z)\rVert \leq \frac{12(d+2)}{1-\lVert z\rVert} \right\},
\end{equation}
which is motivated by the following growth bound on the score $\nabla \log h$, which is proved in Appendix \ref{app:main}.
\begin{lemma}[bound on the score]\label{lem:bound_on_the_score}
    For any $z\in B$ it holds that
    \begin{equation*}
        \lVert \nabla\log h(z)\rVert \leq \frac{d+2}{1-\lVert z\rVert}.
    \end{equation*}
\end{lemma}
With this preparation we are now in a position to state the full generative algorithm given in Algorithm~\ref{alg:alg_gen}.

\begin{algorithm}[ht]
   \caption{Generative algorithm}
   \label{alg:alg_gen}
\begin{algorithmic}
   \STATE {\bfseries Input:} data $\{X_1,\ldots,X_n\} \overset{\text{iid}}{\sim} \Pi^\ast$
   \STATE {$\bullet$} choose an early stopping parameter $\varepsilon$ and network class parameters $\mathsf{L,W,S,B}$ depending on the number $n$ of samples
   \STATE {$\bullet$} determine the empirical denosing score loss minimiser $\hat{s}$ according to \eqref{eq:score_estimator} with $\mathcal{S}$ as in \eqref{eq:approx_class} with the network parameters from step 1.
   \STATE {$\bullet$} for a Brownian motion $W$ independent of the data $\{X_1,\ldots,X_n\}$, simulate the SDE $Z^{\hat{s}}$
    \[\diff{Z^{\hat{s}}_t} = \hat{s}(Z^{\hat{s}}_t) \diff{t} + \diff{W_t}, \quad Z^{\hat{s}}_0 = z_0,\]
    until its first hitting time $\hat{\tau}_{1-\varepsilon}$ of the $(1-\varepsilon)$-sphere $\partial B_{1-\varepsilon}$.
   \STATE {\bfseries Output:} projected value $P_{\partial B} Z^{\hat{s}}_{\hat{\tau}_{1-\varepsilon}}$ as new (approximate) sample for $\Pi^\ast$.
\end{algorithmic}
\end{algorithm}

\begin{remark} \label{remark:generation}
\begin{enumerate}[label=(\roman*)]
\item 
Implicitly, the generative algorithm uses the simplifying assumption that we can evaluate the expectation $L_s(x)$ and that for an obtained score estimator $\hat{s}$ we can exactly simulate the corresponding absorbed SDE. In practice, $L_s(x)$ needs to be numerically approximated via a Monte-Carlo estimator based on simulated paths of $Z^x$. A fast simulation procedure for $Z^x$ with initialisation $z_0 = 0$ that is based on simple rotations of simulated paths of the unconditional stopped Brownian motion to the prescribed exit location $x \in \partial B$ is given in \cite[Proposition 2.11]{ye2022first}. Analysing the numerical convergence rate of such a Monte-Carlo estimator is technically challenging because of the path-dependent random upper limit of the path integral and out of scope of this work. The statistical analysis of the generative model taking into account sampling effects of both the training and generative procedure are therefore left to future work.  
\item  Imposing uniform growth restrictions on the neural networks from the approximation class $\mathcal S$ is a typical feature in statistical analysis of score-based generative models. For diffusion models with deterministic sampling horizon the growth is controlled in time \cite{oko23,tang24,Rousseau2025,Yakovlev2025,Asbjorn_2025,holk26} to match the explosive behaviour of the score close to termination of the algorithm. In our case, the score has no time component, but the distance to the sphere as our target manifold can be regarded as an intrinsic time-scale of the algorithm since it is proportional to the average time left until termination and controls the gowth of the pulling drift $\nabla \log h$ of the generating process. Our spatial growth restriction therefore serves as a natural analogue to estimation strategies in time-inhomogeneous models. Technically, such a condition is needed to control the covering number and the uniform bound of the class of loss functions $\{L_s: s\in\mathcal S\}$, which crucially determine the convergence rate of $\hat s$. The constant in the nominator is chosen to make the handling of constants simpler in the proof, but could, in principle, be chosen as an arbitrarily large number greater than $d+2$. 
\end{enumerate}
\end{remark}

\subsection{Main result}\label{subsec:main}
Our main result is the following.
\begin{theorem}\label{theo:main}
    Let $\varepsilon=n^{-\alpha/(\beta(2\alpha+d-1))}$ for $\beta=(\alpha-(d-1)/2)\land 1$. Then for $\Pi^\ast$ satisfying assumptions \ref{ass:smoothness_1} and \ref{ass:smoothness_2}, there exist neural network size parameters of order
    \begin{align*}
        &\mathsf L\lesssim \log^2 n, &&\mathsf W\lesssim n^{(d-1)/(2\alpha+d-1)}\log^2 n, \\
        &\mathsf S\lesssim n^{(d-1)/(2\alpha+d-1)}\log^3 n, &&\mathsf B\lesssim \mathrm{Poly(n)},
    \end{align*}
    such that Algorithm \ref{alg:alg_gen} with $\mathcal S$ chosen as in \eqref{eq:approx_class}  produces an output $P_{\partial B} Z^{\hat{s}}_{\hat{\tau}_{1-\varepsilon}}$ such that 
    \begin{equation*}
            \E[\mathrm{TV}(\Pi^\ast, P_{\partial B} \sharp \PP^{\hat s}_{\partial B_{1-\varepsilon}})] \lesssim n^{-\alpha/(2\alpha+d-1)}(\log n)^{7/2},
    \end{equation*}
    where $\mathrm{TV}(\Pi^\ast, P_{\partial B} \sharp \PP^{\hat s}_{\partial B_{1-\varepsilon}}) \coloneq \mathrm{TV}(\Pi^\ast, P_{\partial B} \sharp \PP^{s}_{\partial B_{1-\varepsilon}}) \vert_{s =\hat{s}}$ is the total variation distance between $\Pi^\ast$ and the law of $P_{\partial B} Z^{\hat{s}}_{\hat{\tau}_{1-\varepsilon}}$ given the estimator $\hat{s}$.
\end{theorem}

Our proof shows that the same result, but with a smaller logarithmic term $\mathcal O(\log^3n)$, holds true if the neural networks take spherical coordinates as input. By the following proposition, Theorem \ref{theo:main} demonstrates that up to log-factors, the convergence rate of FHDMs matches the minimax lower bound for total variation density estimation for spherical data.

\begin{proposition}\label{prop:lower_bound}
    Let $\mathcal{B}(H^\alpha(\partial B),L) \coloneq \{\pi \in H^\alpha(\partial B): \lVert \pi \rVert_{H^\alpha(\partial B)} \leq L, \int_{\partial B} \pi \diff{\sigma} = 1\}$ 
    be the space of $\alpha$-Sobolev probability densities on the sphere with Sobolev norm bounded by $L\in (0,\infty)$. Then, it holds that
    \begin{equation*}
        \inf_{\hat{\rho}}\sup_{\pi\in \mathcal{B}(H^\alpha(\partial B),L)} \E_{\pi}[\mathrm{TV}(\pi,\hat{\rho})] \gtrsim n^{-\alpha/(2\alpha+d-1)},
    \end{equation*}
    where the infimum is taken over all random probability densities $\hat{\rho}$ on $\partial B$ that are measurable wrt the data $\{X_1,\ldots,X_n\} \overset{\text{iid}}{\sim} \pi$ under $\PP_\pi$.
\end{proposition}

The general proof  of Theorem \ref{theo:main} follows a similar path to previous statistical work on diffusion models. However, the technical details are of a fundamentally different nature and impose new challenges because of the stochastic termination criterion, which places the analytic framework quite naturally into potential theory for Markov processes \cite{blumenthal68,Chung_Walsh_2005}. Before moving to the detailed intermediate results that the proof of Theorem \ref{theo:main} builds upon in the next section, we give a short breakdown of the proof below.

In a first step, we control the early stopping error in \eqref{eq:TV_decomposition} in Proposition \ref{prop:early_stopping}, where we show that the total variation distance between the target distribution $\Pi^\ast$ and the projected distribution of the $h$-transform stopped in $\partial B_{1-\varepsilon}$ is at most of order $\varepsilon^{\beta/2}$, where we recall that $\beta$ is the minimally guaranteed Hölder smoothness of the target density $\pi^\ast$. This allows us to focus on the second error component in \eqref{eq:TV_decomposition}, which is the generation error on the slightly smaller sphere $\partial B_{1-\varepsilon}$. By  Proposition~\ref{prop:bound_TV_Girsanov} and Pinsker's inequality, this error is characterised by the expected explicit score matching loss \eqref{score_matching}  of the empirical risk minimiser $\hat s$. Given the equivalence of explicit and denoising score matching loss (Proposition~\ref{prop:denoising_score_equivalence}), which is also the central component for the statistical analysis of denoising diffusion models \cite{oko23,tang25,Rousseau2025,Yakovlev2025,Asbjorn_2025,holk26}, this can be decomposed into an approximation error and an error influenced by the complexity of the class $\mathcal{L} = \{L_s:s\in\mathcal S\}$ of denoising score losses, cf.\ Theorem~\ref{thrm:oracle_inequality}, which need to be balanced via an appropriate selection of the network parameters to achieve the optimal rate.

The complexity term is determined by the covering number of  $\mathcal{L} = \{L_s:s\in\mathcal S\}$ built from $\mathcal{S}$ and the size of the uniform control $C(\mathcal{L}) = \sup_{s \in \mathcal{S}} \lVert L_s \rVert_{L^\infty(\partial B)}$. These components are controlled via Lemma~\ref{lem:bound_covering_number} in terms of the covering number of the approximation class $\mathcal{S}$ and Proposition~\ref{prop:bound_CS}, respectively, exploiting the growth rate of the neural networks and the score $\nabla \log h$ near $\partial B$. On the other hand, the approximation error of the neural network class is carried out in three main steps: first, the score is approximated on the $(d-1)$-dimensional coordinate space, yielding convergence rates depending on the manifold dimension $(d-1)$ instead of the ambient dimension $d$. Then, the coordinate map $\varphi_{\partial B}\colon \partial B\to \R^{d-1}$ is approximated, which is cheap in terms of neural network size due to the arbitrary smoothness of $\varphi_{\partial B}$. Finally, the two constructed neural networks are concatenated to approximate the score on the ambient space. This yields a parsimonious choice for sufficient network sizes that allows (up to log factors) approximation at rate $\mathcal O(n^{-\alpha/(2\alpha+d-1)})$, while keeping the complexity term at the same order, thus finishing the proof.

\section{Proof of the total variation convergence rate}\label{sec:proofs}

\subsection{Early-stopping bound} \label{sec:early_stopping}
We start with the error contribution by stopping the generation procedure early on $\partial B_{1-\varepsilon}$ and then projecting onto the sphere $\partial B$. This is described by the the first term in \eqref{eq:TV_decomposition} and can be bounded as follows in terms of the early stopping parameter $\varepsilon$ and the minimally guaranteed Hölder smoothness $\beta$ of the data density $\pi^\ast$.

\begin{proposition}\label{prop:early_stopping}
    Let $\beta\coloneqq (\alpha-(d-1)/2)\land 1$.
    Then,
    \begin{equation*}
        \mathrm{TV}(\Pi^\ast, P_{\partial B}\sharp  \Q^{h}_{\partial B_{1-\varepsilon}}) \leq C\varepsilon^{\beta/2}
    \end{equation*}
    for a constant $C$ independent of $\varepsilon$.
\end{proposition}
The proof can be found in Appendix \ref{app:early}.

\subsection{Generalisation error}\label{sec:generalise}
We now treat in detail the second error component in \eqref{eq:TV_decomposition}, which for $s = \hat{s}$ encodes the total variation error between the output of our generation procedure given in Algorithm \ref{alg:alg_gen} and the output generated by the true $h$-transform $Z^h$ stopped early in $\partial B_{1-\varepsilon}$. With Pinsker's inequality and Proposition \ref{prop:bound_TV_Girsanov}, this can be controlled in expectation via the following generalisation bound of the learned score $\hat{s}$.  The proof is based on Proposition \ref{prop:denoising_score_equivalence} and the general proof strategy of \cite[Theorem 4.3]{oko23} and is given in Appendix \ref{app:general}.

\begin{theorem}\label{thrm:oracle_inequality}
    Let $\delta>0$ and $\hat s$ be the empirical risk minimiser of \eqref{Ls} and the driving Brownian motion $W$ of $Z^h$ be independent of the data $(X_i)_{i=1,\ldots,n}$. If $\mathcal N(\mathcal L,\lVert \cdot\rVert_{L^\infty(\partial B)},\delta)\geq 3$, then 
    \begin{align}
        \begin{split}
            \E^{z_0}\Big[\int_0^{\tau_{1-\varepsilon}^h} \lVert \hat s(Z_t^h)-\nabla\log h(Z_t^h) \rVert^2 \diff t\Big] &\lesssim \inf_{s\in\mathcal S} \E^{z_0}\Big[\int_0^{\tau_{1-\varepsilon}^h} \lVert s(Z_t^h)-\nabla\log h(Z_t^h) \rVert^2\diff t\Big] \\
            &+ C(\mathcal L) \frac{\log\mathcal N(\mathcal L,\lVert \cdot\rVert_{L^\infty(\partial B)},\delta)}{n} + \delta, \label{oracle_inequality}
        \end{split}
    \end{align}
    with $C(\mathcal L)\coloneqq \sup_{s\in\mathcal S}\lVert L_s\rVert_{L^\infty(\partial B)}$.
\end{theorem}

For arbitrary approximation classes $\mathcal{S}$, the covering number of $\mathcal{L} = \mathcal{L}(\mathcal{S})$ can be controlled in terms of the covering number of $\mathcal{S}$.

\begin{lemma}[bound on the covering number]\label{lem:bound_covering_number}
    Let $C(\mathcal L)\coloneqq \sup_{s\in\mathcal S} \lVert L_s\rVert_{L^\infty(\partial B)}$. It holds
    \begin{equation*}
        \mathcal N(\mathcal L,\lVert \cdot\rVert_{L^\infty(\partial B)},\delta) \leq \mathcal N\big(\mathcal S,\lVert \cdot\rVert_{L^\infty(B_{1-\varepsilon})},c\delta/\sqrt{C(\mathcal L)}\big))
    \end{equation*}
    for a constant $c$ independent of $\varepsilon$ and $\mathcal{S}$.
\end{lemma}

The term $C(\mathcal{L})$ grows at most logarithmically in $\varepsilon^{-1}$, as the following results demonstrates. The proof combines martingale approximation with $h$-transform techniques and is considerably more involved than the corresponding result in \cite{oko23}, which is based on the linearity of the log-gradient of Gaussian transition densities.

\begin{proposition}[bound on $C(\mathcal L)$]\label{prop:bound_CS}
    Let $L_s$ be given by \eqref{Ls} for $s \in \mathcal{S}$ as in \eqref{eq:approx_class}. 
    Then, it holds that
    \begin{equation*}
        C(\mathcal{L}) = \sup_{s\in\mathcal S}\lVert L_s\rVert_{L^\infty(\partial B)}\lesssim \log\varepsilon^{-1} + 1.
    \end{equation*}
\end{proposition}

The proofs of both lemmata and the previous proposition are deferred to Appendix \ref{app:general}.

\subsection{Approximation error}\label{sec:approx}
The above results allow us to bound the complexity term fully in terms of $\varepsilon$ and the covering number of the neural network class $\mathcal{S}$. To control the latter, we need an efficient neural network approximator for the score given the target rate $n^{-\alpha/(2\alpha + d-1)}$ of the approximation error
\begin{equation}
    \inf_{s\in\mathcal S} \E^{z_0}\Big[\int_0^{\tau^h_{1-\varepsilon}} \lVert s(Z_t^h)-\nabla\log h(Z_t^h) \rVert^2\diff t\Big] = \inf_{s\in\mathcal S} \E^{z_0}\Big[\int_0^{\tau_{1-\varepsilon}} \lVert s(W_t)-\nabla\log h(W_t) \rVert^2 h(W_t)\diff t\Big]. \label{approximation_error}
\end{equation}
The main result of this section is the following.

\begin{theorem}\label{thm:NN_approximation}
    For any $N\in\N\cap [3,\infty)$ and $\varepsilon > 0$ small enough, there exists a neural network $s\in\mathrm{NN}(\mathsf{L,W,S,B})$ with
    \begin{align*}
        &\mathsf L\lesssim \log N\log\log N+\log^2\varepsilon^{-1}, &&\mathsf W\lesssim N\log^2 N, \\
        &\mathsf S\lesssim N\log^3 N + \log^2\varepsilon^{-1}, &&\mathsf B\lesssim \mathrm{Poly}(N)\lor \varepsilon^{-4},
    \end{align*}
    and
    \begin{align*}
        \E^{z_0}\Big[\int_0^{\tau^h_{1-\varepsilon}} \lVert s(Z_t^h)-\nabla\log h(Z_t^h) \rVert^2\diff t\Big] \lesssim N^{-2\alpha/(d-1)}\log N\log\varepsilon^{-1}.
    \end{align*}
\end{theorem}
The proof along with the proofs of intermediate approximation results given below can be found in Appendix \ref{app:approx}. This result encodes explicitly how increasing the size of the neural networks  measured in terms of $N$ improves their ability to achieve a given approximation rate relative to the smoothness $\alpha$ and the intrinsic dimension $d-1$ of the target distribution as well as the early stopping parameter $\varepsilon$. At the same time a larger $N$ also increases the covering number $\mathcal N(\mathcal S,\lVert\cdot\rVert_{L^\infty(B_{1-\varepsilon})},\delta)$ of the class $\mathcal{S}$. As shown in \cite[Lemma~4.2.]{oko23} (up to some minor modifications),  this is of maximal order $\mathsf{LS}\log(\delta^{-1}\mathsf{LWB})$, which becomes $\mathcal O(N\,\mathrm{Poly\, log}(N)\mathrm{Poly\, log}(\varepsilon^{-1}))$ with the network parameter choice from Theorem \ref{thm:NN_approximation}. Trading off these effects based on Theorem \ref{thrm:oracle_inequality}, it will turn out that the sample size dependent choice $N=n^{(d-1)/(2\alpha+d-1)}$ is optimal.

A core difference for the proof of Theorem~\ref{thm:NN_approximation} as compared to the approximation analysis in fixed time denoising  diffusion models  is that the space-time analysis of \eqref{approximation_error} cannot be disentangled since the stopping time $\tau^h_{1-\varepsilon}$ depends on the whole path of $Z^h$. Instead we use that the path integral of stopped Brownian motion can be expressed in terms of its potential measure, that is a spatial Lebesgue integral weighted by the  \emph{Green kernel} $G_{1-\varepsilon}$ of the domain $B_{1-\varepsilon}$, cf.\ \eqref{eq:green_kernel} for its explicit formula, which measures the time spent by the stopped Brownian motion in a Borel set $A \subset B_{1-\varepsilon}$ via the occupation formula
\[\E^{z_0}\Big[\int_0^{\tau_{1-\varepsilon}} \one_A(W_t) \diff{t} \Big] = \int_{A} G_{1-\varepsilon}(z_0,z) \diff{z}.\]
We  split the spatial integral into two sub-domains motivated by the asymptotic behaviour of the Green kernel: a fixed interior ball, where the Green kernel is singular, but the integrand is uniformly bounded, and the remaining annulus, where the Green kernel $G_{1-\varepsilon}(z,z_0)$ decreases linearly to $0$ with the  distance of $z$ to the boundary, cf.\ Lemma \ref{lem:Green_function_bound}, thus limiting the average time spent by the Brownian motion close to the sphere before being absorbed. This allows to appropriately downweigh the explosive behavior of the score $\nabla \log h$ in this critical area for approximation purposes.

\begin{proposition}\label{prop:approximation_error_decomposition}
    Let $\lVert z_0\rVert<R<1 - \varepsilon$ be some fixed radius. Then,
    \begin{equation}\label{approximation_error_decomposition}
    \begin{split}
        \E^{z_0}\Big[\int_0^{\tau^h_{1-\varepsilon}} \lVert s(Z_t^h)-\nabla\log h(Z_t^h) \rVert^2\diff t\Big] &= \int_{B_{1-\varepsilon}} G_{1-\varepsilon}(z,z_0) \lVert s(z) - \nabla\log h(z) \rVert^2\, h(z) \diff z\\ 
        &\lesssim \left\lVert (s- \nabla\log h)\sqrt{G_{1-\varepsilon}(z_0,\cdot)} \right\rVert_{L^2(B_{1-\varepsilon}\setminus B_R)}^2 + \lVert s-\nabla \log h\rVert_{L^\infty(B_R)}^2,
    \end{split}
    \end{equation}
    where $G_{1-\varepsilon}$ is the Green kernel for the ball of radius $1-\varepsilon$.
\end{proposition}

This implies, that we need two different neural networks $\overline{s}$ and $\underline{s}$ for the domains $B_{1-\varepsilon}\setminus B_R$ and $B_R$ respectively and then combine them to the network $s$ on the larger ball $B_{1-\varepsilon}$ with a partition of unity. To this end, one needs some overlap between $\overline{s}$ and $\underline{s}$, so $\underline{s}$ should approximate the score on a ball with some slightly larger radius $\tilde R>R$. For notational convenience, however, we will in the following write the results in terms of $R$, which could be any number between $\lVert z_0\rVert$ and 1.

\paragraph{$L^\infty$ approximation on $B_R$}
\label{subsec:L_infty_approx}

The approximation of the score on $B_R$ is straightforward, since $\nabla \log h$ is infinitely smooth with supremum norm of the  derivatives depending only on the fixed value $R$, which we choose independently of $n$. Thus, we can invoke the general approximation theorem from \cite{Suzuki_2018} almost directly to $\nabla\log h$ componentwise.

\begin{lemma}[\cite{Suzuki_2018} Proposition~1, special case] \label{lem:Suzuki_approximation}
    Let $\gamma>d/2$ and $f\in H^\gamma([-2,2]^d)$ with $\lVert f\rVert_{H^\gamma([-2,2]^d)} \leq 1$, $N\in\N$ sufficiently large. Then there exists $\phi_N\in\mathrm{NN}(\mathsf{L,W,S,B})$ with
    \begin{align*}
        \mathsf L\lesssim \log N,&& \mathsf W\lesssim N,&& \mathsf S\lesssim N\log N, && \mathsf B=\mathrm{Poly}(N),
    \end{align*}
    such that
    \begin{equation*}
        \lVert f-\phi_N \rVert_{L^\infty([-2,2]^d)} \lesssim N^{-\gamma/d}.
    \end{equation*}
\end{lemma}

Since the ball $B_R$ has a smooth boundary, the Sobolev extension theorem allows to extend the components of $\nabla\log h$ on $B_R$ to functions $u_i\in H^\gamma(\R^d)$, $i=1,\ldots,d,$ with $\lVert u_i\rVert_{H^\gamma(\R^d)}\lesssim\lVert (\nabla\log h)_i\rVert_{H^\gamma(B_R)}$ for any $\gamma \geq 1$, which can in turn be approximated with rate $\mathcal O(N^{-\gamma/d})$ on $[-2,2]^d\supset B_R$. Choosing $\gamma\coloneqq \frac{d}{d-1}\alpha$, which is valid for $\alpha>(d-1)/2$, and parallelising the networks obtained from Lemma \ref{lem:Suzuki_approximation}, then yields the desired network. This is summarised in the following proposition.

\begin{proposition}\label{prop:approximation_BR}
    For $N\in\N$, there exists a neural network $\underline{s}\in\mathrm{NN}(\mathsf{L,W,S,B})$ with
    \begin{align*}
        \mathsf L\lesssim \log N,&& \mathsf W\lesssim N,&& \mathsf S\lesssim N\log N, && \mathsf B=\mathrm{Poly}(N),
    \end{align*}
    such that
    \begin{equation*}
        \lVert \nabla\log h-\underline{s} \rVert_{L^\infty(B_R)} \lesssim N^{-\alpha/(d-1)}.
    \end{equation*}
\end{proposition}

This proposition already bounds the second term in \eqref{approximation_error_decomposition}. The remaining approximation on the annulus $B_{1-\varepsilon}\setminus B_R$ is much more involved.

\paragraph{$L^2$ approximation on $B_{1-\varepsilon}\setminus B_R$}

The approximation of the score on $B_{1-\varepsilon}$ requires a more careful construction of a neural network, since the Sobolev norm of $\nabla\log h$ diverges as $\varepsilon$ decreases to zero. Our approximation strategy is inspired by the approach from \cite{Asbjorn_2025} for neural network approximation of space-time functions, where in our setting the radial coordinate plays the rôle of the time variable. The approximation strategy can be summarised as follows:
\begin{enumerate}
    \item We start with approximating $h$ and $\nabla h$  by functions $h_N$ and $\nabla h_N$, which are finite sums of monomials in the radial coordinate multiplied with $L^2$-basis functions on the $(d-1)$-dimensional unit sphere. 
    \item For fixed radii $r_i$ we then construct spherical neural network approximations of $h_N(r_i,\cdot)$ and $\nabla h_N(r_i,\cdot)$, exploiting the induced smoothness from our target distribution $\pi^\ast$.
    \item The neural network approximations for fixed radii are then combined via polynomial interpolation in $r$ to obtain neural network approximations of $(r,x) \mapsto h_N(r,x)$ and $(r,x) \mapsto \nabla h_N(r,x)$, which are in turn approximated by neural networks. Because the finite approximations $h_N$ and $\nabla h_N$ are \emph{entire} functions in $r$ \cite{trefethen13}, it suffices to choose a number of radii that scale logarithmically in $N$ to achieve an optimal convergence rate, while keeping a $\mathcal O(\log n)$ dependence of all network hyperparameters $\mathsf{L,W,S,B}$.
    \item At the end of this procedure, we get neural network approximations $s_{h_N\circ\varphi_\pm^{-1}}$ and $s_{\nabla h_N\circ\varphi_\pm^{-1}}$ of $h_N$ and $\nabla h_N$ in spherical coordinates, i.e., of the functions $h\circ\varphi_\pm^{-1}$ and $\nabla h\circ\varphi_\pm^{-1}$, where $\varphi_\pm$ are stereographical projection maps. In a final step, we  concatenate these with neural network approximations $s_{\varphi_\pm}$ of $\varphi_\pm$ to get approximations of $h_N$ and $\nabla h_N$ in cartesian space.
    \item In a final step, the score $\nabla\log h=\nabla h/h\approx \nabla h_N/h_N$ is approximated by a neural network approximation of the quotient of the neural network approximations of $h_N$ and $\nabla h_N$.
\end{enumerate}

The function $h$ and its spherical representation $\tilde{h}\colon (0,1) \times \partial B \to \R$ of $h$ are uniquely related by $h(z) = \tilde{h}(\lVert z \rVert, z/\lVert z \rVert)$, $z \in B\setminus \{0\}$. To not overburden notation, we will not distinguish between both representations notationally, that is, we identify $h = \tilde{h}$. The same convention will be used for $\nabla h$ and for the approximations $h_N$ and $\nabla h_N$. In order to approximate $h$ by a polynomial in the radial coordinate, we note that the Poisson kernel can be represented in terms of the spherical harmonics $Y_{lm}$ by
\begin{equation*}
    q(y\mid z) = \sum_{l=0}^\infty \lVert z \rVert^l \sum_{m=1}^{M_l} Y_{lm}(y)Y_{lm}(z/\lVert z\rVert), \quad y\in\partial B, z\in B\setminus\{0\}.
\end{equation*}
A straightforward calculation shows that the right hand side is indeed harmonic and that for any $f\in L^2(\partial B)$, $\langle q(\cdot\mid x),f\rangle_{L^2(\partial B)}=f(x)$ for all $x\in\partial B$, since $(Y_{lm})_{lm}$ is an orthonormal basis of $L^2(\partial B)$. The latter corresponds to the well-known reproducing property of the Poisson kernel on the sphere \cite{Axler1992HarmonicFT}.
As a consequence, $h$ can be expressed as
\begin{equation*}
    h(r,x)=\sum_{l=0}^\infty r^l \sum_{m=1}^{M_l} Y_{lm}(x) \underbrace{\int_{\partial B} Y_{lm}(y)\pi^\ast(y)\,\sigma(\diff{y})}_{\eqcolon a_{lm}}.
\end{equation*}
A natural choice for an approximation $h_N$ is therefore obtained by truncation of the representing series, that is,
\begin{equation*}
    h_N(r,x)\coloneqq \sum_{l=0}^N \sum_{m=1}^{M_l} a_{lm}r^l Y_{lm}(x),
\end{equation*}
with $N\in\N$ as in Section~\ref{subsec:L_infty_approx}. The gradient of $h_N$ is then given by
\begin{equation}\label{grad_h_N_expansion}
    \nabla h_N(r,x) = \sum_{l=1}^N \sum_{m=1}^{M_l} a_{lm} \big(lr^{l-1}Y_{lm}(x)x + r^{l-1} \nabla_{\partial B}Y_{lm}(x)\big),
\end{equation}
where $\nabla_{\partial B}$ denotes the covariant derivative on the sphere. The approximation rates of $h_N$ and $\nabla h_N$ are given in the following lemma.

\begin{lemma}\label{lem:approx_score_trunc}
    It holds
    \begin{enumerate}[label = (\roman*), ref = (\roman*)]
        \item \label{approx_1} $\lVert \nabla h_N(h - h_N)\sqrt{G_{1-\varepsilon}(z_0,\cdot)} \rVert_{L^2(B_{1-\varepsilon}\setminus B_R)}\lesssim \log\varepsilon^{-1} (1-\varepsilon)^{2N+d+1}N^{-2\alpha} \lVert \pi^\ast\rVert_{H^\alpha(\partial B)}^2,$
        \item \label{approx_2} $\lVert \nabla h - \nabla h_N \rVert_{L^2(B_{1-\varepsilon})}\lesssim (1-\varepsilon)^{2N+d}N^{-2\alpha+1} \lVert \pi^\ast\rVert_{H^\alpha(\partial B)}.$
    \end{enumerate}
\end{lemma}
For $d\geq 3$, the rates in $N$ are at least as fast as those resulting from the approximation of the score on $B_R$ shown in Proposition~\ref{prop:approximation_BR}. What remains is the approximation of the simpler functions $h_N$ and $\nabla h_N$ by appropriate neural networks, that is, neural networks achieving an $L^2$-approximation rate of order $\mathcal O(N^{-\alpha/(d-1)})$ up to logarithmic factors in $N$ and $\varepsilon^{-1}$ while requiring at most $\mathrm{Poly\, log}(N)$ network layers and $N\,\mathrm{Poly\, log}(N)$ non-zero parameters in order to control the complexity of the neural network class in \eqref{oracle_inequality}. Following the steps 2--4 described above provides networks $s_{h_N}$ and $s_{\nabla h_N}$ with precisely the desired properties.

\begin{proposition}\label{prop:approx_nn_trunc}
    There exist neural networks $s_{h_N}$ and $s_{\nabla h_N}$ in the neural network class $\mathcal S(\mathsf{L,W,S,B})$ with
    \begin{align*}
        &\mathsf L\lesssim \log N\log\log N+\log^2\varepsilon^{-1}, &&\mathsf W\lesssim N\log^2 N, \\
        &\mathsf S\lesssim N\log^3 N + \log^2\varepsilon^{-1}, &&\mathsf B\lesssim N^{1/(d-1)}\lor \varepsilon^{-4},
    \end{align*}
    which achieve
    \begin{align*}
        &\lVert s_{\nabla h_N}(h_N-s_{h_N})\sqrt{G_{1-\varepsilon}(z_0,\cdot)}\rVert_{L^2(B_{1-\varepsilon}\setminus B_R)}\lesssim N^{-\alpha/(d-1)}\log N\log\varepsilon^{-1}, \\
        &\lVert (\nabla h_N-s_{\nabla h_N})\sqrt{G_{1-\varepsilon}(z_0,\cdot)}\rVert_{L^2(B_{1-\varepsilon}\setminus B_R)}\lesssim N^{-\alpha/(d-1)}\log N\log\varepsilon^{-1}.
    \end{align*}
\end{proposition}

Combining the two previous results with Proposition~\ref{prop:approximation_BR} gives us the main ingredients for achieving the overall approximation rate of the score in Theorem~\ref{thm:NN_approximation}. Detailed derivations are given in Appendix~\ref{app:approx}.

\subsection{Proof of the main statement}

\begin{proof}[Proof of Theorem \ref{theo:main}]
Decomposing the risk as in \eqref{eq:TV_decomposition} we get
\begin{align}\label{TV_bound}
    \E^{z_0}[\mathrm{TV}(\Pi^\ast,P_{\partial B} \sharp \PP^{\hat s}_{\partial B_{1-\varepsilon}})] &\leq \mathrm{TV}(\Pi^\ast, P_{\partial B} \sharp \Q^h_{\partial B_{1-\varepsilon}}) + \E[\mathrm{TV}(\Q^h_{\partial B_{1-\varepsilon}},P^{\hat s}_{\partial B_{1-\varepsilon}})] \nonumber\\ 
    &\lesssim \varepsilon^{\beta/2} +  \E[\mathrm{TV}(\Q^h_{\partial B_{1-\varepsilon}},P^{\hat s}_{\partial B_{1-\varepsilon}})] \nonumber\\
    &\lesssim \varepsilon^{\beta/2} +  \Bigg(\inf_{s\in\mathcal S} \E^{z_0}\Big[\int_{\underline{\tau}^h}^{\overline{\tau}^h} \lVert s(Z_t^h)-\nabla\log h(Z_t^h) \rVert^2\diff t\Big] \nonumber\\
    &+ C(\mathcal L) \frac{\log\mathcal N(\mathcal L,\lVert \cdot\rVert_{L^\infty(\partial B)},\delta)}{n} + \delta\Bigg)^{1/2},
\end{align}
where we used Proposition \ref{prop:early_stopping} for the second line and afterwards used Pinsker's inequality combined with Proposition~\ref{prop:bound_TV_Girsanov} and  Theorem~\ref{thrm:oracle_inequality}. This shows that we must choose $\varepsilon=n^{-2\alpha/(\beta(2\alpha+d-1))}$ and $\delta\coloneqq n^{-2\alpha/(2\alpha+d-1)}$ for the desired convergence rate. For the complexity term, we  invoke Proposition~\ref{prop:bound_CS} to bound $C(\mathcal L)$ by $\log\varepsilon^{-1}$ (up to constants) as well as Lemma~\ref{lem:bound_covering_number} to bound the covering number $\mathcal N(\mathcal L,\lVert \cdot\rVert_{L^\infty(\partial B)},\delta)$ by the covering number $\mathcal N(\mathcal S,\lVert \cdot\rVert_{L^\infty(B_{1-\varepsilon})},\delta/\sqrt{C(\mathcal L)})$, where we choose $\mathcal{S}$ such that it contains the approximation network from Theorem \ref{thm:NN_approximation}, thus yielding the order of the network parameters from the statement of the theorem. Since $\mathcal{S}$ is a network class with depth $\mathsf L$, width $\mathsf W$, sparsity $\mathsf S$ and maximal parameter value  $\mathsf B$, its covering number can be bounded with \autocite[Lemma~C.2]{oko23}. For the given order of the network parameters, this yields
\begin{align}
    &\log \mathcal N(\mathcal L,\lVert \cdot\rVert_{L^\infty(\partial B)},\delta)\nonumber\\ 
    &\,\lesssim \log \mathcal N(\mathcal S,\lVert \cdot\rVert_{L^\infty(B_{1-\varepsilon})},c\delta/\sqrt{C(\mathcal L)}) \nonumber\\
    &\,\lesssim \mathsf {LS}\log(c^{-1}\delta^{-1}\sqrt{C(\mathcal L)}\mathsf{LWB}) \nonumber\\
    \begin{split}\label{final_complexity_bound}
        &\,\lesssim (\log N\log\log N+\log^2\varepsilon^{-1})(N\log^3 N+\log\varepsilon^{-1}) \\
        &\qquad\times\log\big(c^{-1}\delta^{-1}\log^{1/2}\varepsilon^{-1} \log\varepsilon^{-1}\, (\log N\log\log N+\log^2\varepsilon^{-1})N\log^2 N\times (\mathrm{Poly}(N)\lor \varepsilon^{-4})\big).
    \end{split}
\end{align}
By choosing $N= n^{(d-1)/(2\alpha + d-1)}$ and $\varepsilon$, $\delta$ as above, the bound \eqref{final_complexity_bound} reduces to $n^{(d-1)/(2\alpha + d-1)}\log^6 n$. Combining this with the approximation result Theorem~\ref{thm:NN_approximation}, we find that \eqref{TV_bound} is bounded by
\begin{align*}
    \E[\mathrm{TV}(\Pi^\ast,P_{\partial B} \sharp \PP^{\hat s}_{\partial B_{1-\varepsilon}})] &\lesssim \varepsilon^{\beta/2} + \Bigg(N^{-\frac{2\alpha}{d-1}}\log N\log\varepsilon^{-1} + \log\varepsilon^{-1} \frac{n^{\frac{d-1}{2\alpha + d-1}}\log^6 n}{n} + \delta \Bigg)^{1/2} \\
    &\lesssim n^{-\frac{\alpha}{2\alpha+d-1}} + \Bigg(n^{-\frac{2\alpha}{2\alpha+d-1}}(\log n)^2 + n^{-\frac{2\alpha}{2\alpha+d-1}}(\log n)^7 + n^{-\frac{2\alpha}{2\alpha+d-1}} \Bigg)^{1/2}\\
    &\lesssim n^{-\frac{\alpha}{2\alpha+d-1}} + n^{-\frac{\alpha}{2\alpha+d-1}}\log n + n^{-\frac{\alpha}{2\alpha+d-1}}(\log n)^{7/2}  \\
    &\lesssim n^{-\frac{\alpha}{2\alpha+d-1}}(\log n)^{7/2}.
\end{align*}
\end{proof}

\section{Conclusion}\label{sec:discussion}

In this paper, we gave the first statistical optimality results for iterative generative models with random termination criterion by showing that FHDMs achieve the minimax optimal convergence rate $\mathcal O(n^{-\alpha/(2\alpha+d-1)})$ up to logarithmic factors for spherical data with $\alpha$-Sobolev smooth and uniformly lower bounded density $\pi^\ast$, where $\alpha \in ((d-1)/2,\infty) \cap \N$. Given the stochastic termination criterion and the time-homogeneous nature of the algorithm, both the probabilistic as well as the score approximation analysis imposed several new challenges compared to the analysis of denoising diffusion models with deterministic time horizon, which we solved based on  arguments rooted in general Markov processes theory and spherical Riemannian geometry. 

To finish the paper, let us comment on limitations that arise from our setting and the made assumptions. The restriction on $\alpha$ implies that the density must be Hölder smooth and it is a technically challenging question how to extend this to a non-smooth setting, where we have no access to Sobolev embeddings, similarly to \cite{oko23} and their general Besov smoothness framework in the context of DDMs. The lower boundedness assumption on $\pi^\ast$, albeit a common assumption in the field of statistics for generative models, also imposes a notable restriction on the generality of the statistical setting and is highly non-trivial to overcome. 
A further natural question that arises from this paper is how the convergence analysis of FHDMs may be extended to more general target manifolds $\mathcal{M}$ than the sphere. The general score matching results from section \ref{sec:main} and the general risk decomposition extend without any problems, where a suitable early stopping criterion becomes the first hitting time of an $\varepsilon$-environment of $\mathcal{M}$. Moreover,  \cite{Krantz2005} shows that the Poisson kernel for more general smooth domains is asymptotically comparable to the Poisson kernel for the sphere  near the boundary, which allows to extend some technical arguments in this paper that are based on the explicit form of the Poisson kernel. A direct extension of our score approximation strategy, however, critically depends on the geometry and smoothness of $\mathcal{M}$, with the availability  of a suitable expansion of the Poisson kernel being a central feature of our analysis. It should however also be noted that FHDMs are specifically designed for data supports admitting an easily implementable Poisson kernel such that highly complicated or even unknown domains require a different ansatz for random adaptive termination rules, see e.g.\ \cite{christensen26}. Finally, we emphasise that we haven't considered discrete sampling effects for both the training as well as the generation procedure as discussed in Section \ref{subsec:generation_and_estimation}, which given the stochastic termination criterion is an interesting question from a numerical point of view for future work.

\appendix

\section{Technical results on hitting times and $h$-transforms}

\begin{lemma}\label{lem:diffusion_exit_time_finite_moments}
    Let $Z$ be the strong solution of
    \begin{equation*}
        \d Z_t = b(Z_t)\d t + \sigma(Z_t)\d W_t, \quad Z_0\sim \eta,
    \end{equation*}
    where $b\colon \R^d\to\R^d$ and $\sigma\colon \R^d\to\R^{d\times n}$, $d\leq n$, are Lipschitz functions, $\sigma\sigma^\top$ is uniformly positive definite, $\nu$ is concentrated on $B$ and $W$ is a $n$-dimensional Brownian motion indpendent of $Z_0$. Denote by $\tau$ the first exit time of $Z$ out of $B$. Then
    \begin{equation}
        \E[\tau^k]<\infty\quad \text{for all }k\in\N.
    \end{equation}
\end{lemma}

\begin{proof}
    According to \cite{Kinateder1998}, the function
    \begin{equation*}
        u_k(x)=\E^x[\tau^k\mid Z_0=x], \quad x\in B
    \end{equation*}
    is the solution of the iterated PDE with $u_0=1$:
    \begin{equation*}
        \begin{cases}
            Au_k=-ku_{k-1} & \text{on }B,\\
            u_k=0 & \text{on }\partial B,
        \end{cases}
    \end{equation*}
    where $A$ is the infinitesimal generator of $Z$. We show that solutions in $C^2(B)$ exist for all $k\in\N$ by induction. Since both $u_0$ and the diffusion coefficients are in $C^\gamma(B)$ for any $\gamma<1$, the solution $u_1\in C^{2+\gamma}(\overline{B})$ exists by \autocite[Corollary~6.9.]{Gilbarg_Trudinger2001}. Thus, for fixed $k\in\N$, $u_{k-1}\in C^\gamma(\overline{B})$ and so $u_k\in C^{2+\gamma}(\overline{B})$ exists. Using that $\overline{B}$ is compact and $u_k$ is continuous, we get $\sup_{x\in\overline{B}}\E[\tau^k\mid Z_0=x]<\infty$ and finally
    \begin{equation*}
        \E[\tau^k]=\int_{\overline{B}}\E[\tau^k\mid Z_0=x]\, \eta(\diff{x}) <\infty.
    \end{equation*}
\end{proof}

\begin{lemma}\label{lem:expectation_norm_difference_BM}
For any $z\in B_{1-\varepsilon}$ it holds that $\E^z[\lVert W_\tau - W_{\tau_{1-\varepsilon}} \rVert]\leq \sqrt{\varepsilon(2-\varepsilon)}$.
\end{lemma}

\begin{proof}
    By Jensen's inequality and the strong Markov property  of Brownian motion \autocite[Chapter~2.3, Theorem~3]{Chung_Walsh_2005}, it follows that
    \begin{align}
        \E^z[\lVert W_\tau - W_{\tau_{1-\varepsilon}} \rVert] &\leq \left(\E^{z}[\lVert W_{\tau}-W_{\tau_{1-\varepsilon}}\rVert^2]\right)^{\frac{1}{2}} \nonumber\\
        &= \E^{z}\big[\E^{W_{\tau_{1-\varepsilon}}}[\lVert W_{\tau} - W_0\rVert^2]\big]^{1/2}\label{expectation_W_tau_minus_W0_squared}
    \end{align}
    where we used for the second line that $\tau_{1-\varepsilon} \leq \tau$ under $\PP^z$ since $z \in B_{1-\varepsilon}$. For any $x \in B$  we have $\E^x[\lVert W_0\rVert^2-d\times 0]=\E^x[\lVert W_{\tau}\rVert^2-d \tau]$ by the optional stopping theorem  and thus
    \begin{align*}
        \lVert x\rVert^2 =\E^x[\lVert W_{\tau}\rVert^2-d\tau]=1-d\, \E^x[\tau].
    \end{align*}
    Similarly, optional stopping for the mean zero martingale $\lVert W_t - W_0 \rVert^2 - dt$ under $\PP^x$ yields 
    \[\E^x[\lVert W_{\tau} - W_0 \rVert^2 - d\tau] = 0.\]
    Combining the above gives $\E^x[\lVert W_{\tau}-W_0\rVert^2] = 1 - \lVert x \rVert^2$, which inserted into \eqref{expectation_W_tau_minus_W0_squared} finally yields
    \begin{align*}
        \E^z[\lVert W_\tau - W_{\tau_{1-\varepsilon}} \rVert] &\leq \E^{z}\left[1-\lVert W_{\tau_{1-\varepsilon}}\rVert^2\right]^{1/2}=\left[1 - (1-\varepsilon)^2 \right]^{1/2}=\sqrt{\varepsilon(2-\varepsilon)}
    \end{align*}
\end{proof}

\begin{lemma}[$L^2$-bound on the score]\label{lem:L^2_bound_score}
    Let $\nu$ be some probability measure on the unit sphere, $\overline{h}$ a function of the form
    \begin{equation*}
        \overline{h}(z) = \int_{\partial B} q(x\mid z)\,\nu(\diff{x}), \quad z\in B,
    \end{equation*}
    $Z^{\overline{h}}$ be the corresponding $\overline{h}$-transform of some Brownian motion absorbed in $\partial B$, and let $\F$-stopping times $\underline\tau\leq \overline\tau\leq \tau^{\overline{h}}_{1-\varepsilon}$ be given. Then, for any $z_0 \in B_{1-\varepsilon}$, it holds
    \begin{align*}
        \E^{z_0}\left[\int_{\underline{\tau}}^{\overline\tau} \lVert \nabla\log \overline{h}(Z_t^{\overline{h}}) \rVert^2 \diff t \right] = 2\E^{z_0}[\log \overline{h}(Z_{\overline{\tau}}^{\overline{h}})] - 2\E^{z_0}[\log \overline{h}(Z_{\underline{\tau}}^{\overline{h}})].
    \end{align*}
    In particular,
    \begin{equation*}
        \E^{z_0}\left[\int_0^{\tau_{1-\varepsilon}^{\overline h}} \lVert \nabla\log \overline{h}(Z_t^{\overline{h}}) \rVert^2 \diff t \right] \lesssim d\log\varepsilon^{-1} +1.
    \end{equation*}
\end{lemma}

\begin{proof}
    First, we note that
    \begin{align*}
        \lVert \nabla\log \overline{h}(z)\rVert^2 = \frac{\Delta \overline{h}(z)}{\overline{h}(z)} - \Delta\log \overline{h}(z), \quad z \in B.
    \end{align*}
    Since $q(x\mid\cdot)$ is bounded and smooth on $\overline{B}_{1-\varepsilon}$ for all $x\in\partial B$, derivatives can be pulled inside the integral and thus $\overline{h}$ is smooth on $\overline{B}_{1-\varepsilon}$ as well. This yields
    \begin{equation*}
        \Delta\overline{h} = \int_{\partial B}\Delta q(x\mid\cdot)\,\nu(\diff{x}) = 0,
    \end{equation*}
    since the Poisson kernel is harmonic in $B$. Thus, we get
    \begin{align*}
        \E^{z_0}\left[\int_{\underline{\tau}}^{\overline\tau} \lVert \nabla\log \overline{h}(Z_t^{\overline{h}}) \rVert^2 \diff t \right] &= -\E^{z_0}\left[\int_{\underline{\tau}}^{\overline\tau} \Delta\log h(Z_t^{\overline{h}}) \diff t \right].
    \end{align*}
    By the strong mean value property of harmonic functions, $\overline{h}>0$ in $\overline{B}_{1-\varepsilon}$ (otherwise $\nu$ would need to be the zero measure), so $\log \overline{h}$ is bounded and smooth in $\overline{B}_{1-\varepsilon}$. Thus, we can use It\^{o}'s formula for the continuous semimartingale $(Z^{\overline{h}}_{t \wedge \tau^{\overline{h}}_{1-\varepsilon}})_{t\geq0}$ to obtain 
    \begin{align*}
        &\E^{z_0}\bigg[\int_{\underline{\tau}}^{\overline\tau} \Delta\log \overline{h}(Z_t^{\overline{h}}) \diff t \bigg]\\ 
        &\,= 2\E^{z_0}\bigg[\log \overline{h}(Z_{\overline{\tau}}^{\overline{h}}) - \log \overline{h}(Z_{\underline{\tau}}^{\overline{h}}) - \int_{\underline{\tau}}^{\overline\tau} \nabla \log \overline{h}(Z_{t}^{\overline{h}})\cdot \diff Z_{t}^{\overline{h}} \bigg] \\
        &\,= 2\bigg(\E^{z_0}[\log \overline{h}(Z_{\overline{\tau}})] - \E^{z_0}[\log \overline{h}(Z_{\underline{\tau}}^{\overline{h}})] -\E^{z_0}\Big[\int_{\underline{\tau}}^{\overline\tau} \lVert\nabla\log \overline{h}(Z_{t}^{\overline{h}}) \rVert^2\diff t \Big] - \E^{z_0}\Big[\int_{\underline{\tau}}^{\overline\tau} \nabla\log \overline{h}(Z_t^{\overline{h}}) \cdot \diff W_t\Big]\bigg)\\ 
        &= 2\bigg(\E^{z_0}[\log \overline{h}(Z_{\overline{\tau}}^{\overline{h}})] - \E^{z_0}[\log \overline{h}(Z_{\underline{\tau}}^{\overline{h}})] -\E^{z_0}\Big[\int_{\underline{\tau}}^{\overline\tau} \lVert\nabla\log \overline{h}(Z_{t}^{\overline{h}}) \rVert^2\diff t \Big] \bigg)
    \end{align*}
    where the last line follows from optional stopping\footnote{optional stopping is applicable here since $h$ and $\nabla \log h$ are bounded on $\overline{B}_{1-\varepsilon}$ and hence $\E^{z_0}[\langle M \rangle_{\overline{\tau}}] \lesssim \E^{z_0}[\overline{\tau}] \leq \E[\tau_{1-\varepsilon}^{\overline{h}}] < \infty$ by Lemma~\ref{lem:diffusion_exit_time_finite_moments}.} for the martingale $M = \int_0^\cdot \nabla \log \overline{h}(Z^{\overline{h}}_{s \wedge \tau^{\overline{h}}_{1-\varepsilon}}) \cdot \diff{W_s}$ and the fact that on $\{t \leq \overline{\tau}\} \subset \{t \leq \tau^{\overline{h}}_{1-\varepsilon}\}$ it holds that $M_t = \int_0^t \nabla \log\overline{h}(Z^{\overline{h}}_s) \cdot \diff{W_s}$. 
    Finally, we can solve for the term of interest and get
    \begin{equation*}
        \E\left[\int_{\underline{\tau}}^{\overline\tau} \lVert \nabla\log \overline{h}(Z_t^{\overline{h}}) \rVert^2 \diff t \right] =  2 \E[\log \overline{h}(Z_{\overline{\tau}}^{\overline{h}})] - 2\E[\log \overline{h}(Z_{\underline{\tau}}^{\overline{h}})] .
    \end{equation*}
    For the special case $\overline{\tau}=\tau^{\overline{h}}_{1-\varepsilon}$ and $\underline{\tau}=0$, it follows
    \begin{align*}
        \E^{z_0}\left[\int_0^{\tau_{1-\varepsilon}^{\overline h}} \lVert \nabla\log \overline{h}(Z_t^{\overline{h}}) \rVert^2 \diff t \right] &=  2 \E^{z_0}[\log \overline{h}(Z_{\tau^{\overline{h}}_{1-\varepsilon}}^{\overline{h}})] - 2\E^{z_0}[\log \overline{h}(Z_0^{\overline{h}})] \\
        &\lesssim 2d\log\varepsilon^{-1} + 2\log\overline{h}(z_0) .
    \end{align*}
    where we used 
    \begin{align*}
        \overline{h}(z) = \int_{\partial B} q(x\mid z) \, \nu(\diff{x}) = \frac{1-\lVert z \rVert^2}{\sigma(\partial B)} \int_{\partial B} \frac{1}{\lVert x - z \rVert^d} \, \nu(\diff{x}) \lesssim \int_{\partial B} \frac{1}{(1 - \lVert z \rVert)^d} \nu(\diff{x}) = (1 - \lVert z \rVert)^{-d} = \varepsilon^{-d}
    \end{align*}
    for $z\in\partial B_{1-\varepsilon}$.
\end{proof}

\begin{lemma}\label{lem:expectation_tau_x}
    Let $\tau^x_{1-\varepsilon}=\tau_{1-\varepsilon}(Z^x)$ for $x\in\partial B$ and $\tau$ be the first hitting of the unit sphere by a Brownian motion. Let $z_0 \in B_{1-\varepsilon}$ and $\varepsilon < (1 - \lVert z_0 \rVert/2)$. Then, for any $x \in \partial B$
    \begin{equation*}
        \E^{z_0}[\tau^x_{1-\varepsilon}]\leq C(\E^0[\tau] + \E^{z_0}[\tau]) < \infty,
    \end{equation*}
    for a constant $C$ depending on $z_0$ but not on $\varepsilon$ and $x$.
\end{lemma}

\begin{proof}
    Let $R = 5\lVert z_0\rVert/4$ be a fixed radius such that $z_0 \in B_R$ and $B_R \cap \partial B_{1-\varepsilon} = \varnothing$ by assumption. Since $Z^x$ is a $q(x \mid \cdot)$-transform of the stopped Brownian motion, we have
    \begin{align}
        \E^{z_0}[\tau^x_{1-\varepsilon}] &= \frac{1}{q(x \mid z_0)}\E^{z_0}\Big[\int_0^{\tau_{1-\varepsilon}} q(x \mid W_t) \diff{t} \Big] \nonumber\\ 
        &= \frac{1}{q(x \mid z_0)}\int_{B_{1-\varepsilon}} G_{1-\varepsilon}(z_0,z) q(x \mid z) \diff{z} \nonumber\\ 
        &\lesssim \frac{1}{q(x \mid z_0)} \int_{B_R} G_{1-\varepsilon}(z,z_0) \diff{z} + \frac{1}{q(x \mid z_0)}\int_{B_{1-\varepsilon} \setminus B_R} G_{1-\varepsilon}(z,z_0) q(x \mid z) \diff{z} \nonumber\\
        &\lesssim \frac{1}{q(x \mid z_0)} \int_{B_{1-\varepsilon}} G_{1-\varepsilon}(z,z_0) \diff{z} + \frac{1}{q(x \mid z_0)}\int_{B_{1-\varepsilon} \setminus B_R} G_{1-\varepsilon}(z,z_0)q(x \mid z) \diff{z}. \label{tau_x_decomposition}
    \end{align}
    For the first term we obtain by definition of the Green kernel $G_{1-\varepsilon}$ of the Brownian motion killed in $\partial B_{1-\varepsilon}$,
    \begin{equation*}
        \frac{1}{q(x \mid z_0)} \int_{B_{1-\varepsilon}} G_{1-\varepsilon}(z,z_0) \diff{z} = \frac{1}{q(x \mid z_0)}\E^{z_0}[\tau_{1-\varepsilon}] = \frac{(1-\varepsilon)^2-\lVert z_0\rVert^2}{d\,q(x \mid z_0)}\leq \frac{(1+\lVert z_0\rVert)^d}{d(1-\lVert z_0\rVert^2)} < \infty.
    \end{equation*}
    The second term in \eqref{tau_x_decomposition} can be related to the expected hitting time of a Brownian motion started in the origin with a change of measure from $G_{1-\varepsilon}(\cdot,z_0)$ to $G_{1-\varepsilon}(\cdot,0)$. This is possible, due to the 3G-Theorem \cite[Theorem 6.5]{chung95}, which gives
    \begin{equation*}
        \frac{G_1(x,y)G_1(y,z)}{G_1(x,z)} \lesssim \frac{\lVert x - z \rVert^{d-2}}{\lVert x - y \rVert^{d-2} \,\lVert y - z \rVert^{d-2}}.
    \end{equation*}
    Together with the fact that $G_1(z/(1-\varepsilon),z_0/(1-\varepsilon)) = (1-\varepsilon)^{d-2}G_{1-\varepsilon}(z,z_0)$, it implies
    \begin{equation*}
        \frac{G_{1-\varepsilon}(z, z_0)}{G_{1-\varepsilon}(z,0)} = \frac{G_{1}(z/(1-\varepsilon), z_0/(1-\varepsilon))}{G_1(z/(1-\varepsilon),0)} \lesssim \frac{1}{G_1(z_0/(1-\varepsilon),0)}\, \frac{(1-\varepsilon)^{d-2}\lVert z \rVert^{d-2}}{\lVert z - z_0 \rVert^{d-2} \, \lVert z_0 \rVert^{d-2} }
    \end{equation*}
    and the right hand side is uniformly bounded for $\varepsilon < (1 - \lVert z_0 \rVert)/2$ and $z \in B_{1-\varepsilon} \setminus B_R$.

    Thus,
    \begin{align*}
        \frac{1}{q(x \mid z_0)}\int_{B_{1-\varepsilon}\setminus B_R} G_{1-\varepsilon}(z,z_0) q(x \mid z) \diff{z} &\lesssim\frac{1}{q(x \mid z_0)}\int_{B_{1-\varepsilon}\setminus B_R} G_{1-\varepsilon}(z,0) q(x \mid z) \diff{z} \\
        &\leq\frac{(1+\lVert z_0\rVert)^d}{1-\lVert z_0\rVert^2}\E^0[\tau^x_{1-\varepsilon}].
    \end{align*}
    Under $\PP^0$, it holds that $W_{\tau_{1-\varepsilon}}$ and $\tau_{1-\varepsilon}$ are independent, see \cite[Chapter 4, Theorem 39.6]{rogers00}. Thus, 
    \[\E^0[\tau^x_{1-\varepsilon}] = \frac{1}{q(x \mid 0)}\E^0[\tau_{1-\varepsilon} q(x \mid W_{\tau_{1-\varepsilon}})] = \E^0[\tau_{1-\varepsilon}] \, \E^0[q(x \mid W_{\tau_{1-\varepsilon}})] = \E^0[\tau_{1-\varepsilon}] \leq \E^0[\tau] < \infty, \]
    where we used that $(q(x \mid W_{t \wedge \tau}))_{t \geq 0}$ is a martingale, because $q(x \mid \cdot)$ is harmonic for the stopped Brownian motion, and thus by optional stopping and dominated convergence
    \[\E^0[q(x \mid W_{\tau_{1-\varepsilon}})] = \lim_{n \to \infty} \E^0[q(x \mid W_{n \wedge \tau_{1-\varepsilon}})] = \lim_{n \to \infty} \E^0[q(x \mid W_0)] = q(x \mid 0)=1. \]
    Combining the previous bounds yields the claim.
\end{proof}

\begin{lemma}\label{lem:expectation_one_minus_norm_BM}
    It holds
    \begin{equation*}
        \E^{z_0}\left[\int_0^{\tau_{1-\varepsilon}}\frac{1}{(1-\lVert W_t\rVert)^2}\diff t \right]\lesssim \log \varepsilon^{-1} + 1.
    \end{equation*}
\end{lemma}

\begin{proof}
    First, we use the binomial formula $1-r=(1-r^2)/(1+r)\geq(1-r^2)/2$ to get
    \begin{equation*}
        \E^{z_0}\left[\int_0^{\tau_{1-\varepsilon}}\frac{1}{(1-\lVert W_t\rVert)^2}\diff t\right]\leq 4\E^{z_0}\left[\int_0^{\tau_{1-\varepsilon}}\frac{1}{(1-\lVert W_t\rVert^2 )^2}\diff t\right],
    \end{equation*}
    which renders the integrand differentiable on the whole domain. Then, we define $u\in C^2(B_{1-\varepsilon})$ as the solution of the Dirichlet boundary problem
    \begin{align*}
        \begin{cases}
            \frac{1}{2}\Delta u(z) = \frac{1}{(1-\lVert z\rVert^2)^2}, & z\in B_{1-\varepsilon},\\
            u(z) = 0, & z\in\partial B_{1-\varepsilon}.
        \end{cases}
    \end{align*}
    Then, by It\^{o}'s formula, it holds
    \begin{equation*}
        u(z_0)=\E^{z_0}\left[\int_0^{\tau_{1-\varepsilon}}\frac{1}{(1-\lVert W_t\rVert^2 )^2}\diff t\right].
    \end{equation*}
    Now, it is easy to see that $u$ is isotropic, i.e., there is a funcion $f\colon [0,1-\varepsilon]\to\R$ such that $u(z)=f(\lVert z\rVert)$, which means that the above PDE reduces to the ODE
    \begin{equation*}
        f''(r)+\frac{d-1}{r}f'(r)=r^{1-d}\frac{\diff}{\diff{r}}\left(r^{d-1}f'(r)\right)=-\frac{2}{(1-r)^2},\qquad r\in(0,1-\varepsilon).
    \end{equation*}
    This equation has the explicit solution
    \begin{align*}
        f(\lVert z_0\rVert)&= \int_{\lVert z_0\rVert}^{1-\varepsilon} v^{1-d} \left(\int_0^v \frac{r^{d-1}}{(1-r)^2}\diff r \right)\diff v \\
        &\leq \int_{\lVert z_0\rVert}^{1-\varepsilon} v^{1-d} v^{d-1}\left(\int_0^v \frac{1}{(1-r)^2}\diff r \right)\diff v\\
        &=\int_{\lVert z_0\rVert}^{1-\varepsilon} \Big(\frac{1}{1-v}-1\Big)\diff v\\
        &=\log\varepsilon^{-1} + c.
    \end{align*}
\end{proof}

\section{Remaining proofs for Section \ref{sec:main}}\label{app:main}

\begin{proof}[Proof of Lemma \ref{lem:terminal}]
    The family $(q_x)_{x \in \partial B} \coloneqq (q(x \mid \cdot))_{x \in \partial B}$ is a family of harmonic functions for the absorbed Brownian motion $Z$ and we may  express $h$ as 
    \[h(z) = \int_{\partial B} q_x(z)\, \nu (dx),\]
    for the measure $\nu(\diff{x}) \coloneqq \pi^\ast(x)\,  \sigma(\diff{x})$. The claim is  therefore a consequence of \cite[Proposition 11.10]{Chung_Walsh_2005} and the argument used in the proof of \cite[Theorem 13.39]{Chung_Walsh_2005}. 
\end{proof} 

\begin{proof}[Proof of Lemma \ref{lem:score}]
    We have
    \begin{align*} 
        \nabla \log h(z) = \int \nabla_z \log q(x \mid z) \frac{\pi^\ast(x) q(x \mid z)}{h(z)}\, \sigma(\diff{x}) &= \int \nabla_z \log q(x \mid z) \, \PP^z(Z^h_{\tau^h} \in \diff{x}) \\
        &= \E^z\big[ \nabla_2 \log q(Z^h_{\tau^h} \mid Z^h_0)\big],
    \end{align*}
    where we used Lemma \ref{lem:terminal} for the second equality.
\end{proof}

\begin{proof}[Proof of Proposition \ref{prop:bound_TV_Girsanov}]
    To ease notation, define $b\coloneqq s-\nabla\log h$. First, we apply Girsanov's theorem \cite[Theorem 5.22]{LeGall_2016} to obtain for any $T>0$ and $A\in\mathcal G_T$, 
    \begin{align*}
        \P^s_\varepsilon(A) = \int_A \mathcal E\left(\int_0^{T\land \tau_{1-\varepsilon}}b(\xi_t)\diff w_t \right) \diff{\Q^h_\varepsilon}.
    \end{align*}
    where $w$ is a Brownian motion under $\Q^h_\varepsilon$ such that $\diff \xi_t = s(\xi_t) \diff{t} + \diff{w_t}$ on $\llbracket 0,\tau_{1-\varepsilon}\rrbracket$, $\Q^h_\varepsilon$-a.s.
    Let us  note here that the stochastic integral is well-defined, since $b$ is bounded on $\overline{B}_{1-\varepsilon}$ and $T\land \tau_{1-\varepsilon}<\infty$ a.s.\ Girsanov's theorem can be applied since the following Novikov condition is fulfilled: 
    \begin{equation*}
        \E_{\Q^h_\varepsilon}\left[\exp\left(\frac{1}{2}\int_0^{T\land \tau_{1-\varepsilon}} \lVert b(\xi_t)\rVert^2\d t\right)\right] \leq \E_{\Q^h_\varepsilon}\left[\exp\left(\frac{1}{2}\sup_{z\in\overline{B}_{1-\varepsilon}}\lVert b(z) \rVert^2\, T\land\tau_{1-\varepsilon}\right) \right] \leq \mathrm{e}^{\sup_{z\in\overline{B}_{1-\varepsilon}} \lVert b(z)\rVert ^2 T}<\infty.
    \end{equation*}
    The process $\mathcal E(\int_0^{T\land \tau_{1-\varepsilon}}b(\xi_t)\diff{w_t} )_{T \geq 0}$ is a uniformly integrable martingale wrt $(\mathcal G_t)_{t\geq0}$ thanks to Lemma~\ref{lem:diffusion_exit_time_finite_moments} (note here that the  drift $b$ of $Z^{h}$ is Lipschitz on the compact set $\overline{B_{1-\varepsilon}}$ since $s$ is locally Lipschitz, $(x,z) \mapsto q(x \mid z)$ is smooth on $\overline{B_{1-\varepsilon}}$ and $h \geq \pi_{\min} > 0$), which implies
    \begin{equation}
        \E_{\Q^h_\varepsilon}\left[\Big\langle\int_0^{\cdot} b(\xi_t)\d w_t \Big\rangle_{\tau_{1-\varepsilon}}\right] = \E_{\Q^h_\varepsilon}\left[\int_0^{\tau_{1-\varepsilon}} \lVert b(\xi_t)\rVert^2 \d t \right] \lesssim \E_{\Q^h_\varepsilon}[\tau_{1-\varepsilon}] = \E^{z_0}[\tau_{1-\varepsilon}^h]<\infty. \label{martingale}
    \end{equation}
    Since $\tau_{1-\varepsilon} < \infty$, $\Q^h_\varepsilon$-a.s., for $\Q^h_{\varepsilon}$-a.e.\ $\omega$ it holds $\tau_{1-\varepsilon}(\omega) \wedge T = \tau_{1-\varepsilon}(\omega)$ for $T$ large enough and we get
    \begin{equation*}
        \lim_{T\to\infty} \mathcal E\left(\int_0^{T\land \tau_{1-\varepsilon}}b(\xi_t)\diff w_t \right) = \mathcal E\left(\int_0^{\tau_{1-\varepsilon}}b(\xi_t)\diff w_t \right) \quad \Q^h_\varepsilon \text{-a.s.}.
    \end{equation*}
    By Fatou's lemma, it follows for $A\in\mathcal G_t$, $t>0$,
    \begin{equation}
        \int_A \mathcal E\left(\int_0^{\tau_{1-\varepsilon}}b(\xi_t)\diff w_t \right) \diff\Q^h_\varepsilon \leq \liminf_{T\to\infty} \int_A \mathcal E\left(\int_0^{T\land \tau_{1-\varepsilon}}b(\xi_t)\diff{w_t} \right) \diff\Q^h_\varepsilon = \P^s_\varepsilon(A) . \label{abs_continuity_on_F_t}
    \end{equation}
    We extend \eqref{abs_continuity_on_F_t} to all $A\in\mathcal G_\infty$: define
    \begin{equation*}
        \mu(A)\coloneqq \P^s_\varepsilon(A)-\int_A \mathcal E\left(\int_0^{\tau_{1-\varepsilon}}b(\xi_t)\diff w_t \right) \diff\Q^h_\varepsilon \quad \text{for all }A\in\bigcup_{t\geq0}\mathcal G_t .
    \end{equation*}
    Due to \eqref{abs_continuity_on_F_t}, $\mu$ is a pre-measure on $\bigcup_{t\geq0}\mathcal G_t$, which can be extended to a measure on $\mathcal G_\infty=\sigma(\cup_{t\geq0}\mathcal G_t)$ using Carathéodory's theorem. Then, we have
    \begin{equation*}
        \int_\cdot \mathcal E\left(\int_0^{\tau_{1-\varepsilon}}b(\xi_t)\diff w_t \right) \diff\Q_\varepsilon^h + \mu = \P^s_\varepsilon
    \end{equation*}
    on $\bigcup_{t\geq0}\mathcal G_t$ by construction, and therefore on $\mathcal G_\infty$. Since $\mu$ is non-negative, we get \eqref{abs_continuity_on_F_t} for all $A\in\mathcal G_\infty$.
    
    Consequently, using also that $\mathcal{E}(\int_0^{\tau_{1-\varepsilon}} b(\xi_t) \diff{w_t}) > 0$, $\Q^h_\varepsilon$-a.s., it follows that $\P^s_\varepsilon(A)=0$ implies $\Q_\varepsilon^h(A)=0$ for all $\mathcal{G}_\infty$, and we get $\Q^h_\varepsilon\ll \P^s_\varepsilon$ on $\mathcal{G}_\infty$. With symmetric arguments, the above  can be repeated with $\Q^h_\varepsilon$ and $\P^s_\varepsilon$ interchanged, yielding $\P^s_\varepsilon \ll \Q^h_\varepsilon$ as well. Therefore, according to \autocite[Proposition~5.20]{LeGall_2016}, the process of conditional Radon--Nikodym densities is uniformly integrable and
    \begin{equation*}
        \frac{\d\P^s_\varepsilon}{\d\Q^h_\varepsilon} = \lim_{T\to\infty} \E_{\Q^h_\varepsilon}\left[\frac{\d\P^s_\varepsilon}{\d\Q^h_\varepsilon}\,\Bigg|\,\mathcal G_T\right] = \lim_{T\to\infty}\mathcal E\left(\int_0^{T\land \tau_{1-\varepsilon}}b(\xi_t)\diff w_t \right) = \mathcal E\left(\int_0^{\tau_{1-\varepsilon}}b(\xi_t)\diff w_t \right).
    \end{equation*}
    Thus, we conclude
    \begin{align*}
        \mathrm{KL}(\Q^h_\varepsilon \, \Vert\, \P^s_\varepsilon) &= \E_{\Q^h_\varepsilon}\left[\log\left(\frac{\d\Q^h_\varepsilon}{\d\P^s_\varepsilon}\right)\right]\\
        &= -\underbrace{\E^{z_0}\left[\int_0^{\tau_{1-\varepsilon}^h}b(Z_t^h)\d W_t \right]}_{=0 \text{ (u.i. martingale due to \eqref{martingale})}} + \frac{1}{2}\E^{z_0}\left[\int_0^{\tau_{1-\varepsilon}^h}||b(Z_t^h)||^2\d t \right]\\
        &= \frac{1}{2}\E^{z_0}\left[\int_0^{\tau_{1-\varepsilon}^h}\lVert b(Z_t^h)\rVert^2\d t \right].
    \end{align*}
\end{proof}

\begin{proof}[Proof of Proposition \ref{prop:denoising_score_equivalence}]
    Using the score representation from Lemma \ref{lem:score}, we obtain
    \begin{align*} 
    \E^z\Big[\int_{\underline{\tau}}^{\overline{\tau}} \langle \nabla \log h(Z^h_t),  s(Z^h_t)\rangle \diff{t}\Big] &= \E^z\Big[\int_{\underline{\tau}}^{\overline{\tau}} \big\langle \E^{Z^h_t}\big[\nabla_2 \log q(Z^h_{\tau^h} \mid Z^h_0) \big], s(Z^h_t) \big\rangle \diff{t}\Big] \\ 
    &= \E^z\Big[\int_0^\infty  \E\big[\one_{\{\overline{\tau} > t \geq \underline{\tau}\}} \big\langle \nabla_2 \log q(Z^h_{\tau^h} \mid Z_t^h), s(Z_t^h) \big\rangle \,\big\vert\, \mathcal{F}_t   \big]  \diff{t} \Big] \\ 
    &= \E^z\Big[\int_{\underline{\tau}}^{\overline{\tau}} \big\langle \nabla_2 \log q(Z^h_{\tau^h} \mid Z_t^h), s(Z_t^h) \big\rangle \diff{t} \Big].
    \end{align*}
    For the second line we used the Markov property and the facts that $\{\overline{\tau} > t \geq \underline{\tau}\} \in \mathcal{F}_t$ and that since $\tau^h$ is a terminal time we have $Z^h_{\tau^h} \circ \theta_t = Z^h_{t + \tau^h \circ \theta_t} = Z^h_{\tau^h}$ on $\{\overline{\tau} > t\} \subset \{\tau^h > t\}$ for the shift operators $(\theta_t)_{t \geq 0}$ of $Z^h$. The last line is then a consequence of the tower property of the conditional expectation and Fubini. This shows that 
    \begin{align*} 
    &\E^z\Big[ \int_{\underline{\tau}}^{\overline{\tau}} \lVert \nabla \log h(Z^h_t) - s(Z^h_t)\Vert^2 \diff{t}\Big]\\ 
    &\quad= \E^z\Big[ \int_{\underline{\tau}}^{\overline{\tau}} \lVert s(Z^h_t)\Vert^2 \diff{t}\Big] - 2\E^z\Big[\int_{\underline{\tau}}^{\overline{\tau}} \langle \nabla \log h(Z^h_t), s(Z^h_t)\rangle \diff{t}\Big] + C\\ 
    &\quad= \E^z\Big[ \int_{\underline{\tau}}^{\overline{\tau}} \lVert s(Z^h_t)\Vert^2 \diff{t}\Big] - 2\E^z\Big[\int_{\underline{\tau}}^{\overline{\tau}} \big\langle \nabla_2 \log q(Z^h_{\tau^h} \mid Z_t^h), s(Z_t^h) \big\rangle \diff{t} \Big] + C\\ 
    &\quad= \E^z\Big[ \int_{\underline{\tau}}^{\overline{\tau}} \lVert \nabla_2 \log q(Z^h_{\tau^h} \mid Z_t^h) - s(Z^h_t)\Vert^2 \diff{t}\Big] + C^\prime,
    \end{align*}
    where the constants $C,C^\prime$ are independent of $s$ and $C^\prime$ is given by
    \begin{align*} 
    C^\prime &\coloneq \E^z\Big[ \int_{\underline{\tau}}^{\overline{\tau}} \lVert \nabla \log h(Z^h_t)\Vert^2 \diff{t}\Big] - \E^z\Big[ \int_{\underline{\tau}}^{\overline{\tau}} \lVert \nabla_2 \log q(Z^h_{\tau^h} \mid Z_t^h)\Vert^2 \diff{t}\Big]\\ 
    &= -\E^z\Big[ \int_{\underline{\tau}}^{\overline{\tau}} \lVert \nabla_2 \log q(Z^h_{\tau^h} \mid Z_t^h) - \nabla \log h(Z^h_t)\Vert^2 \diff{t}\Big],
    \end{align*}
    where the second line follows from plugging in $s = \nabla \log h$ in \eqref{eq:denoising_score}. 
    Finally, we prove \eqref{eq:denoising_2}. By the Markov property of $W$ it $\PP^{z_0}$-a.s.\ holds
    \begin{align*} 
    &\E^{z_0}\Big[\lVert s(W_t) - \nabla \log q(W_\tau \mid W_t) \rVert^2 \pi^\ast(W_\tau) \mid \mathcal{F}_t \Big] \one_{\{t < \tau \}}\\ 
    &\,= \E^{W_t}\Big[\lVert s(W_0) - \nabla \log q(W_\tau \mid W_0) \rVert^2 \pi^\ast(W_\tau) \Big] \one_{\{t < \tau \}}\\ 
    &\,= \E^{z}\Big[\lVert s(z) - \nabla \log q(W_\tau \mid z) \rVert^2 \pi^\ast(W_\tau) \Big] \one_{\{t < \tau \}} \,\Big\vert_{z = W_t}\\
    &\,= \int_{\partial B} \lVert s(z) - \nabla \log q(x \mid z) \rVert^2 \pi^\ast(x) q(x \mid z) \,\sigma(\diff{x}) \, \one_{\{t < \tau\}} \, \Big\vert_{z = W_t}\\ 
    &\,= \int_{\partial B} \lVert s(W_t) - \nabla \log q(x \mid W_t) \rVert^2 \pi^\ast(x) q(x \mid W_t) \,\sigma(\diff{x}) \,\one_{\{t < \tau\}}.
    \end{align*}
    Using this together with \eqref{eq:h_path_3}, the  Markov property and Fubini's theorem it follows that
    \begin{align*}
        &\E^{z_0}\Big[\int_0^{\tau_{1-\varepsilon}^h} \lVert s(Z_t^h)-\nabla_2\log q(Z^h_{\tau^h} \mid Z_t^h) \rVert^2 \diff t\Big] \\
        &\,= \E^{z_0}\Big[\int_0^\infty \E^{z_0}\Big[\lVert  s(W_t)-\nabla\log q(W_\tau\mid W_t) \rVert^2 \pi^\ast(W_\tau)\mid \mathcal F_t\Big] \one_{\{t < \tau_{1-\varepsilon}\}}\diff t\Big]  \\
        &\,= \E^{z_0}\Big[\int_0^{\tau_{1-\varepsilon}} \int_{\partial B} \lVert  s(W_t)-\nabla\log q(x\mid W_t) \rVert^2 \pi^\ast(x) q(x\mid W_t) \,\sigma(\diff{x}) \diff t \Big]\\ 
        &\, = \int_{\partial B} \E^{z_0}\Big[\int_0^{\tau_{1-\varepsilon}} \lVert  s(W_t)-\nabla\log q(x\mid W_t) \rVert^2  q(x\mid W_t) \diff{t} \Big] \,\pi^\ast(x) \,\sigma(\diff{x})  \\ 
        &= \int_{\partial B} \E^{z_0}\Big[\int_0^{\tau_{1-\varepsilon}} \lVert  s(W_t)-\nabla\log q(x\mid W_t) \rVert^2   \diff{t}\,  q(x \mid W_{\tau_{1-\varepsilon}}) \Big] \,\pi^\ast(x) \,\sigma(\diff{x}),
    \end{align*}
    where for the last line we used that $q(x \mid \cdot)$ is bounded on $B_{1-\varepsilon}$ and harmonic for the killed Brownian motion, which implies that $(q(x \mid W_{t \wedge \tau_{1-\varepsilon}}))_{t \geq 0}$ is a bounded martingale and hence $q(x \mid W_t) \one_{\{t< \tau_{1-\varepsilon}\}} = \E^{z_0}[q(x \mid W_{\tau_{1-\varepsilon}}) \mid \mathcal{F}_t] \one_{\{t< \tau_{1-\varepsilon}\}}$, $\PP^{z_0}$-a.s.\ by the  Markov property and the optional stopping theorem.
    By \cite[Theorem 11.9]{Chung_Walsh_2005} and the fact that $\tau^x_{1-\varepsilon} < \tau^x$ a.s., the expectation on the rhs can be expressed in terms of the $q(x \mid \cdot)$-transform $Z^x$, yielding 
    \[\E^{z_0}\Big[\int_0^{\tau_{1-\varepsilon}} \lVert s(W_t)-\nabla\log q(x\mid W_t) \rVert^2   \diff{t}\,  q(x \mid W_{\tau_{1-\varepsilon}}) \Big] = \E^{z_0}\Big[\int_0^{\tau^x_{1-\varepsilon}} \lVert  s(Z_t^x)-\nabla\log q(x\mid Z^x_t) \rVert^2 \diff{t} \Big] q(x \mid z_0),\]
    and thus by the above, 
    \begin{align*}
    &\E^{z_0}\Big[\int_0^{\tau_{1-\varepsilon}^h} \lVert s(Z_t^h)-\nabla_2\log q(Z^h_{\tau^h} \mid Z_t^h) \rVert^2 \diff t\Big]\\
    &\,= \int_{\partial B} \E^{z_0}\Big[\int_0^{\tau^x_{1-\varepsilon}} \lVert  s(Z_t^x)-\nabla\log q(x\mid Z^x_t) \rVert^2  q(x\mid Z^x_t) \diff{t} \Big] \,\pi^\ast(x) q(x \mid z_0) \,\sigma(\diff{x})\\ 
    &\,= \int_{\partial B} \E^{z_0}\Big[\int_0^{\tau^x_{1-\varepsilon}} \lVert  s(Z_t^x)-\nabla\log q(x\mid Z^x_t) \rVert^2  q(x\mid Z^x_t) \diff{t} \Big] \,\Pi^\ast(\diff{x})
    \end{align*}
    as claimed.
\end{proof}

\begin{proof}[Proof of Lemma \ref{lem:bound_on_the_score}]
    According to Lemma~\ref{lem:score}, we have $\nabla\log h(z) = \E[\nabla_z\log q(Z_{\tau^h}^h\mid z)\mid Z_0^h=z]$ and for $x\in\partial B$ it holds
    \begin{align*}
        \nabla_z\log q(x\mid z) = -\frac{2z}{1-\lVert z\rVert^2} - d\frac{z-x}{\lVert z-x\rVert^2}.
    \end{align*}
    Combining both results, we get
    \begin{align*}
        \lVert \nabla\log h(z)\rVert &= \big\lVert \E[\nabla_z\log q(Z_{\tau^h}^h\mid z)\mid Z_0^h=z] \big\rVert \\
        &= \Big\lVert -\frac{2z}{1-\lVert z\rVert^2} -d\E\Big[\frac{Z_{\tau^h}^h-z}{\lVert Z_{\tau^h}^h-z\rVert^2}\,\Big\vert\, Z_0^h=z\Big] \Big\rVert \\
        &\leq 2\frac{\lVert z\rVert}{1-\lVert z\rVert^2} + d \E\Big[\frac{1}{\lVert Z_{\tau^h}^h-z\rVert}\,\Big\vert\, Z_0^h=z\Big] \\
        &\leq \frac{2}{(1-\lVert z\rVert)(1+\lVert z\rVert)} + \frac{d}{1-\lVert z\rVert} \\
        &\leq \frac{d+2}{1-\lVert z\rVert},
    \end{align*}
    where we used the reverse triangle inequality in the fourth line.
\end{proof}

\begin{proof}[Proof of Proposition \ref{prop:lower_bound}]
    The sphere can be split into an upper and a lower half $B_\pm$, which can then be separately parameterised in stereographic coordinates via bijective mappings $\varphi_{\partial B,\pm} \colon B_\pm \to B_1^{(d-1)}$, cf.\ Section \ref{app:approx} for details. Let $\Pi^\ast(\diff{x}) = \pi(x) \,\sigma(\diff{x})$ and $\hat{\mu}(\diff{x}) = \hat{\rho}(x) \, \sigma(\diff{x})$ and  consider a density $\pi$ concentrated on $B_+$. Then,
    \begin{align*}
        \lVert \Pi^\ast - \hat{\mu} \rVert_{\mathrm{TV}} &= \lVert (\Pi^\ast - \hat{\mu}) \vert_{B_+} \rVert_{\mathrm{TV}} + \lVert (\Pi^\ast - \hat{\mu}) \vert_{B_-} \rVert_{\mathrm{TV}} \\ 
        &= \lVert (\Pi^\ast - \hat{\mu}\vert_{B_+})  \circ \varphi_{\partial B,+}^{-1} \rVert_{\mathrm{TV}} + \lVert \hat{\mu} \vert_{B_-} \circ  \varphi_{\partial B,-}^{-1}\rVert_{\mathrm{TV}},
    \end{align*}
    where $\lVert \cdot \rVert_{\mathrm{TV}}$ denotes the total variation norm on $\partial B$ in the first and  on $B_1^{(d-1)}$ in the second line, and therefore
    \begin{align*}
        \lVert \Pi^\ast - \hat{\mu} \rVert_{\mathrm{TV}}  \geq \lVert (\Pi^\ast - \hat{\mu}\vert_{B_+}) \circ \varphi_{\partial B,+}^{-1} \rVert_{\mathrm{TV}} &= \int_{B_1^{(d-1)}} \big\lvert \pi \circ \varphi^{-1}_{\partial B,+} \lvert J_{\varphi^{-1}_{\partial B,+}} \rvert(x) - \hat{\rho} \circ \varphi^{-1}_{\partial B,+} \lvert J_{\varphi^{-1}_{\partial B,+}} \rvert(x) \big \rvert \diff{x} \\ 
        &\eqcolon \int_{B_1^{(d-1)}} \lvert \overline{\pi}(x) - \tilde{\rho}(x) \rvert \diff{x},
    \end{align*}
    where $\lvert J_{\varphi^{-1}_{\partial B,+}} \rvert$ denotes the determinant of the Jacobian of $\varphi^{-1}_{\partial B,+}$ and $\overline{\pi} = \pi \circ \varphi^{-1}_{\partial B,+} \lvert J_{\varphi^{-1}_{\partial B,+}} \rvert$ is a probability density on $B_1^{(d-1)}$ such that $\overline{\pi} \in \mathcal{B}(H^\alpha(B_1^{(d-1)}), C)$ for some $C > 0$, which follows from Lemma \ref{lem:bound_regular_Sobolev_norm} and the fact that $\lvert J_{\varphi_{\partial B,+}^{-1}} \rvert$ is uniformly bounded on $B_1^{(d-1)}$. Note also that above $\tilde{\rho}$ can be considered as a measurable function wrt the data $\{Y_i\}_{i=1}^{n} = \{\varphi_{\partial B,+}(X_i)\}_{i=1}^n \overset{\text{iid}}{\sim} \overline{\pi}$ since $\{X_i\}_{i=1}^n$ may be interpreted as iid $B_+$-valued random variables under $\PP_{\pi}$ for $\pi$ concentrated on $B_+$ as above. Thus, 
    \begin{align*} 
    \inf_{\hat{\rho}}\sup_{\pi\in \mathcal{B}(H^\alpha(\partial B),L)} \E_{\pi}[\mathrm{TV}(\pi,\hat{\rho})]  &\geq \inf_{\hat{\rho}}\sup_{\pi\in \mathcal{B}(H^\alpha(\partial B),L), \supp \pi \subset B_+} \E_{\pi}[\mathrm{TV}(\pi,\hat{\rho})]  \\ 
    & \geq \inf_{\tilde{\rho}} \sup_{\overline{\pi}\in \mathcal{B}(H^\alpha(B_1^{(d-1)}),C)} \E_{\overline{\pi}}\big[\lVert \overline{\pi} - \tilde{\rho} \rVert_{L^1(B_1^{(d-1)}} \big].
    \end{align*}
    where the last infimum is taken over all random functions (not necessarily Lebesgue probability densities) $\tilde{\rho}$ on $B_1^{(d-1)}$ that are measurable with respect to data $\{Y_1,\ldots,Y_n\}$ such that $Y_1,\ldots,Y_n \overset{\text{iid}}{\sim} \overline{\pi}$ under $\PP_{\overline{\pi}}$. By \cite[Theorem 4]{yang99} (see the proof of \cite[Proposition D.4]{oko23} for a verification of the imposed assumptions), the latter minimax risk is lower bounded by 
    \[\inf_{\tilde{\rho}} \sup_{\overline{\pi}\in \mathcal{B}(H^\alpha(B_1^{(d-1)}),C)} \E_{\overline{\pi}}\big[\lVert \overline{\pi} - \tilde{\rho} \rVert_{L^1(B_1^{(d-1)})} \big] \gtrsim n^{-\frac{\alpha}{2\alpha + d -1}},\]
    which finishes the proof.
\end{proof}

\section{Proof of the early stopping bound}\label{app:early}
\begin{proof}[Proof of Proposition \ref{prop:early_stopping}]
    The proof uses Scheffé's lemma to relate the total variation distance to the $L^1$-distance of the densities wrt some reference measure, which we choose as the distribution $\P^{z_0}(W_\tau\in\cdot)$. While the density of $\Pi^\ast$ is $\pi^\ast$ by definition, the density of $P_{\partial B}\sharp  \Q^h_{\partial B_{1-\varepsilon}}$ is not directly obvious. For any $x \in B \setminus \{z_0\}$ let  $q(x\mid z_0)$ be defined as in the first equality of \eqref{Poisson_kernel}. For $A\in\mathcal B$ it holds
    \begin{align*}
        P_{\partial B}\sharp  \Q^h_{\partial B_{1-\varepsilon}}(A) &= \P^{z_0}\Big(\tfrac{Z^h_{\tau_{1-\varepsilon}^h}}{1-\varepsilon}\in A\Big) \\
        &= \E^{z_0}\Big[\one_{A}\left(\frac{W_{\tau_{1-\varepsilon}}}{1-\varepsilon}\right)\, h(W_{\tau_{1-\varepsilon}})\Big] \\
        &= \frac{1}{\sigma(\partial B_{1-\varepsilon})} \int_{\partial B_{1-\varepsilon}} \one_A\left(\frac{y}{1-\varepsilon}\right) h(y) q(y\mid z_0) \,\sigma(\diff{y}) \\
        &= \frac{1}{\sigma(\partial B)(1-\varepsilon)^{d-1}} \int_{\partial B} \one_A(x) h((1-\varepsilon)x)  q((1-\varepsilon) x\mid z_0 ) (1-\varepsilon)^{d-1}\,\sigma(\diff{x}) \\
        &= \int_{\partial B} \one_A(x) h((1-\varepsilon)x)  \frac{q((1-\varepsilon) x\mid z_0)}{q(x\mid z_0)} \, \P^{z_0}(W_\tau\in\diff x),
    \end{align*}
    which shows that for $x\in\partial B$
    \begin{equation*}
        \frac{\diff P_{\partial B}\sharp  \Q^h_{\partial B_{1-\varepsilon}}}{\diff\P^{z_0}(W_\tau\in\cdot)}(x) = h ((1-\varepsilon)x) \frac{q((1-\varepsilon) x\mid z_0)}{q(x \mid z_0)}.
    \end{equation*}
    Thus, using $h(z)=\E^z[\pi^\ast(W_\tau)]$, $z\in B$ and the fact that $(W_{\tau_{1-\varepsilon}}, W_{\tau_{1-\varepsilon}}/(1-\varepsilon)) \overset{d}{=} (W_\tau(1-\varepsilon),W_\tau)$ unter $\PP^{0}$, we obtain
      \begin{align}
        \mathrm{TV}(\Pi^\ast, P_{\partial B}\sharp  \Q^h_{\partial B_{1-\varepsilon}}) &\lesssim \E^{z_0}\Bigg[\left\lvert \pi^\ast(W_\tau)-h((1-\varepsilon)W_\tau) \frac{q((1-\varepsilon) W_\tau \mid z_0)}{q(W_\tau \mid z_0)} \right\rvert \Bigg] \nonumber\\
        &= \E^{0}\Bigg[\left\lvert \pi^\ast(W_\tau)-h((1-\varepsilon)W_\tau) \frac{q((1-\varepsilon) W_\tau \mid z_0)}{q(W_\tau \mid z_0)} \right\rvert q(W_\tau \mid z_0) \Bigg] \sigma(\partial B) \nonumber\\
        &\asymp \E^{0}\Bigg[\left\lvert \pi^\ast\left(\frac{W_{\tau_{1-\varepsilon}}}{1-\varepsilon}\right) q\left(\frac{W_{\tau_{1-\varepsilon}}}{1-\varepsilon} \,\Big\vert\, z_0\right)-h(W_{\tau_{1-\varepsilon}})q(W_{\tau_{1-\varepsilon}} \mid z_0)\right\rvert \Bigg] \nonumber\\
        &= \E^{0}\Bigg[\left\lvert \pi^\ast\left(\frac{W_{\tau_{1-\varepsilon}}}{1-\varepsilon}\right) q\left(\frac{W_{\tau_{1-\varepsilon}}}{1-\varepsilon} \,\Big\vert\, z_0\right)-\E^{W_{\tau_{1-\varepsilon}}}\big[\pi^\ast(W_\tau)\big] q(W_{\tau_{1-\varepsilon}} \mid z_0)\right\rvert \Bigg] \nonumber\\
        &= \E^{0}\Bigg[\left\lvert  \E^0\Bigg[\pi^\ast\left(\frac{W_{\tau_{1-\varepsilon}}}{1-\varepsilon}\right) q\left(\frac{W_{\tau_{1-\varepsilon}}}{1-\varepsilon} \,\Big\vert\, z_0\right)-\pi^\ast(W_\tau) q(W_{\tau_{1-\varepsilon}} \mid z_0) \, \Big\vert \, \mathcal{F}_{\tau_{1-\varepsilon}} \Bigg] \right\rvert\Bigg] \nonumber\\
        &\leq \E^{0}\Bigg[\left\lvert  \pi^\ast\left(\frac{W_{\tau_{1-\varepsilon}}}{1-\varepsilon}\right) q\left(\frac{W_{\tau_{1-\varepsilon}}}{1-\varepsilon} \,\Big\vert\, z_0\right)-\pi^\ast(W_\tau) q(W_{\tau_{1-\varepsilon}} \mid z_0)\right\rvert\Bigg] \nonumber\\
        \begin{split}\label{Hoelder_TV}
            &\leq \E^{0}\left[ q(W_{\tau_{1-\varepsilon}} \mid z_0)\, \Big\lvert \pi^\ast\left(\frac{W_{\tau_{1-\varepsilon}}}{1-\varepsilon}\right)-\pi^\ast(W_\tau) \Big\rvert \right] \\
            &\quad + \E^{0}\left[\pi^\ast\left(\frac{W_{\tau_{1-\varepsilon}}}{1-\varepsilon}\right)\,  \Big\lvert q\left(\frac{W_{\tau_{1-\varepsilon}}}{1-\varepsilon} \,\Big\vert\,  z_0\right)- q(W_{\tau_{1-\varepsilon}} \mid z_0) \Big\rvert \right].
        \end{split}
    \end{align}
    To bound the first term in \eqref{Hoelder_TV}, we use that for any $x \in \partial B_{1-\varepsilon}$
    \begin{equation}\label{eq:bound_q}
        q(x \mid z_0) = \frac{1}{\sigma(\partial B)} \frac{(1-\varepsilon)^2 - \lVert z_0 \rVert^2}{(1-\varepsilon)\lVert z_0 -x \rVert^d} \leq \frac{2}{\sigma(\partial B)(1+\lVert z_0\rVert)}\frac{1 - \lVert z_0 \rVert^2}{((1 - \varepsilon) - \lVert z_0 \rVert)^d} \leq \frac{2^{d+1}}{\sigma(\partial B)(1-\lVert z_0\rVert)^{d-1}}
    \end{equation}
    since by assumption $1 - \varepsilon > (1+ \lVert z_0 \rVert)/2$. Furthermore, according to the Sobolev embedding theorem on manifolds \autocite[Theorem~2.20]{Aubin_1998}, $\pi^\ast$ is $\beta$-Hölder smooth in the sense
    \begin{align*}
        \lVert \pi^\ast(x)-\pi^\ast(y)\rVert \lesssim d_{\partial B}(x,y)^\beta, \qquad x,y\in\partial B,
    \end{align*}
    where $d_{\partial B}$ is the distance measure on the sphere
    \begin{align*}
        d_{\partial B}(x,y) \coloneqq \inf\left\{\int_0^1 \lVert \tfrac{\diff}{\diff t}\gamma(t)\rVert \diff t \,\Big\vert\, \gamma\colon [0,1]\to\partial B \text{ differentiable}, \gamma(0)=x,\gamma(1)=y \right\}.
    \end{align*}
    On the sphere it thus holds $d_{\partial B}(x,y)\leq \pi\lVert x-y\rVert$ and therefore
    \begin{align}\label{Hoelder_estimate_pi}
        \Big\lVert \pi^\ast\left(\frac{W_{\tau_{1-\varepsilon}}}{1-\varepsilon}\right)-\pi^\ast(W_\tau)\Big\rVert &\lesssim \left\lVert \frac{W_{\tau_{1-\varepsilon}}}{1-\varepsilon} -W_\tau \right\rVert^\beta \nonumber\\
        &\leq \frac{\lVert W_{\tau_{1-\varepsilon}} - W_\tau \rVert^\beta}{1-\varepsilon} + \lVert W_\tau\rVert^\beta \left(\frac{1}{1-\varepsilon}-1\right)^\beta \nonumber\\
        &\lesssim \lVert W_{\tau_{1-\varepsilon}} - W_\tau \rVert^\beta + \varepsilon^\beta.
    \end{align}
    For the second term in \eqref{Hoelder_TV}, we simply note that based on \eqref{Poisson_kernel}, extending $q(\cdot \mid z_0)$ to the annulus $\Lambda_{z_0} \coloneq \{x \in \R: \lVert x \rVert \in (1- \lVert z_0\rVert/2,1+ \lVert z_0\rVert/2)\}$ by
    \[q(x \mid z_0) = \frac{1}{\sigma(\partial B)} \frac{\lVert x\rVert^2 - \lVert z_0 \rVert^2}{\lVert x\rVert\lVert z_0 - x \rVert^d}, \quad x \in \Lambda_{z_0},\]
    gives a smooth function on $\Lambda_{z_0}$, which is thus in particular Lipschitz and hence Hölder continuous for any $\gamma \in (0,1]$. Choosing $\gamma=\beta$ therefore yields
    \begin{equation}\label{Hoelder_estimate_q}
        \Big\lvert q(W_{\tau_{1-\varepsilon}}/(1-\varepsilon) \mid z_0)- q(W_{\tau_{1-\varepsilon}} \mid z_0) \Big\rvert\lesssim \lVert W_{\tau_{1-\varepsilon}}/(1-\varepsilon) - W_{\tau_{1-\varepsilon}} \rVert^\beta=\varepsilon^\beta.
    \end{equation}
    Inserting \eqref{eq:bound_q}--\eqref{Hoelder_estimate_q} in \eqref{Hoelder_TV} yields
    \begin{align*}
        \mathrm{TV}(\Pi^\ast, P_{\partial B}\sharp  \Q^h_{\partial B_{1-\varepsilon}}) &\lesssim 4 \sigma(B)^{-1}\,\big(\E^0[\lVert W_\tau - W_{\tau_{1-\varepsilon}} \rVert^\beta]+\varepsilon^\beta\big) + \lVert \pi^\ast \rVert_{L^\infty(\partial B)} \varepsilon^\beta\\ 
        &\lesssim \E^0[\lVert W_\tau - W_{\tau_{1-\varepsilon}} \rVert]^\beta + \varepsilon^\beta \\ 
        &\lesssim (\varepsilon(2-\varepsilon))^{\beta/2} + \varepsilon^\beta \\ 
        &\lesssim \varepsilon^{\beta/2},
    \end{align*}
    where we used Jensen's inequality for the second and Lemma~\ref{lem:expectation_norm_difference_BM} for the third line.
\end{proof}

\section{Proofs for Section \ref{sec:generalise}}\label{app:general}

\begin{proof}[Proof of Theorem \ref{thrm:oracle_inequality}]
    The only difference to the proof \cite[Lemma~C.4]{oko23} lies in the first few lines due to the different definition of $L_s$. Recall that the driving Brownian motion $W$ of the generative process $Z^h$ is chosen independently of the data $X_1,\ldots,X_n$, making $Z^h$ and $\{X_1,\ldots,X_n\}$ independent.  By conditioning and using the first part of Proposition~\ref{prop:denoising_score_equivalence} we then obtain 
    \begin{align*}
        &\E^{z_0}\Big[\int_0^{\tau_{1-\varepsilon}^h} \lVert \hat s(Z_t^h)-\nabla\log h(Z_t^h) \rVert^2 \diff t\Big] \\
        &\, \E^{z_0}\Big[\int_0^{\tau_{1-\varepsilon}^h} \lVert \hat s(Z_t^h)-\nabla\log h(Z_t^h) \rVert^2 \diff t\Big] - \E^{z_0}\Big[\int_0^{\tau_{1-\varepsilon}^h} \lVert \nabla\log h(Z_t^h)-\nabla\log h(Z_t^h) \rVert^2 \diff t\Big] \\
        &\,= \E^{z_0}\Big[\int_0^{\tau_{1-\varepsilon}^h} \lVert \hat s(Z_t^h)-\nabla\log q(Z_{\tau^h}^h\mid Z_t^h) \rVert^2 \diff t\Big] + C - \E^{z_0}\Big[\int_0^{\tau_{1-\varepsilon}^h} \lVert \nabla\log h(Z_t^h)-\nabla\log q(Z_{\tau^h}^h\mid Z_t^h) \rVert^2 \diff t\Big] - C\\
        &\,= \E^{z_0}\Big[\int_0^{\tau_{1-\varepsilon}^h} \lVert \hat s(Z_t^h)-\nabla\log q(Z_{\tau^h}^h\mid Z_t^h) \rVert^2 \diff t\Big]- \E^{z_0}\Big[\int_0^{\tau_{1-\varepsilon}^h} \lVert \nabla\log h(Z_t^h)-\nabla\log q(Z_{\tau^h}^h\mid Z_t^h) \rVert^2 \diff t\Big].
    \end{align*}
    We now use \eqref{eq:denoising_2} from Proposition~\ref{prop:denoising_score_equivalence}  to find
    \begin{align*}
        &\E^{z_0}\Big[\int_0^{\tau_{1-\varepsilon}^h} \lVert \hat s(Z_t^h)-\nabla\log q(Z_{\tau^h}^h\mid Z_t^h) \rVert^2 \diff t\Big] - \E^{z_0}\Big[\int_0^{\tau_{1-\varepsilon}^h} \lVert \nabla\log h(Z_t^h)-\nabla\log q(Z_{\tau^h}^h\mid Z_t^h) \rVert^2 \diff t\Big] \\
        &= \int_{\partial B} \E^{z_0}\Big[ \int_{0}^{\tau^x_{1-\varepsilon}} \lVert \hat s(Z^x_t) - \nabla_2 \log q(x \mid Z_t^x) \rVert^2 - \int_{0}^{\tau^x_{1-\varepsilon}} \lVert \nabla\log h(Z^x_t) - \nabla_2 \log q(x \mid Z_t^x) \rVert^2 \diff{t}\Big] \, \Pi^\ast(\diff{x}) \\
        &= \E^{z_0}\left[\frac{1}{n}\sum_{i=1}^n (L_{\hat s}(X_i) - L_{\nabla\log h}(X_i)) \right],
    \end{align*}
    where the last line uses Fubini's theorem and $X_i \sim \Pi^\ast$ for all $i=1,\ldots,n$. From this point on, the proof is completely analogous to \cite[Lemma~C.4]{oko23}, where we further note that essentially repeating the calculation from \autocite[Lemma~B.1]{Asbjorn_2025} based on Proposition \ref{prop:denoising_score_equivalence} gives the Bernstein inequality $\E^{z_0}[(L_s(X_1) - L_{\nabla \log h}(X_1))^2] \leq 4 C(\mathcal{L})\E^{z_0}[L_s(X_1) - L_{\nabla \log h}(X_1)]$, which is needed to justify a critical step in \cite[Lemma~C.4]{oko23}; see \cite[Section B]{Asbjorn_2025} for details.
\end{proof}

\begin{proof}[Proof of Lemma \ref{lem:bound_covering_number}]
    Let $s_1,\ldots,s_m\in \mathcal S$ be a $\tilde\delta$-net for $\mathcal S$ wrt $\lVert\cdot\rVert_{L^\infty(B_{1-\varepsilon})}$. Then, for any $s\in\mathcal S$, choose $j\in\N$ s.t. $\lVert s-s_j\rVert_{L^\infty(B_{1-\varepsilon})}\leq \tilde\delta$.
    For better readability, we define for any $x\in\partial B$
    \begin{align*}
        b(z)\coloneqq s(z)-\nabla \log q(x\mid z), && b_j(z)\coloneqq s_j(z)-\nabla \log q(x\mid z).
    \end{align*}
    Then,
    \begin{align*}
        \lvert L_s(x) - L_{s_j}(x)\rvert &= \Big\lvert\E^{z_0}[\int_{0}^{\tau^x_{1-\varepsilon}} (\lVert b(Z_t^x)\rVert^2 - \lVert b_j(Z_t^x)\rVert^2) \diff t] \Big\rvert \\
        &\leq \E^{z_0}\Big[\int_{0}^{\tau^x_{1-\varepsilon}} \big\lvert\lVert b(Z_t^x)\rVert - \lVert b_j(Z_t^x)\rVert\big\rvert\,\big(\lVert b(Z_t^x)\rVert + \lVert b_j(Z_t^x)\rVert\big)  \diff t\Big] \\
        &\leq\E^{z_0}\Big[\int_{0}^{\tau^x_{1-\varepsilon}} \lVert b(Z_t^x) - b_j(Z_t^x)\rVert \, (\lVert b(Z_t^x)\rVert + \lVert b_j(Z_t^x)\rVert)  \diff t\Big] \\
        &= \E^{z_0}\Big[\int_{0}^{\tau^x_{1-\varepsilon}} \underbrace{\lVert s(Z_t^x) - s_j(Z_t^x)\rVert}_{\leq \tilde\delta} \big( \lVert b(Z_t^x)\rVert + \lVert b_j(Z_t^x)\rVert \rVert\big)  \diff t\Big] \\
        &\leq 2\tilde\delta\sup_{s \in \mathcal{S}} \E^{z_0}\Big[\int_{0}^{\tau^x_{1-\varepsilon}} \lVert s(Z^x_t) - \nabla\log q(x\mid Z_t^x) \rVert \diff t\Big]\\ 
        &\leq 2\tilde{\delta} \sqrt{\E^{z_0}[\tau^x_{1-\varepsilon}]} \sup_{s \in \mathcal{S}}\Big(\E^{z_0}\Big[\int_{0}^{\tau^x_{1-\varepsilon}} \lVert s(Z^x_t) -\nabla\log q(x\mid Z_t^x) \rVert^2 \diff t\Big]\Big)^{1/2}\\ 
        &\lesssim 2\tilde{\delta} \sqrt{C(\mathcal{L})},
    \end{align*}
    where we used the Cauchy--Schwarz inequality in the penultimate line and  Lemma~\ref{lem:expectation_tau_x} to  uniformly bound $\E^{z_0}[\tau^x_{1-\varepsilon}]$.
    Taking the supremum over all $x$, we get
    \begin{equation*}
         \lVert L_s - L_{s_j}\rVert_{L^\infty(\partial B)} \lesssim \tilde\delta\sqrt{C(\mathcal L}).
    \end{equation*}
    By choosing $\tilde\delta\lesssim\delta/\sqrt{C(\mathcal L})$, we see that the functions $L_{s_1},\ldots,L_{s_m}$ form a $\delta$-net for $\mathcal L$ wrt $\lVert\cdot\rVert_{L^\infty(\partial B)}$.
\end{proof}

\begin{proof}[Proof of Proposition \ref{prop:bound_CS}]
    First, we have the elementary bound
    \begin{align}
        L_s(x)\leq 2\E^{z_0}\Big[\int_{0}^{\tau^x_{1-\varepsilon}}\lVert s(Z_t^x)\rVert^2 \diff t\Big] + 2\E^{z_0}\Big[\int_{0}^{\tau^x_{1-\varepsilon}} \lVert\nabla\log q(x\mid Z_t^x) \rVert^2 \diff t\Big]. \label{Lsx}
    \end{align}
    Up to a multiplicative constant, the second term is bounded by $\log \varepsilon^{-1}$ according to Lemma~\ref{lem:L^2_bound_score} with $\overline{h}=q(x\mid\cdot)$. For the first term, we note that
    \begin{equation}
        \lVert s(z) \rVert\lesssim (1-\lVert z\rVert)^{-1}\leq\frac{2\lVert z\rVert}{1-\lVert z\rVert^2} +1 = \lVert\nabla\log f(z)\rVert+1, \label{s_rewritten}
    \end{equation}
    with $f(z)\coloneqq 1-\lVert z\rVert^2$ for $z\in B$, because of
    \begin{equation*}
        (1-\lVert z\rVert)\left(\frac{2\lVert z\rVert}{1-\lVert z\rVert^2} +1 \right) =\frac{2\lVert z\rVert}{1+\lVert z\rVert} +1-\lVert z\rVert\geq 2\lVert z\rVert+1-\lVert z\rVert \geq1.
    \end{equation*}
    The advantage of this rewriting is that $f$ is smooth on the open ball $B$ (as compared to $(1 - \lVert z \rVert)^{-1}$) allowing the use of Iô's formula and one can get rid of the squared norm on the right-hand side of \eqref{Lsx} thanks to $\lVert\nabla\log f\rVert^2=\Delta f/f-\Delta\log f$, $f\geq 0$, and
    \begin{equation*}
        \Delta f(z) = -2 <0.
    \end{equation*}
    Therefore, we have
    \begin{align}
    \E^{z_0}\Big[\int_{0}^{\tau^x_{1-\varepsilon}}\lVert s(Z_t^x)\rVert^2 \diff t\Big] &\lesssim \E^{z_0}\Big[\int_{0}^{\tau^x_{1-\varepsilon}}\lVert \nabla\log f(Z_t^x)\rVert^2 \diff t\Big] + \E^{z_0}[\tau^x_{1-\varepsilon}] \nonumber\\
        &\leq\E^{z_0}\Big[\int_{0}^{\tau^x_{1-\varepsilon}} -\Delta\log f(Z_t^x) \diff t\Big] + \E^{z_0}[\tau^x_{1-\varepsilon}] \nonumber\\
        &= \frac{1}{q(x\mid z_0)} \E^{z_0}\Big[\int_{0}^{\tau_{1-\varepsilon}} -\Delta\log f(W_t)q(x\mid W_t) \diff t\Big] + \E^{z_0}[\tau^x_{1-\varepsilon}]. \label{integral_delta_log_f}
    \end{align}
    The second term is bounded by a constant that is independent of $\varepsilon$ by Lemma~\ref{lem:expectation_tau_x}.
    For the first term, the idea is  to use Itô's formula to eliminate the time-integral, for which we need the Laplacian to act on the product $q(x\mid\cdot)\log f$. This can be achieved using the vector identity $q(x\mid\cdot)\Delta \log f= \Delta (q(x\mid\cdot)\log f) - \Delta_z q(x\mid\cdot)\log f - 2\nabla \log f\cdot\nabla_z q(x\mid\cdot)$. Moreover, the second term vanishes, since $q(x\mid\cdot)$ is harmonic. Inserting this identity in \eqref{integral_delta_log_f} yields
    \begin{align}
        &\E^{z_0}\Big[\int_0^{\tau^x_{1-\varepsilon}}\lVert \nabla\log f(Z_t^x)\rVert^2 \diff t\Big] \nonumber\\
        &\,\leq \frac{1}{q(x\mid z_0)} \bigg(-\E^{z_0}\Big[\int_{0}^{\tau_{1-\varepsilon}} \Delta(\log f(W_t)q(x\mid W_t)) \diff t\Big] + 2\E^{z_0}\Big[\int_{0}^{\tau_{1-\varepsilon}} \nabla\log f(W_t)\cdot \nabla_z q(x\mid W_t) \diff t\Big]\bigg) \nonumber\\
        &\,= -\frac{1}{q(x\mid z_0)}\E^{z_0}\Big[\int_{0}^{\tau_{1-\varepsilon}} \Delta(\log f(W_t)q(x\mid W_t)) \diff t\Big] + 2\E^{z_0}\Big[\int_{0}^{\tau^x_{1-\varepsilon}} \nabla\log f(Z_t^x)\cdot \nabla_z\log q(x\mid Z_t^x) \diff t\Big]. \label{integral_delta_log_f_2}
    \end{align}
    For the first term, we apply It\^o's formula:
    \begin{align*}
        \frac{1}{q(x\mid z_0)}\E^{z_0}\Big[\int_{0}^{\tau_{1-\varepsilon}} \Delta(\log f(W_t)q(x\mid W_t)) \diff t\Big] &= \frac{2}{q(x\mid z_0)}\bigg(\E^{z_0}[\log f(W_{\tau_{1-\varepsilon}})q(x\mid W_{\tau_{1-\varepsilon}})] \\
        &\qquad\qquad\quad -\E^{z_0}[\log f(W_{0})q(x\mid W_{0})] \\
        &\qquad\qquad\quad -\E^{z_0}\Big[\int_{0}^{\tau_{1-\varepsilon}} \nabla(\log f(W_t)q(x\mid W_t)) \cdot\diff W_t\Big]\bigg) \\
        &= 2\E^{z_0}[\log f(Z_{\tau^x_{1-\varepsilon}}^x)] - 2\E^{z_0}[\log f(Z_{0}^x)] \\
        &\quad -\E^{z_0}\Big[\int_{0}^{\tau_{1-\varepsilon}} \nabla(\log f(W_t)q(x\mid W_t)) \cdot\diff W_t\Big] \\
        &= 2 \log \frac{1 - (1-\varepsilon)^2}{1- \lVert z_0 \rVert^2} \\
        &\quad -\E^{z_0}\Big[\int_{0}^{\tau_{1-\varepsilon}} \nabla(\log f(W_t)q(x\mid W_t)) \cdot\diff W_t\Big],
    \end{align*}
    The stochastic integral term vanishes, since $\nabla(q(x\mid\cdot)\log f)$ is bounded in $B_{1-\varepsilon}$ and so is the expectation of its quadratic variation.
    For the second term in \eqref{integral_delta_log_f_2} we use the Cauchy--Schwarz inequality and therefore obtain with the above
    \begin{align*}
        \E^{z_0}\Big[\int_{0}^{\tau^x_{1-\varepsilon}}\lVert \nabla\log f(Z_t^x)\rVert^2 \diff t\Big] 
        &\leq 2 \log \frac{1- \lVert z_0 \rVert^2}{1 - (1-\varepsilon)^2} \\
        &\quad+ 2 
       \left(\E^{z_0}\Big[\int_{0}^{\tau^x_{1-\varepsilon}} \lVert\nabla\log f(Z_t^x)\rVert^2 \diff t\Big]\right)^{1/2} \left(\E^{z_0}\Big[\int_{0}^{\tau^x_{1-\varepsilon}} \lVert\nabla_z\log q(x\mid Z_t^x)\rVert^2 \diff t\Big] \right)^{1/2}
    \end{align*}
    This is now a quadratic inequality for the square root of the integral of interest, which implies
    \begin{align*} 
        \left(\E^{z_0}\Big[\int_{0}^{\tau^x_{1-\varepsilon}}\lVert \nabla\log f(Z_t^x)\rVert^2 \diff t\Big]\right)^{1/2}
        &\leq \left(\E^{z_0}\Big[\int_{0}^{\tau^x_{1-\varepsilon}} \lVert\nabla_z\log q(x\mid Z_t^x)\rVert^2 \diff t\Big] \right)^{1/2}\\ 
        &\quad+ \left(\E^{z_0}\Big[\int_{0}^{\tau^x_{1-\varepsilon}} \lVert\nabla_z\log q(x\mid Z_t^x)\rVert^2 \diff t\Big] + 2 \log \frac{1- \lVert z_0 \rVert^2}{1 - (1-\varepsilon)^2} \right)^{1/2} \\
        &\lesssim \left(\E^{z_0}\Big[\int_{0}^{\tau^x_{1-\varepsilon}} \lVert\nabla_z\log q(x\mid Z_t^x)\rVert^2 \diff t\Big] \right)^{1/2} + (1+ \log \varepsilon^{-1})^{1/2}\\
        &\lesssim (1+ \log \varepsilon^{-1})^{1/2},
    \end{align*}
    where in the last step we used Lemma~\ref{lem:L^2_bound_score}. 
    Squaring both sides and inserting in \eqref{Lsx} finally yields the claim.
\end{proof}

\section{Proofs for Section \ref{sec:approx}}\label{app:approx}
We start with proving the representation of the explicit score loss decomposition in Proposition \ref{prop:approximation_error_decomposition}.

\begin{proof}[Proof of Proposition \ref{prop:approximation_error_decomposition}]
    By writing $(P_t^{1-\varepsilon})_{t \geq 0}$ for the semigroup of the Brownian motion killed on first exit of $B_{1-\varepsilon}$, then the potential measure $U^{1-\varepsilon}$ for the killed Brownian Motion started in $z_0$ is given by $U^{1-\varepsilon}(z_0,\diff{z}) = \int_0^\infty P_t^{1-\varepsilon}(z_0,\diff{z}) \diff{t} = G_{1-\varepsilon}(z_0,z) \diff{z}$ and
    \begin{align} 
        \E^{z_0}\Big[\int_0^{\tau_{1-\varepsilon}} \lVert s(W_t)-\nabla\log h(W_t) \rVert^2 h(W_t)\diff t\Big] &= \int_0^\infty \int_{B_{1-\varepsilon}} \lVert s(z) - \nabla \log h(z) \rVert^2 h(z) P_t^{1-\varepsilon}(z_0,\diff{z}) \diff{t} \nonumber\\ 
        &= \int_{B_{1-\varepsilon}}  \lVert s(z) - \nabla \log h(z) \rVert^2 h(z) U^{1-\varepsilon}(z_0,\diff{z}) \nonumber\\ 
        &= \int_{B_{1-\varepsilon}} G_{1-\varepsilon}(z_0,z) \lVert s(z) - \nabla \log h(z) \rVert^2 h(z) \diff{z}, \label{representation_formula}
    \end{align}
    which yields the first statement about the representation of the explicit score loss as a weighted $L^2$ norm with respect to the Lebesgue density $G(z_0,\cdot)h$. 
    The function $h$ can be bounded by $\lVert \pi^\ast\rVert_{L^\infty(\partial B)}$ once again. The Green kernel, however, diverges as $z \to z_0$ with rate $\mathcal O (\lVert \cdot-z_0\rVert^{-d+2})$. Therefore, we split the ball into a smaller ball $B_R$, which contains $z_0$ and an annulus $B_{1-\varepsilon}\setminus B_R$ on which $G_{1-\varepsilon}$ is bounded. This yields the estimate
    \begin{align*}
        \int_{B_{1-\varepsilon}} G_{1-\varepsilon}(z,z_0) \lVert s(z) - \nabla\log h(z) \rVert^2 h(z) \diff z &\leq \lVert \pi^\ast\rVert_{L^\infty(\partial B)} \left\lVert (s-\nabla \log h)\sqrt{G_{1-\varepsilon}(z_0,\cdot)} \right\rVert_{L^2(B_{1-\varepsilon}\setminus B_R)}^2 \\
        &+ \lVert s-\nabla \log h\rVert_{L^\infty(B_R)}^2\int_{B_R} G_{1-\varepsilon}(z,z_0) \diff z\Big).
    \end{align*}
    The remaining integral over the Green function can be evaluated by applying \eqref{representation_formula} in reverse direction and replacing the integrand $\lVert s - \nabla \log h \rVert^2 h$ by 1:
    \begin{equation*}
        \int_{B_R} G_{1-\varepsilon}(z,z_0) \diff z \leq \int_{B_{1-\varepsilon}} G_{1-\varepsilon}(z,z_0) \diff z = \E^{z_0}[\tau_{1-\varepsilon}] = \frac{(1-\varepsilon)^2-\lVert z_0\rVert^2}{d} <\infty.
    \end{equation*}
\end{proof}

The remainder of this section is devoted to the proof of Theorem~\ref{thm:NN_approximation}. To this end, our goal is to first motivate and derive the central technical results Lemma~\ref{lem:approx_score_trunc} and Proposition~\ref{prop:approx_nn_trunc} based on the approximation strategy outlined in Section \ref{sec:approx}, and then put things together for the proof of Theorem~\ref{thm:NN_approximation}. In our analysis we will make repeated use of the following uniform bounds on the gradient $\nabla h_N$ and the Hessian $\nabla^2 h_N$. It also nicely shows the necessity for the restriction $\alpha>(d-1)/2$ on the smoothness of $\pi^*$. 

\begin{lemma}\label{lem:bounds_grad_Hess_of_h_N}
    For $r\in (0,1)$, it holds 
    \begin{enumerate}[label = (\roman*), ref = (\roman*)]
        \item \label{bound_1} $\lVert\nabla h_N (r,\cdot)\rVert_{L^\infty(\partial B)} \lesssim \frac{1}{1-r^2}\lVert \pi^\ast\rVert_{H^{(d-1)/2}(\partial B)},$
        \item \label{bound_2} $\lVert \nabla^2 h_N(r,\cdot)\rVert_{L^\infty(\partial B)}\coloneqq \sup_{i,j} \lVert\partial_i\partial_j h_N(r,\cdot)\rVert_{L^\infty(\partial B)}  \lesssim \frac{r}{(1-r^2)^2} \lVert \pi^\ast\rVert_{H^{(d-1)/2}(\partial B)}$.
    \end{enumerate}
\end{lemma}

\begin{proof}
    Let $x\in\partial B$. Then, using \autocite[Theorem~4]{Seeley_1966}
    \begin{align*}
        \lVert \nabla h_N(r,x)\rVert &\leq \sum_{l=1}^N \sum_{m=1}^{M_l} \lvert a_{lm} \rvert\big(lr^{l-1} \lvert Y_{lm}(x)\rvert \lVert x\rVert + r^{l-1} \lVert \nabla_{\partial B}Y_{lm}(x)\rVert \big) \nonumber\\
        &\lesssim \sum_{l=1}^N\sum_{m=1}^{M_l} \lvert a_{lm} \rvert r^{l-1}\big(l\, l^{d/2-1} + l^{d/2} \big) \nonumber\\
        &\leq \left(\sum_{l=1}^N lr^{2(l-1)} \right)^{1/2} \left(\sum_{l=1}^N\sum_{m=1}^{M_l} \lvert a_{lm} \rvert^2 l^{d-1} \right)^{1/2} \nonumber\\
        &\leq \frac{1}{1-r^2} \lVert \pi^\ast\rVert_{H^{(d-1)/2}(\partial B)},
    \end{align*}
    which proves \ref{bound_1}. For the Hessian of $h_N$, we have
    \begin{align*}
        \nabla^2 h_N(r,x) &= \sum_{l=1}^N \sum_{m=1}^{M_l} a_{lm}r^{l-2}\Big[ l\big(\mathbb{I}+(l-2)xx^\top )Y_{lm}(x) + l\big(x\nabla_{\partial B}Y_{lm}(x)^\top + \nabla_{\partial B}Y_{lm}(x)x^\top\big) \\
        &\qquad\qquad\qquad\qquad + \nabla_{\partial B}^2Y_{lm}(x) \Big].
    \end{align*}
    By using \autocite[Theorem~4]{Seeley_1966} once again, it follows similarly
    \begin{align*}
        \lVert \nabla^2 h_N(r,\cdot)\rVert_{L^\infty(\partial B)} &\leq \sum_{l=1}^N \sum_{m=1}^{M_l} \lvert a_{lm}\rvert r^{l-2}\big[l(l-1)\lVert Y_{lm}\rVert_{L^\infty(\partial B)} + 2l\lVert\nabla_{\partial B}Y_{lm}\rVert_{L^\infty(\partial B)} + \lVert\nabla_{\partial B}^2Y_{lm}\rVert_{L^\infty(\partial B)} \big] \\
        &\lesssim \sum_{l=1}^N \sum_{m=1}^{M_l} \lvert a_{lm}\rvert r^{l-2} \big(l^2l^{d/2-1} + l\, l^{d/2} + l^{d/2+1}\big) \\
        &\leq \left(\sum_{l=1}^N l^3 r^{2(l-2)} \right)^{1/2} \left(\sum_{l=1}^N \sum_{m=1}^{M_l} \lvert a_{lm}\rvert^2 l^{d-1} \right)^{1/2} \\
        &\leq \left(\frac{1}{r^2}+8+6r^2\sum_{l=3}^N l(l-1)(l-2) r^{2(l-3)} \right)^{1/2} \lVert\pi^\ast\rVert_{H^{(d-1)/2}(\partial B)} \\
        &\leq \left(\frac{1}{R^2}+8+6r^2\frac{6}{(1-r^2)^4} \right)^{1/2} \lVert\pi^\ast\rVert_{H^{(d-1)/2}(\partial B)} \\
        &\lesssim \frac{r}{(1-r^2)^2}\lVert\pi^\ast\rVert_{H^{(d-1)/2}(\partial B)},
    \end{align*}
    which proves \ref{bound_2}.
\end{proof}

Defining the full network as
\begin{equation}\label{overline_s}
    \overline{s}(z)\coloneqq \phi_{\mathrm{mult}}(s_{\nabla h_N}(z),\phi_{\mathrm{rec}}(s_{h_N}(z)\lor\pi^\ast_{\mathrm{min}}), \qquad z\in B_{1-\varepsilon}\setminus B_R,
\end{equation}
with a multiplication network $\phi_{\mathrm{mult}}$ as in Lemma~\ref{lem:multiplication_function}, a network $\phi_{\mathrm{rec}}$ approximating $z \mapsto 1/z$ as in Lemma~\ref{lem:reciprocal_function}, and approximating networks $s_{h_N}$ and $s_{\nabla h_N}$ for $h_N$ and $\nabla s_{h_N}$ defined later, the total approximation error on $B_{1-\varepsilon}\setminus B_R$ can  be bounded by 
\begin{align}
    \begin{split}
        \lVert (\overline{s}-\nabla\log h)\sqrt{G_{1-\varepsilon}(z_0,\cdot)}\rVert_{L^2(B_{1-\varepsilon}\setminus B_R)} &\leq \lVert (\overline{s}- s_{\nabla h_N}\phi_{\mathrm{rec}}\circ (s_{h_N}\lor\pi^\ast_{\mathrm{min}}))\sqrt{G_{1-\varepsilon}(z_0,\cdot)}\rVert_{L^2(B_{1-\varepsilon}\setminus B_R)} \\
        &\quad+ \left\lVert s_{\nabla h_N} \Big(\phi_{\mathrm{rec}}\circ (s_{h_N}\lor\pi^\ast_{\mathrm{min}}) - \frac{1}{s_{h_N}\lor\pi^\ast_{\mathrm{min}}} \Big)\sqrt{G_{1-\varepsilon}(z_0,\cdot)}\right\rVert_{L^2(B_{1-\varepsilon}\setminus B_R)} \\
        &\quad+ \left\lVert  \Big(\frac{s_{\nabla h_N}}{s_{h_N}\lor\pi^\ast_{\mathrm{min}}} - \nabla\log h \Big)\sqrt{G_{1-\varepsilon}(z_0,\cdot)}\right\rVert_{L^2(B_{1-\varepsilon}\setminus B_R)}.
    \end{split}\label{L2_error_decomposition_spherical}
\end{align}
The first term is controlled by
\begin{align*}
    &\lVert (\overline{s}- s_{\nabla h_N}\phi_{\mathrm{rec}}\circ (s_{h_N}\lor\pi^\ast_{\mathrm{min}}))\underbrace{\sqrt{G_{1-\varepsilon}(z_0,\cdot)}}_{\lesssim 1}\rVert_{L^2(B_{1-\varepsilon}\setminus B_R)} \\
    &\,\lesssim 2^{-l_1} \lVert s_{\nabla h_N}\rVert_{L^\infty(B_{1-\varepsilon})}\,\lVert \phi_{\mathrm{rec}}\circ (s_{h_N}\lor\pi^\ast_{\mathrm{min}})\rVert_{L^\infty(B_{1-\varepsilon})} \\
    &\,\lesssim 2^{-l_1}\varepsilon^{-1} \left( \Big\lVert \phi_{\mathrm{rec}}\circ (s_{h_N}\lor\pi^\ast_{\mathrm{min}})-\frac{1}{s_{h_N}\lor\pi^\ast_{\mathrm{min}}} \Big\rVert_{L^\infty(B_{1-\varepsilon})} + \Big\lVert\frac{1}{s_{h_N}\lor\pi^\ast_{\mathrm{min}}} \Big\rVert_{L^\infty(B_{1-\varepsilon})}\right) \\
    &\,\lesssim 2^{-l_1}\varepsilon^{-1}(2^{-l_2}+\pi_{\mathrm{min}}^{-1}),
\end{align*}
while the second term is bounded by
\begin{align*}
    \left\lVert s_{\nabla h_N} \Big(\phi_{\mathrm{rec}}\circ (s_{h_N}\lor\pi^\ast_{\mathrm{min}}) - \frac{1}{s_{h_N}\lor\pi^\ast_{\mathrm{min}}} \Big)\sqrt{G_{1-\varepsilon}(z_0,\cdot)}\right\rVert_{L^2(B_{1-\varepsilon}\setminus B_R)}&\lesssim \left\lVert \phi_{\mathrm{rec}}\circ (s_{h_N}\lor\pi^\ast_{\mathrm{min}}) - \frac{1}{s_{h_N}\lor\pi^\ast_{\mathrm{min}}}\right\rVert_{L^\infty(B_{1-\varepsilon})} \\
    &\quad\times \int_{B_{1-\varepsilon}} \frac{1}{(1-\lVert z\rVert)^2}G_{1-\varepsilon}(z_0,z)\diff z \\
    &\lesssim 2^{-l_2}\log \varepsilon^{-1},
\end{align*}
using Lemma~\ref{lem:expectation_one_minus_norm_BM}. The third and last term is the most challenging one and can be further decomposed as
\begin{align}
    &\left\lVert \Big(\frac{s_{\nabla h_N}}{s_{h_N}\lor\pi^\ast_{\mathrm{min}}}-\nabla\log h \Big)\sqrt{G_{1-\varepsilon}(z_0,\cdot)} \right\rVert_{L^2(B_{1-\varepsilon}\setminus B_R)}\nonumber\\ 
    &\,= \left\lVert \Big(\frac{\nabla h}{h} -\frac{\nabla h_N}{h_N\lor\pi^\ast_{\mathrm{min}}}\Big)\sqrt{G_{1-\varepsilon}(z_0,\cdot)} \right\rVert_{L^2(B_{1-\varepsilon}\setminus B_R)} + \left\lVert \Big(\frac{\nabla h_N}{h_N\lor\pi^\ast_{\mathrm{min}}} - \frac{s_{\nabla h_N}}{s_{h_N}\lor\pi^\ast_{\mathrm{min}}} \Big)\sqrt{G_{1-\varepsilon}(z_0,\cdot)} \right\rVert_{L^2(B_{1-\varepsilon}\setminus B_R)} \nonumber\\
    &\,\leq \left\lVert \Big(\frac{\nabla h - \nabla h_N}{h} \Big)\sqrt{G_{1-\varepsilon}(z_0,\cdot)} \right\rVert_{L^2(B_{1-\varepsilon}\setminus B_R)} + \left\lVert \nabla h_N \Bigg( \frac{1}{h} - \frac{1}{h_N\lor\pi^\ast_{\mathrm{min}}} \Bigg)\sqrt{G_{1-\varepsilon}(z_0,\cdot)}\right\rVert_{L^2(B_{1-\varepsilon}\setminus B_R)} \nonumber\\
    &\quad + \left\lVert \frac{\nabla h_N - s_{\nabla h_N}}{h_N\lor\pi^\ast_{\mathrm{min}}}\sqrt{G_{1-\varepsilon}(z_0,\cdot)} \right\rVert_{L^2(B_{1-\varepsilon}\setminus B_R)} + \left\lVert s_{\nabla h_N} \Bigg( \frac{1}{h_N\lor\pi^\ast_{\mathrm{min}}} - \frac{1}{s_{h_N}\lor\pi^\ast_{\mathrm{min}}} \Bigg)\sqrt{G_{1-\varepsilon}(z_0,\cdot)}\right\rVert_{L^2(B_{1-\varepsilon}\setminus B_R)} \nonumber\\
    \begin{split}
        &\,\lesssim \pi_{\mathrm{min}}^{-1}\lVert \nabla h - \nabla h_N \rVert_{L^2(B_{1-\varepsilon})} + \pi_{\mathrm{min}}^{-2} \lVert\nabla h_N\underbrace{\lvert h - h_N\lor\pi^\ast_{\mathrm{min}}\rvert}_{\leq \lvert h-h_N\rvert} \sqrt{G_{1-\varepsilon}(z_0,\cdot)}\rVert_{L^2(B_{1-\varepsilon}\setminus B_R)} \\
        &\quad+ \pi_{\mathrm{min}}^{-1}\lVert (\nabla h_N - s_{\nabla h_N})\sqrt{G_{1-\varepsilon}(z_0,\cdot)} \rVert_{L^2(B_{1-\varepsilon}\setminus B_R)} \\
        &\quad+ \pi_{\mathrm{min}}^{-2}\lVert s_{\nabla h_N}\underbrace{\lvert h_N\lor\pi^\ast_{\mathrm{min}} - s_{h_N}\lor\pi^\ast_{\mathrm{min}}\rvert}_{\leq \lvert h_N-s_{h_N}\rvert} \sqrt{G_{1-\varepsilon}(z_0,\cdot)}\rVert_{L^2(B_{1-\varepsilon}\setminus B_R)}, \label{L2_approximation_spherical}
    \end{split}
\end{align}
where we used $h=\E^\cdot[\pi^\ast(W_\tau)]\geq \pi_{\mathrm{min}}\coloneqq \inf_{x\in\partial B} \pi^\ast(x) > 0$. The first two terms in \eqref{L2_approximation_spherical} are bounded by Lemma~\ref{lem:approx_score_trunc}, which we now prove.

\begin{proof}[Proof of Lemma \ref{lem:approx_score_trunc}]
    For part \ref{approx_1}, we first expand the left hand side as
    \begin{equation}\label{L2_norm_grad_h_N_h_minus_h_N_}
    \begin{split}
        &\lVert \nabla h_N(h - h_N)\sqrt{G_{1-\varepsilon}(z_0,\cdot)} \rVert_{L^2(B_{1-\varepsilon}\setminus B_R)}^2 \\ 
        &\,\leq \int_R^{1-\varepsilon} \lVert\nabla h_N(r,\cdot) \rVert^2_{L^\infty(\partial B)} \lVert (h(r,\cdot)-h_N(r,\cdot)) \underbrace{\sqrt{G_{1-\varepsilon}(z_0,(r,\cdot))}}_{\lesssim \sqrt{1-r}} \rVert_{L^2(\partial B)}^2  r^{d-1}\diff r.
    \end{split}
    \end{equation}
    According to Lemma~\ref{lem:Green_function_bound}, the Green kernel is bounded by $1-r$ up to constants. The $L^2$-distance between $h(r,\cdot)$ and $h_N(r,\cdot)$ evaluates as
    \begin{align}
        \lVert h(r,\cdot)-h_N(r,\cdot)\rVert_{L^2(\partial B)}^2 &= \sum_{l=N+1}^\infty\sum_{m=1}^{M_l} r^{2l} a_{lm}^2 \nonumber\\
        &\lesssim(1-\varepsilon)^{2(N+1)}N^{-2\alpha} \sum_{l=0}^\infty\sum_{m=1}^{M_l} (l(l+d-2))^\alpha \lvert\left\langle \pi^\ast,Y_{lm} \right\rangle \rvert^2 \nonumber\\
        &\lesssim(1-\varepsilon)^{2(N+1)}N^{-2\alpha}\lVert \pi^\ast\rVert_{H^\alpha(\partial B)}^2.\label{L2_norm_h_minus_h_N}
    \end{align}
    Inserting \eqref{L2_norm_h_minus_h_N} and Lemma~\ref{lem:bounds_grad_Hess_of_h_N}\ref{bound_1} in \eqref{L2_norm_grad_h_N_h_minus_h_N_}, and using $(1-r^2)^{-1}\leq (1-r)^{-1}$, one gets
    \begin{align*}
        \lVert \nabla h_N(h - h_N) \rVert_{L^2(B_{1-\varepsilon})}^2 &\lesssim (1-\varepsilon)^{2(N+1)}N^{-2\alpha}\lVert \pi^\ast\rVert_{H^{(d-1)/2}(\partial B)}^2\lVert \pi^\ast\rVert_{H^\alpha(\partial B)}^2 \int_0^{1-\varepsilon} \frac{r^{d-1}}{1-r}\diff r \\
        &\lesssim (1-\varepsilon)^{2N+d+1}N^{-2\alpha}\log\varepsilon^{-1} \lVert \pi^\ast\rVert_{H^\alpha(\partial B)}^4,
    \end{align*}
    which proves \ref{approx_1}. For part \ref{approx_2} we apply a similar idea to the expansion \eqref{grad_h_N_expansion}:
    \begin{align*}
        \lVert \nabla h(r,\cdot)-\nabla h_N(r,\cdot) \rVert_{L^2(\partial B)}^2 = \sum_{l,k=N+1}^\infty \sum_{m=1}^{M_l}\sum_{m'=1}^{M_k} r^{l+k-2} a_{lm}a_{km'} \Big(&l^2 \underbrace{\int_{\partial B} Y_{lm}(x) Y_{km^\prime}(x) \overbrace{\lVert x\rVert^2}^{=1} \,\sigma(\diff{x})}_{=\delta_{lk}\delta_{mm'}} \\
        &+ 2l\int_{\partial B} Y_{km'}(x) x\cdot\nabla_{\partial B} Y_{lm}(x)\,\sigma(\diff{x}) \\
        &+ \langle \nabla_{\partial B} Y_{lm}\cdot \nabla_{\partial B} Y_{km'}\rangle\Big).
    \end{align*}
    The second term vanishes, since the covariant derivative $\nabla_{\partial B}Y_{lm}(x)$ is tangential to the manifold and thus orthogonal to the normal vector $x$. The third term can be evaluated using integration by parts on the sphere:
    \begin{equation*}
        \langle \nabla_{\partial B} Y_{lm}\cdot \nabla_{\partial B} Y_{km'}\rangle = \int_{\partial B} \nabla_{\partial B} Y_{lm} \cdot\nabla_{\partial B} Y_{km'}\diff\sigma = -\int_{\partial B} Y_{km'} \Delta_{\partial B} Y_{lm}\diff\sigma = l(l+d-2)\delta_{lk}\delta_{km'} .
    \end{equation*}
    Thus, it follows
    \begin{equation*}
        \lVert \nabla h(r,\cdot)-\nabla h_N(r,\cdot) \rVert_{L^2(\partial B)}^2 \leq \sum_{l=N+1}^\infty\sum_{m=1}^{M_l} r^{2(l-1)} l(l+d-2)\lvert a_{lm}\rvert^2,
    \end{equation*}
    and
    \begin{align*}
        \lVert \nabla h-\nabla h_N \rVert_{L^2(B_{1-\varepsilon})}^2 &= \int_0^{1-\varepsilon} \lVert \nabla h(r,\cdot)-\nabla h_N(r,\cdot) \rVert_{L^2(\partial B)}^2 r^{d-1} \diff r \\
        &\leq \sum_{l=N+1}^\infty\sum_{m=1}^{M_l} l\underbrace{\frac{(l+d-2)}{2l+d-2}}_{\leq 1} (1-\varepsilon)^{2l+d-2} \lvert a_{lm}\rvert^2 \\
        &\leq (1-\varepsilon)^{2N+d}N^{-2(\alpha-1/2)}\sum_{l=0}^\infty\sum_{m=1}^{M_l}(l(l+d-2))^\alpha \lvert a_{lm}\rvert^2 \\
        &\lesssim (1-\varepsilon)^{2N+d}N^{-2\alpha+1} \lVert \pi^\ast\rVert_{H^\alpha(\partial B)}^2,
    \end{align*}
    which proves \ref{approx_2}.
\end{proof}

For the two remaining terms in \eqref{L2_approximation_spherical}, we turn to the approximation of $h_N$ and $\nabla h_N$. To this end, it will be necessary to separate the radial from the spherical direction. Thus, we define the stereographic projection maps (see, e.g., \cite{lee13} and \cite{lee97}) given by
\begin{align*}
    &\varphi_+(x)\coloneqq (r, \theta_1,\ldots,\theta_{d-1})= \left(\sqrt{\sum_{i=1}^d x_i^2}, -\frac{x_1}{1+x_d},\ldots,-\frac{x_{d-1}}{1+x_d} \right), \qquad x_d\in (-1,1], \\
    &\varphi_-(x)\coloneqq (r, \theta_1,\ldots,\theta_{d-1})= \left(\sqrt{\sum_{i=1}^d x_i^2}, \frac{x_1}{1-x_d},\ldots,\frac{x_{d-1}}{1-x_d} \right), \qquad x_d \in [-1,1).
\end{align*}
These two coordinate maps map the northern (for $\varphi_+$) and the southern (for $\varphi_-$) half of the unit ball to $[0,1]\times B^{(d-1)}_1$, where $B^{(d-1)}_R\subset \R^{d-1}$ denotes the ball of radius $R$ in $(d-1)$-dimensional space. Thus, for fixed radius $r$, their inverse maps parametrise the northern and the southern hemisphere by taking vectors from the $(d-1)$-dimensional unit sphere as input, which are given explicitly by
\begin{align*}
    &\varphi^{-1}_+(r,\theta) \coloneqq \left(\frac{2r^2\theta_1}{\lVert \theta\rVert^2+r^2}, \dots, \frac{2r^2\theta_{d-1}}{\lVert \theta\rVert^2+r^2}, r\frac{\lVert\theta\rVert^2-r^2}{\lVert \theta\rVert^2+r^2}\right), \\
    &\varphi^{-1}_-(r,\theta) \coloneqq \left(-\frac{2r^2\theta_1}{\lVert \theta\rVert^2+r^2}, \dots, -\frac{2r^2\theta_{d-1}}{\lVert \theta\rVert^2+r^2},- r\frac{\lVert\theta\rVert^2-r^2}{\lVert \theta\rVert^2+r^2}\right),
\end{align*}
The stereographic coordinate system has the advantage, compared to standard spherical coordinates, that it does not involve trigonometric functions and the determinant of its metric, which appears in surface integrals, is strictly bounded from below. This will come in handy, when approximating $\varphi_+$ and $\varphi_-$ with neural networks and for changing from the curvy-linear surface measure on $\partial B$ to the Lebesgue measure on $B^{(d-1)}_1$. In the following, since the analysis of the northern and southern half of the ball is completely analogous, we will sometimes write $\varphi_\pm$ as a place holder for either $\varphi_+$ or $\varphi_-$.

We will at some point consider functions at a fixed radius $r$ and approximate their composition with $\varphi^{-1}_\pm(r,\cdot)$. For better readability, we define $\varphi_{\partial B,\pm}(x)\coloneqq (\varphi_\pm(x)_2,\ldots,\varphi_\pm(x)_d)$ for $x\in\partial B$, which means that $\varphi_{\partial B,\pm}^{-1}=\varphi^{-1}_\pm(r,\cdot)/r=\varphi^{-1}_\pm(1,\cdot)$.

Our approximation strategy of $h_N$ and $\nabla h_N$ is based on the following steps:
\begin{enumerate}
    \item Approximate $h_N(r_i,\cdot)\circ\varphi_{\partial B,\pm}^{-1}$ and $\nabla h_N(r_i,\cdot)\circ\varphi_{\partial B,\pm}^{-1}$ uniformly on the coordinate space for the northern and southern half separately with neural networks $s_{h_N\circ\varphi^{-1}_\pm,r_i}$ and $s_{\nabla h_N\circ\varphi^{-1}_\pm,r_i}$ for fixed $r_i$.
    \item Interpolate $s_{h_N\circ\varphi^{-1}_\pm,r_i}$ and $s_{\nabla h_N\circ\varphi^{-1}_\pm,r_i}$ with $\mathcal O(\log N)$ Chebyshev polynomials and nodes $r_i$ on shells $[1-2^{m-1}\varepsilon,1-2^{m+1}\varepsilon]\times \partial B$, $m=1,\ldots,M$ and $M=\mathcal O(\log N)$.
    \item Approximate Chebyshev polynomials by neural networks (analogously to \cite{Asbjorn_2025}).
    \item Approximate $h_N\circ\varphi^{-1}_\pm$ and $\nabla h_N\circ\varphi^{-1}_\pm$ in $L^2$ by combining the approximations on the shells with a partition of unity of neural networks (analogously to \cite{Asbjorn_2025}) to networks $s_{h_N\circ\varphi^{-1}_\pm}$ and $s_{\nabla h_N\circ\varphi^{-1}_\pm}$.
    \item Approximate $\varphi_+$ and $\varphi_-$ by neural networks $s_{\varphi_+}$ and $s_{\varphi_-}$.
    \item Approximate $h_N$ with a partition of unity $(p_+, p_-)$ via $s_{h_N}=\phi_{\mathrm{mult}}\circ(p_+, s_{h_N\circ\varphi^{-1}_+}\circ s_{\varphi_+}) + \phi_{\mathrm{mult}}\circ(p_-, s_{h_N\circ\varphi^{-1}_-}\circ s_{\varphi_-})$ and similarly for $\nabla h_N$.
\end{enumerate}

The first four steps are summarised in the following proposition.

\begin{proposition}\label{prop:neural_network_approximation}
    There exist neural networks $s_{h_N\circ\varphi^{-1}_\pm},s_{\nabla h_N\circ\varphi^{-1}_\pm}\in \mathcal S(\mathsf{L,W,S,B})$ with
    \begin{align*}
        \mathsf L\lesssim \log N\log\log N, \qquad \mathsf W\lesssim N\log^2 N  \qquad \mathsf S\lesssim N\log^3 N, \qquad \mathsf B\lesssim N^{1/(d-1)}\lor \varepsilon^{-1},
    \end{align*}
    and for $r\in[R,1-\varepsilon]$
    \begin{align*}
        &\lVert h_N(r,\cdot)\circ\varphi^{-1}_{\partial B,\pm}-s_{h_N\circ\varphi^{-1}_\pm}(r,\cdot)\rVert_{L^\infty([-2,2]^{d-1})}\lesssim N^{-\alpha/(d-1)}\log N, \\
        &\lVert \nabla h_N(r,\cdot)\circ\varphi^{-1}_{\partial B,\pm}-s_{\nabla h_N\circ\varphi^{-1}_\pm}(r,\cdot)\rVert_{L^\infty([-2,2]^{d-1})}\lesssim\frac{1}{1-r} N^{-\alpha/(d-1)}\log N.
    \end{align*}
\end{proposition}

\begin{proof}
    We follow the step-by-step approach outlined above. Since the procedure is identitcal for both halves of the ball, we just write $\varphi_{\partial B}\coloneqq \varphi_{\partial B,\pm}$.

    \paragraph{Step 1: Approximation at fixed radius}

    For fixed radius $r\in[R,1-\varepsilon]$, we aim to apply Lemma~\ref{lem:Suzuki_approximation}, which means that we have to bound the Sobolev norm of $h_N(r,\cdot)\circ\varphi_{\partial B}^{-1}$ and $\nabla h_N(r,\cdot)\circ\varphi_{\partial B}^{-1}$. Using Lemma~\ref{lem:bound_regular_Sobolev_norm}, we get
    \begin{align*}
        \lVert h_N(r,\cdot)\circ\varphi_{\partial B}^{-1}\rVert_{H^\alpha([-2,2]^{d-1})}^2 &\lesssim \lVert h_N(r,\cdot)\rVert_{H^\alpha(\partial B)}^2 \\
        &=\lVert h_N(r,\cdot) \rVert_{L^2(\partial B)}^2 + \sum_{l=0}^N (l(l+d-2))^\alpha\sum_{m=1}^{M_l}\lvert\langle h_N(r,\cdot),Y_{lm}\rangle\rvert^2 \\
        &= \sum_{l=0}^N r^{2l}\sum_{m=1}^{M_l}\lvert a_{lm}\rvert^2 + \sum_{l=0}^N r^{2l} (l(l+d-2))^\alpha\sum_{m=1}^{M_l} \lvert a_{lm}\rvert^2 \\
        &\leq \left\lVert \pi^\ast \right\rVert_{H^\alpha(\partial B)}^2,
    \end{align*}
    where in the last step we used $r \leq 1$. Furthermore, using additionally Lemma~\ref{lem:homogeneous_harmonic_polynomials} for the second line, we obtain 
    \begin{align*}
        \lVert \nabla h_N(r,\cdot)\circ\varphi_{\partial B}^{-1}\rVert_{H^\alpha([-2,2]^{d-1})}^2 &\lesssim \lVert \nabla h_N(r,\cdot)\rVert_{H^\alpha(\partial B)}^2 \\
        &= \sum_{l,k=1}^N r^{l+k-2} \Big\langle\sum_{m=1}^{M_l}(lY_{lm}\cdot + \nabla_{\partial B}Y_{lm} ), \sum_{m'=1}^{M_k}(kY_{km'}\cdot + \nabla_{\partial B}Y_{km'} )\Big\rangle_{H^\alpha(\partial B)}^2 \\
        &= \sum_{l=1}^N r^{2(l-1)} \left(l(l+d-1)(1+((l-1)(l+d-3))^\alpha) \right) \lvert a_{lm}\rvert^2 \\
        &\lesssim \sum_{l=0}^N l^2r^{2l} (1 + (l(l+d-2))^\alpha) ) \lvert a_{lm}\rvert^2 \\
        &\leq \sup_{l\in\N} l^2r^{2l}\, \left\lVert \pi^\ast \right\rVert_{H^\alpha(\partial B)}^2.
    \end{align*}
    The supremum of the sequence $(l^2r^{2l})_{l\in\N}$ is obviously bounded by the supremum of its extension to the non-negative real line, which can be computed by setting its derivative to zero, which yields
    \begin{equation*}
        \sup_{l\in\N} l^2r^{2l}\leq \log(r)^{-2} \text{e}^{-2}\lesssim (1-r)^{-2}.
    \end{equation*}
    Thus, Lemma~\ref{lem:Suzuki_approximation} provides neural networks $\phi_{h_N\circ\varphi^{-1},r}$ and $\phi_{\nabla h_N\circ\varphi^{-1},r}$ with size
    \begin{align*}
        \mathsf L\lesssim \log N,&& \mathsf W\lesssim N,&&S\lesssim N\log N, && \mathsf B=\mathrm{Poly}(N),
    \end{align*}
    such that
    \begin{align*}
        &\lVert h_N(r,\cdot)\circ\varphi^{-1}_{\partial B}-\phi_{h_N\circ\varphi^{-1},r}\rVert_{L^\infty([-2,2]^{d-1})}\lesssim N^{-\alpha/(d-1)}, \\
        &\lVert \nabla h_N(r,\cdot)\circ\varphi^{-1}_{\partial B}-\phi_{\nabla h_N\circ\varphi^{-1},r}\rVert_{L^\infty([-2,2]^{d-1})}\lesssim\frac{1}{1-r} N^{-\alpha/(d-1)}.
    \end{align*}

    \paragraph{Step 2: Approximation on $[1-2^{m-1}\varepsilon,1-2^{m+1}\varepsilon]\times\partial B$ via Chebyshev interpolation}

    We now use step 1 to approximate $h_N$ and $\nabla h_N$ on $[b_m-a_m,b_m+a_m]\times\partial B$ with $a_m=3\times 2^{m-2}\varepsilon$ and $b_m=1- 5\times 2^{m-2}\varepsilon$ for $m=1,\ldots,M$, where $M\coloneqq \lceil\log_2((1-R)/\varepsilon)+1\rceil$ in order to cover the entire annulus $B_{1-\varepsilon}\setminus B_R$. Since the procedure is the same for both $h_N$ and $\nabla h_N$, we denote the rescaled functions $(r,x)\mapsto h_N(a_m r+b_m,x)$ and $(r,x)\mapsto \nabla h_N(a_m r+b_m,x)$, $r\in[-1,1]$, $x\in\partial B$, by $f_m$. This can be approximated by the Chebyshev polynomials
    \begin{equation}\label{Chebyshev_polynomials}
        \psi_{f_m}(r,x) = \sum_{i=1}^kc_i p_i(r)f_m(r_i,x), \quad r\in[-1,1],\, x\in\partial B,
    \end{equation}
    with $r_i=\cos(i\pi/k)$, $p_i(r)=\prod_{i\neq j}^k(r-r_i)$, $c_i=1/p_i(r_i)$ and $k\in\N$ to be determined later as well. Since $f_m$ is an entire function on $\C$, it holds that \cite[Theorem 8.2]{trefethen13}
    \begin{equation}\label{Chebyshev_approximation_rate}
        \lvert f_m(r,x) - \psi_{f_m}(r,x)\rvert \leq \frac{4M_{m,\rho}(x)\rho^{-k}}{\rho-1},
    \end{equation}
    for some $\rho>1$ and
    \begin{equation*}
        M_{m,\rho}(x)\coloneqq \max_{z\in\partial E_{\rho}} \lvert f_m(z,x)\rvert,\qquad \partial E_\rho\coloneqq \left\{\frac{z+z^{-1}}{2} : z\in\C,\lvert z\rvert=\rho \right\}.
    \end{equation*}
    This function is indeed bounded for all $\rho>1$ on $[b_m-a_m,b_m+a_m]\times\partial B$. To see this, note that for $y=(z+z^{-1})/2\in E_\rho$ with $z=\rho\mathrm{e}^{\mathrm{i}\varphi}$
    \begin{align*}
        &\lvert y\rvert^2 = \frac{\lvert z\rvert^2+\lvert z\rvert^{-2}+2\mathrm{Re}(z/\bar{z})}{4} = \frac{\rho^2+\rho^{-2}+2\cos(2\varphi)}{4}\leq \left(\frac{\rho+\rho^{-1}}{2}\right)^2, \\
        &\mathrm{Re}(y) = \frac{\mathrm{Re}(z)+\mathrm{Re}(1/z)}{2}=\frac{\rho+\rho^{-1}}{2}\cos(\varphi)\leq \frac{\rho+\rho^{-1}}{2}.
    \end{align*}
    Thus,
    \begin{align*}
        \lvert h_N(y,x)\rvert &\leq\sum_{l=0}^N\lvert a_my+b_m\rvert^l \lvert Y_{lm}(x)\rvert\lvert a_{lm}\rvert \\
        &=\sum_{l=0}^N (a_m^2\lvert y\rvert^2+2a_mb_m\mathrm{Re}(y)+b_m^2)^{l/2}\lvert Y_{lm}(x)\rvert\lvert a_{lm}\rvert \\
        &\leq \sum_{l=0}^N \Bigg[a_m^2\left(\frac{\rho+\rho^{-1}}{2} \right)^2+2a_mb_m\frac{\rho+\rho^{-1}}{2}+b_m^2 \Bigg]^{l/2}\lvert Y_{lm}(x)\rvert\lvert a_{lm}\rvert \\
        &=\sum_{l=0}^N \left(a_m\frac{\rho+\rho^{-1}}{2}+b_m \right)^l\lvert Y_{lm}(x)\rvert\lvert a_{lm}\rvert,
    \end{align*}
    and similarly
    \begin{align*}
        \lVert \nabla h_N(y,x)\rVert_{\infty} &\leq\sum_{l=1}^N\lvert a_my+b_m\rvert^{l-1} \lvert a_{lm}\rvert (l\lvert Y_{lm}(x)\rvert\, \underbrace{\lVert x\rVert_{\infty}}_{\leq1} + \lVert \nabla_{\partial B}Y_{lm}(x)\rVert_{\infty}) \\
        &\leq \sum_{l=1}^N \left(a_m\frac{\rho+\rho^{-1}}{2}+b_m \right)^{l-1} \lvert a_{lm}\rvert (l\lvert Y_{lm}(x)\rvert + \lVert \nabla_{\partial B}Y_{lm}(x)\rVert_{\infty}).
    \end{align*}
    For the right-hand sides to be bounded for $N\to\infty$, one needs $a_m(\rho+\rho^{-1})/2+b_m\leq 1$ for all $m\in\N$, which holds for $\rho \in (1,3]$. Fix such $\rho$ in the following. Theorem~4 in \cite{Seeley_1966} provides the bounds $\lvert Y_{lm}(x)\rvert \lesssim l^{d/2-1}$ and $\lVert \nabla_{\partial B}Y_{lm}(x)\rVert\lesssim l^{d/2}$. This results in
    \begin{align*}
        &\max_{y\in E_\rho}\lvert h_N(y,x)\rvert \lesssim \sum_{l=0}^N l^{d/2-1}\lvert a_{lm}\rvert \leq \sqrt{N} \sqrt{\sum_{l=0}^N l^{d-2}\lvert a_{lm}\rvert^2}\leq \sqrt{N}\lVert \pi^\ast\rVert_{H^{d/2-1}(\partial B)}, \\
        &\max_{y\in E_\rho}\lVert \nabla h_N(y,x)\rVert \lesssim \sum_{l=1}^N l^{d/2} \lvert a_{lm}\rvert \leq \sqrt{\sum_{l=1}^N l}\sqrt{\sum_{l=1}^N l^{d-1}\lvert a_{lm}\rvert^2 }\leq \sqrt{\frac{N(N+1)}{2}} \lVert \pi^\ast\rVert_{H^{(d-1)/2}(\partial B)},
    \end{align*}
    using Hölder's inequality in both cases, which means that
    \begin{align*}
        \lVert M_{m,\rho}\rVert_{L^\infty(\partial B)} \lesssim N \lVert\pi^\ast\rVert_{H^{(d-1)/2}(\partial B)}
    \end{align*}
    and for \eqref{Chebyshev_approximation_rate}
    \begin{equation*}
        \lVert f_m(r,\cdot) - \psi_{f_m}(r,\cdot)\rVert_{L^\infty(\partial B)} \lesssim N \lVert\pi^\ast\rVert_{H^{(d-1)/2}(\partial B)} \rho^{-k}.
    \end{equation*}
    Choosing $k\coloneqq \lceil (\alpha/(d-1)+1)\log_\rho N\rceil$ yields the desired $N^{-\alpha/(d-1)}$ rate of convergence.

    Now, since $f_m$ can be approximated by a logarithmic number of Chebyshev polynomials with the appropriate rate, $\psi_m$ can be approximated summand-by-summand with a polynomial rate in $N$, while keeping a total network size of $\mathcal O(\log N)$.
    Thus, we define the neural network approximation for $\psi_m$ as
    \begin{align*}
        \phi_m(r,x)=\sum_{i=1}^kc_i\, \phi_{\mathrm{mult}}(\phi_{p_i}(r),\phi_{f,a_m r_i+b_m}(x)),
    \end{align*}
    where $\phi_{p_i}$ is the neural network approximation of $p_i$ constructed in step 3 and $\phi_{\mathrm{mult}}$ as in Lemma~\ref{lem:multiplication_function}. Repeating the arguments from \autocite[Lemma~3.13.]{Asbjorn_2025}, where we replace $\underline{T}$ by $\varepsilon^2$, it follows that
    \begin{equation*}
        \lVert \phi_m(r,\cdot)-\psi_{f_m}(r,\cdot)\rVert_{L^\infty(\partial B)}\lesssim \begin{cases}
            N^{-\alpha/(d-1)}\log N, & \text{if }f_m=h_N \\
            \frac{1}{2^{m-1}\varepsilon}N^{-\alpha/(d-1)}\log N, &\text{if } f_m=\nabla h_N,
        \end{cases}
    \end{equation*}
    with a network size
    \begin{align*}
        \mathsf L\lesssim \log N\log\log N,&& \mathsf W\lesssim N,&& \mathsf S\lesssim N\log N, && \mathsf B\lesssim N^{1/(d-1)}\lor \varepsilon^{-1}.
    \end{align*}

    \paragraph{Step 3: Approximation of Chebyshev polynomials}
    
    The approximation of the Chebyshev polynomials $p_i$ with neural networks $\phi_{p_i}\in\mathrm{NN}(\mathsf{L,W,S,B})$ has been shown in \autocite[Lemma~3.13]{Asbjorn_2025}, which have sizes
    \begin{align*}
        &\mathsf{L}\lesssim \log N\log\log N, && \mathsf W\lesssim \log N \\
        &\mathsf{S}\lesssim \log^2 N, && \mathsf B = const.
    \end{align*}
    
    \paragraph{Step 4: Combining neural networks on shells with partition of unity}

    From the three previous steps, we get neural networks $\phi_1,\dots,\phi_M$ such that
    \begin{align*}
        \lVert \phi_m(r,\cdot)-f_m(r,\cdot)\rVert_{L^\infty(\partial B)} \lesssim \begin{cases}
            N^{-\alpha/(d-1)}\log N, & \text{if } f=h_N \\
            \frac{1}{2^{m-1}\varepsilon}N^{-\alpha/(d-1)}\log N, & \text{if }f=\nabla h_N,
        \end{cases}
    \end{align*}
    for $r\in[1-2^{m+1}\varepsilon,1-2^{m-1}\varepsilon]$. We define a partition of unity $(\mathsf p_m)_{m=1,\dots,M}$ in radial direction via
    \begin{align*}
        &\mathsf p_1(r) \coloneqq 0\lor \left(1 \land \frac{1-2^{m+1}\varepsilon-r}{2^{m-1}\varepsilon}\right), \\
        &\mathsf p_m(r) \coloneqq 0\lor \left(\frac{r-(1-2^{m-1}\varepsilon)}{2^{m-1}\varepsilon} \land \frac{1-2^{m+1}\varepsilon-r}{2^{m-1}\varepsilon}\right), \quad m=2,\dots,M-1, \\
        &\mathsf p_M(r) \coloneqq 0\lor \left(\frac{r-(1-2^{m-1}\varepsilon)}{2^{m-1}\varepsilon} \land 1\right).
    \end{align*}
    These functions can be represented exactly by two-layer neural networks due to Lemma~\ref{lem:maximum_and_minimum_function} (expressing max and min by neural networks). Thus, we define the overall network for $f$, being either $h_N$ or $\partial_i h_N$, $i=1,\dots,d$, as
    \begin{align*}
        s_{f\circ\varphi^{-1}}(r,x)\coloneqq \sum_{m=1}^M \phi_{\mathrm{mult}}(\phi_m(r,x),\mathsf p_m(r)),
    \end{align*}
    where $\phi_{\mathrm{mult}}$ is again given by Lemma~\ref{lem:multiplication_function}.

   \paragraph{Step 5: Putting things together} 
    
    Using the approximation properties of $\phi_{\mathrm{mult}}$ and the results from the previous steps, we can now control the approximation error of $s_{f_m\circ\varphi^{-1}}$: for $r\in[R,1-\varepsilon]$, there exists $m\in\{1,\dots,M-1\}$ such that $r\in[1-2^{m}\varepsilon,1-2^{m-1}\varepsilon]$ and
    \begin{align*}
        \lVert f\circ\varphi^{-1}(r,\cdot) - s_{f\circ\varphi^{-1}}(r,\cdot) \rVert_{L^\infty(\partial B)} &= \lVert \phi_{\mathrm{mult}}(\phi_m(r,\cdot),\mathsf p_m(r)) + \phi_{\mathrm{mult}}(\phi_{m+1}(r,\cdot),\mathsf p_{m+1}(r)) - f\circ\varphi^{-1}(r,\cdot)\rVert_{L^\infty(\partial B)} \\
        &\leq 2\times 2^{-l_1}\lVert \phi_m(r,\cdot)\rVert_{L^\infty(\partial B)} + \lVert \mathsf p_m(r) (\phi_m(r,\cdot) - f_m\circ\varphi^{-1}(r,\cdot))\rVert_{L^\infty(\partial B)} \\
        &\quad + \lVert \mathsf p_{m+1}(r)(\phi_{m+1}(r,\cdot) - f_{m+1}\circ\varphi^{-1}(r,\cdot))\rVert_{L^\infty(\partial B)} \\
        &\lesssim 2^{-l_1+1}\varepsilon^{-1} + \lVert \mathsf \phi_m(r,\cdot) - f_m\circ\varphi^{-1}(r,\cdot)\rVert_{L^\infty(\partial B)} \\
        &\quad + \lVert \mathsf \phi_{m+1}(r,\cdot) - f_{m+1}\circ\varphi^{-1}(r,\cdot)\rVert_{L^\infty(\partial B)} \\
        &\lesssim \varepsilon^{-1}2^{-l_1} + \begin{cases}
            N^{-\alpha/(d-1)}\log N, &\text{if }f=h_N \\
            \frac{1}{2^{m-2}\varepsilon}N^{-\alpha/(d-1)}\log N, &\text{if }f=\nabla h_N.
        \end{cases}
    \end{align*}
    Choosing $l_1\coloneqq \lceil\log_2(\varepsilon^{-1}) + \log_2(\alpha/(d-1))\rceil$ and noting that
    \begin{equation*}
        \frac{1}{2^{m-2}\varepsilon} = \frac{4}{1-(1-2^m\varepsilon)}\leq \frac{4}{1-r},
    \end{equation*}
    we finally get the desired approximation error bound.

    For the size of $s_{f\circ\varphi^{-1}}\in \mathcal S(\mathsf{L,W,S,B})$, it suffices to note that it is a sum of compositions of $\phi_{\mathrm{mult}}$ with a parallelisation of $\phi_m$ with $p_m$. By Lemma~\ref{lem:compsition_of_neural_networks} (composition of neural networks), Lemma~\ref{lem:parallelisation_of_neural_networks} (parallelisation of neural networks), Lemma~\ref{lem:sum_of_neural_networks} (sum of neural networks), Lemma~\ref{lem:multiplication_function} (approximation of multiplication) and the size of $\phi_m$ determined previously, we get
    \begin{align*}
        \mathsf{L} \lesssim \log N\log\log N, && W\lesssim MN\log N, && S\lesssim MN(\log N)^2, && B\lesssim N^{1/(d-1)}\lor \varepsilon^{-1},
    \end{align*}
    with $M\lesssim \log N$.
\end{proof}

Now, all that is left is to approximate the coordinate maps $\varphi_\pm$. 

\begin{lemma}\label{lem:coordinate_map_approximation}
    For any $\gamma>0$, there exist neural networks $s_{\varphi_+}\in \mathcal S(\mathsf{L,W,S,B})$ and $s_{\varphi_-}\in \mathcal S(\mathsf{L,W,S,B})$ with
    \begin{align*}
        \mathsf L\lesssim \log^2 \gamma^{-1}, \qquad \mathsf W=\mathrm{const},  \qquad \mathsf S\lesssim \log^2 \gamma^{-1}, \qquad \mathsf B\lesssim \gamma^{-1},
    \end{align*}
    such that
    \begin{align*}
        &\lVert \varphi_\pm-s_{\varphi_\pm}\rVert_{L^\infty(B_\pm\setminus B_R)}\leq \gamma,
    \end{align*}
    with
    \begin{equation*}
        B_\pm\coloneqq \{ z\in B_{1-\varepsilon}: \pm z_d \geq -1/2\}.
    \end{equation*}
    Moreover, $R\leq s_{\varphi_\pm}(z)_1 \leq 1-\varepsilon$ for $z\in B_\pm\setminus B_R$.
\end{lemma}

\begin{proof}
    The construction focuses mainly on the approximation of the radial coordinate $r$, since the angular coordinates are simply quotients of linear functions and the reciprocal function on $[1,2]$ can be efficiently approximated using Lemma~\ref{lem:reciprocal_function}. For the latter, we can define
    \begin{align*}
        s_{\varphi_{\pm}}(x)_i \coloneqq \phi_{\mathrm{mult}}(\mp x_i,\phi_{\mathrm{rec}}(1\pm x_d))
    \end{align*}
    for $i=2,\ldots,d$, with $\phi_{\mathrm{mult}}$ given by Lemma~\ref{lem:multiplication_function} and with size
    \begin{align*}
        \mathsf L\lesssim \log \gamma^{-1} + \log \gamma^{-1}\,\log\log \gamma^{-1}, && \mathsf W = 1, && S\lesssim \log \gamma^{-1} + \log \gamma^{-1}\,\log\log \gamma^{-1}, && B = \mathrm{const},
    \end{align*}
    for a maximal error of size $\lvert \varphi_\pm(x)_i - s_{\varphi_\pm}(x)_i \rvert \leq \gamma$ for $x \in B_\pm \setminus B_R$.

    For the approximation of the radial coordinate, we need to approximate $y\mapsto y^2$ and $y\mapsto \sqrt y$. The former is a special case of Lemma \ref{lem:multiplication_function}: there exists a neural network $\phi_{\square}$ of size $\mathsf L \lesssim \log \gamma^{-1}$, $\mathsf W=1$, $\mathsf S \lesssim \log \gamma^{-1}$ and $\mathsf B=\mathrm{const}$, such that
    \begin{equation*}
        \lvert \phi_{\square}(x) - x^2\rvert \leq \gamma.
    \end{equation*}
    For the square root, we note that on $[R,1-\varepsilon]$ it has the Taylor expansion
    \begin{equation*}
        \sqrt{y} = 1-\sum_{k=1}^\infty \underbrace{\frac{(2k-1)!!}{2^kk!}}_{\eqqcolon a_k} (1-y)^k\eqcolon f(y),
    \end{equation*}
    where
    \begin{equation*}
        (2k-1)!!\coloneqq (2k-1)(2k-3)\cdots3\times 1.
    \end{equation*}
    We now follow a two-step procedure that first approximates the  square root function by its partial Taylor sums $f_K$, $K\in\N$ and in a second step approximate the partial sums by a neural network, i.e., 
    \begin{align}\label{sqrt_approx_decomposition}
        \lvert \sqrt y-s_{\varphi_{\pm}}(y)\rvert \leq  \lvert f(y)-f_K(y)\rvert + \lvert f_K(y)-s_{\varphi_{\pm}}(y)\rvert
    \end{align}
    
    The convergence rate of the series of partial sums is evaluated straightforwardly:
    \begin{align*}
        \lvert f(y) - f_K(y)\rvert &= \sum_{k=K+1}^\infty a_k (1-y)^k \\
        &=(1-y)^K\sum_{k=1}^\infty a_{k+K}(1-y)^k \\
        &\leq (1-R)^K \sum_{k=1}^\infty a_k(1-R)^k \\
        &=(1-R)^K (1-\sqrt{R}) \leq (1-R)^K.
    \end{align*}
    Thus, it suffices to choose $K=\log \gamma/\log(1-R)$ to bound the first term in \eqref{sqrt_approx_decomposition} by $\gamma$.

    For the second term, we can invoke Lemma~\ref{lem:polynomial_function_approximation} to get a neural network $\phi_{\mathrm{sqrt}}\in\mathcal S(K(\log \gamma^{-1}+\log K),9,K(\log \gamma^{-1}+\log K),1)$ such that
    \begin{equation*}
        \lvert \phi_{\mathrm{sqrt}}(y)-f_K(y)\rvert \leq \gamma.
    \end{equation*}
    The network approximating the radial coordinate can then be defined as
    \begin{equation*}
        s_{\varphi_{\pm}}(x)_1\coloneqq \phi_{\mathrm{sqrt}}\left( 1\land \sum_{i=1}^d\phi_{\square}(x_i)\lor 0 \right),
    \end{equation*}
    whose size, according to Lemma~\ref{lem:sum_of_neural_networks}, can be chosen of order
    \begin{align*}
        &\mathsf L \lesssim K (\log \gamma^{-1} + \log K) + 2 + \log \gamma^{-1}\lesssim \log^2 \gamma^{-1}, && \mathsf W\leq 9d, \\
        &\mathsf S \lesssim K (\log \gamma^{-1} + \log K) + 2 + d + d\,\log \gamma^{-1} \lesssim  \log^2 \gamma^{-1}, && \mathsf B = \mathrm{const}. 
    \end{align*}
    
    The approximation error of $s_\varphi$ is then finally  evaluated by
    \begin{align*}
        \lvert r(x)-s_{\varphi_{\pm}}(x)_1\rvert &\leq \lVert \sqrt\cdot - \phi_{\mathrm{sqrt}}\rVert_{L^\infty([R,1-\varepsilon])} + \sum_{i=1}^d\lvert \phi_{\square}(x_i)-x_i^2\rvert \\
        &\leq \gamma + d \gamma
    \end{align*}
    where in the first inequality we used the approximation error of $\phi_{\mathrm{sqrt}}$ and in the second one the approximation error for the square function. Setting $(d+1)\gamma\mapsto\gamma$ yields the claim.
\end{proof}

Finally, the neural networks $s_{h_N\circ\varphi^{-1}_\pm}$ and $s_{\nabla h_N\circ\varphi^{-1}_\pm}$ can be concatenated with $s_{\varphi_\pm}$ to approximate the score on the cartesian space. This is the content of Proposition \ref{prop:approx_nn_trunc} and we are now in a position to prove it.

\begin{proof}[Proof of Proposition \ref{prop:approx_nn_trunc}]
    Some parts of the proof follow the same reasoning for both $h_N$ and $\partial_i h_N$ for any $i=1,\dots,d$, which we commonly denote as $f$. We define the neural networks as
    \begin{align*}
        s_f(z)\coloneqq (\phi^{\mathrm{cap}}(s_{\varphi_+}(z)_1)\land\phi_{\mathrm{mult}}(p_+(z_d), s_{f\circ\varphi^{-1}_+}(s_{\varphi_+}(z))) + \phi_{\mathrm{mult}}(p_-(z_d), s_{f\circ\varphi^{-1}_-}(s_{\varphi_-}(z))))\lor (-\phi^{\mathrm{cap}}(s_{\varphi_+}(z)_1)),
    \end{align*}
    with $\phi^{\mathrm{cap}}$ as in Lemma~\ref{lem:growth_neural_network}, $\phi_{\mathrm{mult}}$ as in Lemma~\ref{lem:multiplication_function}, $s_{\varphi_\pm}$ as in Lemma~\ref{lem:coordinate_map_approximation} with $\gamma=\varepsilon^4$, and
    \begin{equation*}
        p_\pm(z_d)\coloneqq 1 \land (\pm z_d + 1/2) \lor 0.
    \end{equation*}
     Bounding the values of $s_f$ with $\phi^{\mathrm{cap}}$ ensures that
    \begin{align*}
        \lvert s_f(z)\rvert\leq 6\frac{d+2}{1-s_{\varphi_+}(z)_1}&\leq 6\frac{d+2}{1-\lVert z\rVert} + 6(d+2)\left\lvert\frac{1}{1-s_{\varphi_+}(z)_1} -\frac{1}{1-\lVert z\rVert} \right\rvert \\
        &= 6\frac{d+2}{1-\lVert z\rVert} + 6\frac{\lvert \lVert z\rVert - s_{\varphi_+}(z)_1\rvert}{(1-s_{\varphi_+}(z)_1)(1-\lVert z\rVert)} \\
        &\leq \frac{6(d+2)(1+\gamma\varepsilon^{-1})}{1-\lVert z\rVert} \\
        &\leq \frac{12(d+2)}{1-\lVert z\rVert},
    \end{align*}
    showing that $s_f\in\mathcal S(\mathsf{L,W,S,B})$ with hyperparameters determined at the end of the proof. Concerning the approximation error, since $\lvert f(z)\rvert\leq (d+2)/(1-\lVert z\rVert)$ and for $\varepsilon$ small enough
    \begin{equation*}
        \phi^{\mathrm{cap}}(s_{\varphi_+}(z)_1)\geq \frac{2(d+2)}{1-s_{\varphi_+}(z)_1}\geq \frac{2(d+2)(1-\varepsilon^{-1}\gamma)}{1-\lVert z\rVert}\geq \frac{d+2}{1-\lVert z\rVert},
    \end{equation*}
    cutting off a neural network with $\phi^{\mathrm{cap}}$ does not alter the approximation of $f$, i.e.,
    \begin{align*}
        \lvert f(z)-s_f(z)\rvert &\leq \lvert f(z)-\phi_{\mathrm{mult}}(p_+(z_d), s_{f\circ\varphi^{-1}_+}(s_{\varphi_+}(z))) - \phi_{\mathrm{mult}}(p_-(z_d), s_{f\circ\varphi^{-1}_-}(s_{\varphi_-}(z)))\rvert \\
        &\leq 2\times 2^{-l_1}\lVert s_{f\circ\varphi^{-1}_\pm}\circ s_{\varphi_\pm}\rVert_{L^\infty(B_{1-\varepsilon})} + \lvert f(z) - p_+(z_d) s_{f\circ\varphi^{-1}_+}(s_{\varphi_+}(z)) - p_-(z_d) s_{f\circ\varphi^{-1}_-}(s_{\varphi_-}(z))\rvert \\
        &\leq 2^{-l_1+1}\varepsilon^{-1} + \underbrace{p_+(z_d)}_{\leq \one_{B_+}} \lvert f(z) - s_{f\circ\varphi^{-1}_+}(s_{\varphi_+}(z))\rvert + \underbrace{p_-(z_d)}_{\leq \one_{B_-}} \lvert f(z) - s_{f\circ\varphi^{-1}_-}(s_{\varphi_-}(z))\rvert.
    \end{align*}
    By choosing $l_1\coloneqq \lceil \log\varepsilon^{-1} + \alpha/(d-1)\log N\rceil$, the first term achieves the desired convergence rate. Thus, without loss of generality, it suffices to check the statement for $s_f\coloneqq s_{f\circ\varphi^{-1}_\pm}\circ s_{\varphi_\pm}$ on the sliced balls $B_\pm$. We decompose the errors as
    \begin{align*}
        \lVert s_{\nabla h_N}(h_N-s_{h_N})\sqrt{G_{1-\varepsilon}(z_0,\cdot)}\rVert_{L^2(B_\pm\setminus B_R)}&\leq \lVert s_{\nabla h_N}(h_N-h_N\circ \varphi_\pm^{-1}\circ s_{\varphi_\pm})\sqrt{G_{1-\varepsilon}(z_0,\cdot)}\rVert_{L^2(B_\pm\setminus B_R)} \\
        &+ \lVert s_{\nabla h_N}(h_N\circ \varphi_\pm^{-1}\circ s_{\varphi_\pm}-s_{h_N\circ\varphi_\pm^{-1}}\circ s_{\varphi_\pm})\sqrt{G_{1-\varepsilon}(z_0,\cdot)}\rVert_{L^2(B_\pm\setminus B_R)},
    \end{align*}
    and
    \begin{align*}
        \lVert (\nabla h_N-s_{\nabla h_N})\sqrt{G_{1-\varepsilon}(z_0,\cdot)}\rVert_{L^2(B_\pm\setminus B_R)}&\leq \lVert (\nabla h_N-\nabla h_N\circ \varphi_\pm^{-1}\circ s_{\varphi_\pm})\sqrt{G_{1-\varepsilon}(z_0,\cdot)}\rVert_{L^2(B_\pm\setminus B_R)} \\
        &+ \lVert (\nabla h_N\circ \varphi_\pm^{-1}\circ s_{\varphi_\pm}-s_{\nabla h_N\circ\varphi_\pm^{-1}}\circ s_{\varphi_\pm})\sqrt{G_{1-\varepsilon}(z_0,\cdot)}\rVert_{L^2(B_\pm\setminus B_R)}.
    \end{align*}
    For the first terms, we use that $h_N\circ\varphi_\pm^{-1}$ and $\nabla h_N\circ\varphi_\pm^{-1}$ are differentiable on $[R,1-\varepsilon]\times B^{(d-1)}_2$, and thus locally Lipschitz, to get
    \begin{align*}
        \lVert (\nabla h_N-\nabla h_N\circ \varphi_\pm^{-1}\circ s_{\varphi_\pm})\sqrt{G_{1-\varepsilon}(z_0,\cdot)}\rVert_{L^2(B_\pm\setminus B_R)} &\leq \lVert\nabla^2 h_N\circ\varphi_\pm^{-1}\rVert_{L^\infty([R,1-\varepsilon]\times B^{(d-1)}_2)}  \\
        &\qquad\times\lVert (\varphi_\pm - s_{\varphi_\pm}) \sqrt{G_{1-\varepsilon}(z_0,\cdot)}\rVert_{L^2(B_\pm\setminus B_R)} \\
        &\lesssim \varepsilon^{-2}\gamma,
    \end{align*}
    and
    \begin{align*}
        \lVert s_{\nabla h_N}(h_N- h_N\circ \varphi_\pm^{-1}\circ s_{\varphi_\pm})\sqrt{G_{1-\varepsilon}(z_0,\cdot)}\rVert_{L^2(B_\pm\setminus B_R)}&\leq \lVert s_{\nabla h_N} \nabla h_N\circ\varphi_\pm^{-1}\rVert_{L^\infty([R,1-\varepsilon]\times B^{(d-1)}_2)} \\
        &\qquad\times\lVert (\varphi_\pm - s_{\varphi_\pm}) \sqrt{G_{1-\varepsilon}(z_0,\cdot)}\rVert_{L^2(B_\pm\setminus B_R)} \\
        &\lesssim \varepsilon^{-2}\gamma,
    \end{align*}
    using Lemma~\ref{lem:bounds_grad_Hess_of_h_N} for the bounds on $\nabla h_N$ and $\nabla^2 h_N$ and Lemma~\ref{lem:coordinate_map_approximation} in both cases.
    For the second term of the approximation of $h_N$, we directly get from Proposition~\ref{prop:neural_network_approximation} and Lemma~\ref{lem:expectation_one_minus_norm_BM}
    \begin{align*}
        &\lVert s_{\nabla h_N}(h\circ \varphi_\pm^{-1}\circ s_{\varphi_\pm}-s_{h_N\circ\varphi_\pm^{-1}}\circ s_{\varphi_\pm})\sqrt{G_{1-\varepsilon}(z_0,\cdot)}\rVert_{L^2(B_\pm\setminus B_R)}^2 \\
        &\leq \int_{B_\pm\setminus B_R} \lVert s_{\nabla h_N}\rVert_{L^\infty} \lVert h(\varphi_\pm^{-1}(s_{\varphi_\pm}(z)))-s_{h_N\circ\varphi_\pm^{-1}}(s_{\varphi_\pm}(z)) \rVert^2 G_{1-\varepsilon}(z_0,z)\diff z \\
        &\lesssim N^{-2\alpha/(d-1)}\log N \int_{B_\pm\setminus B_R} \frac{1}{(1-\lVert z\rVert)^2} G_{1-\varepsilon}(z_0,z)\diff z \\
        &\lesssim N^{-2\alpha/(d-1)}\log N\log \varepsilon^{-1}.
    \end{align*}
    For the second term of the approximation of $\nabla h_N$, we first apply the same idea:
    \begin{align*}
        &\lVert (\nabla h\circ \varphi_\pm^{-1}\circ s_{\varphi_\pm}-s_{\nabla h_N\circ\varphi_\pm^{-1}}\circ s_{\varphi_\pm})\sqrt{G_{1-\varepsilon}(z_0,\cdot)}\rVert_{L^2(B_\pm\setminus B_R)}^2 \\
        &= \int_{B_\pm\setminus B_R} \lVert \nabla h(\varphi_\pm^{-1}(s_{\varphi_\pm}(z)))-s_{\nabla h_N\circ\varphi_\pm^{-1}}(s_{\varphi_\pm}(z)) \rVert^2 G_{1-\varepsilon}(z_0,z)\diff z \\
        &\lesssim N^{-2\alpha/(d-1)}\log N \int_{B_\pm\setminus B_R} \frac{1}{(1-s_{\varphi_\pm}(z)_1)^2} G_{1-\varepsilon}(z_0,z)\diff z.
    \end{align*}
    Since $s_{\varphi_\pm}(z)_1$ is an approximation of $\lVert z\rVert$, the remaining integral can be controled again by $\log\varepsilon^{-1}$ plus the approximation error of the network:
    \begin{align*}
        &\int_{B_\pm\setminus B_R} \frac{1}{(1-s_{\varphi_\pm}(z)_1)^2} G_{1-\varepsilon}(z_0,z)\diff z \\
        &= \int_R^{1-\varepsilon}\int_{\partial B} \frac{1}{(1-s_{\varphi_\pm}(\varphi_\pm^{-1}(r,x))_1)^2} G_{1-\varepsilon}(z_0,(r,x))\diff\sigma(x) r^{d-1}\diff r\\
        &\leq \int_R^{1-\varepsilon}\int_{\partial B} \frac{1}{(1-r)^2} G_{1-\varepsilon}(z_0,(r,x))\diff\sigma(x) r^{d-1}\diff r \\
        &\quad+ \int_R^{1-\varepsilon}\int_{\partial B} \Bigg\lvert \frac{1}{(1-r)^2} - \frac{1}{(1-s_{\varphi_\pm}(\varphi_\pm^{-1}(r,x))_1)^2}\Bigg\rvert G_{1-\varepsilon}(z_0,(r,x))\diff\sigma(x) r^{d-1}\diff r \\
        &\lesssim \log \varepsilon^{-1} + \varepsilon^{-4} \int_R^{1-\varepsilon}\int_{\partial B} \lvert (1-s_{\varphi_\pm}(\varphi_\pm^{-1}(r,x))_1)^2 - (1-r)^2 \rvert \diff\sigma(x) r^{d-1}\diff r \\
        &= \log \varepsilon^{-1} + \varepsilon^{-4} \int_R^{1-\varepsilon}\int_{\partial B} \lvert s_{\varphi_\pm}(\varphi_\pm^{-1}(r,x))_1 - r\rvert\, \underbrace{\lvert s_{\varphi_\pm}(\varphi_\pm^{-1}(r,x))_1 + r +2 \rvert}_{\leq 4} \diff\sigma(x) r^{d-1}\diff r \\
        &\lesssim \log \varepsilon^{-1} + \varepsilon^{-4} \int_{B_\pm\setminus B_R} \lvert s_{\varphi_\pm}(z)_1 - \lVert z\rVert \rvert \diff z \\
        &\lesssim \log \varepsilon^{-1} + \varepsilon^{-4} \gamma,
    \end{align*}
    using Lemma~\ref{lem:coordinate_map_approximation} in the last step once again. Choosing $\gamma=\varepsilon^4$ yields the desired approximation error.

    The network size evaluates as follows: $s_f$ is a sum of compositions of $\phi_{\mathrm{mult}}$ with $p_\pm$, $s_{f\circ\varphi^{-1}_\pm}$ and $s_{\varphi_\pm}$. For $\phi_{\mathrm{mult}}$, we choose $l_1=\lceil \alpha/(d-1)\log_2 N \rceil+1$, thus the network belongs to the class $\mathrm{NN}(\log N, 1, \log N, 1)$ up to constants. Combining the sizes of $s_{f\circ\varphi^{-1}_\pm}$ and $s_{\varphi_\pm}$ from Proposition~\ref{prop:neural_network_approximation} and Lemma~\ref{lem:coordinate_map_approximation} using Lemma~\ref{lem:compsition_of_neural_networks} (composition of neural networks), we get that $s_{f\circ\varphi^{-1}_\pm}\circ s_{\varphi_\pm}\in\mathrm{NN}(\mathsf{L_\pm,W_\pm,S_\pm,B_\pm})$ with
    \begin{align*}
        &\mathsf L_\pm\lesssim \log N\log\log N + \log^2\varepsilon^{-1}, && \mathsf W_\pm\lesssim (N\log^2 N)\lor 1 = N\log^2 N, \\
        &\mathsf S_\pm\lesssim N\log^3 N + \log^2\varepsilon^{-1}, && \mathsf B_\pm\lesssim N^{1/(d-1)}\lor\varepsilon^{-1}\lor\varepsilon^{-4} = N^{1/(d-1)}\lor\varepsilon^{-4}.
    \end{align*}
    Parallelizing $s_{f\circ\varphi^{-1}_\pm}\circ s_{\varphi_\pm}$ with $p_\pm$ does not alter the asymptotic size of the network, since the latter has constant size, thus, using Lemma~\ref{lem:compsition_of_neural_networks} once again and the size of $\phi_{\mathrm{mult}}$, the total size of $\phi_{\mathrm{mult}}\circ(p_\pm,s_{f\circ\varphi^{-1}_\pm}\circ s_{\varphi_\pm})$ evaluates as
    \begin{align*}
        \mathsf L &\lesssim \log N + \log N\log\log N + \log^2\varepsilon^{-1}\leq 2 \log N\log\log N+\log^2\varepsilon^{-1}, \\
        \mathsf W&\lesssim (N\log^2 N)\lor 1 = N\log^2 N, \\
        \mathsf S&\lesssim \log N + N\log^3 N + \log^2\varepsilon^{-1}\leq 2N\log^3 N + \log^2\varepsilon^{-1},\\
        \mathsf B&\lesssim N^{1/(d-1)}\lor \varepsilon^{-4}\lor 1 = N^{1/(d-1)}\lor \varepsilon^{-4}.
    \end{align*}
    Similarly, the composition and parallelization with the remaining networks $\phi^{\mathrm{cap}}\circ s_{\varphi_+}(\cdot)_1$, max and min does not affect the size of $s_f$, since their sizes are dominated by that of $\phi_{\mathrm{mult}}\circ(p_\pm,s_{f\circ\varphi^{-1}_\pm}\circ s_{\varphi_\pm})$.
\end{proof}

Finally, Theorem~\ref{thm:NN_approximation} can be proven.

\begin{proof}[Proof of Theorem~\ref{thm:NN_approximation}]
    Let $R,\tilde R$ be fixed radii such that $\lVert z_0\rVert<R<\tilde R<1-\varepsilon$. We define the overall network $s$ as
    \begin{equation*}
        s(z) \coloneqq \phi_{\mathrm{mult}}(\overline{p}(s_{r^2}(z)),\overline{s}(z))+\phi_{\mathrm{mult}}(\underline{p}(s_{r^2}(z)),\underline{s}(z)),
    \end{equation*}
    with $\underline{s}$ given by Proposition~\ref{prop:approximation_BR}, $\overline{s}$ defined as in \eqref{overline_s}, $s_{r^2}$ the approximation of the squared radial coordinate, $\phi_{\mathrm{mult}}$ as in Lemma~\ref{lem:multiplication_function}, and with the neural networks 
    \begin{align*}
        \overline{p}(r)\coloneqq 0\lor\left( \frac{r-R^2}{\tilde R^2-R^2} \right)\land 1, &&  \underline{p}(r) \coloneqq 0\lor\left( \frac{\tilde R^2-r}{\tilde R^2-R^2} \right)\land 1,
    \end{align*}
    which act as a partition of unity.  The network $s_{r^2}$ is constructed as
    \begin{align*}
        s_{r^2}(z)\coloneqq \sum_{i=1}^d s_{\square}(z_i), \quad z\in[-1,1]^d,
    \end{align*}
    with $s_{\square}$ the neural network approximation of $x\mapsto x^2$ given in Lemma~\ref{lem:multiplication_function}. According to the same lemma and Lemma~\ref{lem:sum_of_neural_networks}, $s_{r^2}\in\mathrm{NN}(m+1,d,dm,C)$ for some $C>0$ with
    \begin{equation*}
        \lvert s_{r^2}(z)-r^2\rvert\leq dC 2^{-m}.
    \end{equation*}
    We  choose $m\coloneqq \lceil \log(dC\varepsilon^{-1})\rceil$. Then, by Lemma~\ref{lem:multiplication_function} and Proposition~\ref{prop:approximation_error_decomposition} (with the separation at $\tilde R$ for $\underline{s}$),
    \begin{align}\label{final_approximation_error}
        &\inf_{s\in\mathcal S} \E^{z_0}\left[\int_0^{\tau^h_{1-\varepsilon}} \lVert s(Z_t^h)-\nabla\log h(Z_t^h) \rVert^2\diff t\right] \nonumber\\
        &\lesssim 2^{-2l_1}\lVert\overline{s}\rVert_{L^\infty(B_{1-\varepsilon})}^2 + 2^{-2l_1}\lVert\underline{s}\rVert_{L^\infty(B_{1-\varepsilon})}^2 +\Big\lVert \underbrace{\overline{p}\circ s_{r^2}}_{\leq \one_{B_{1-2\varepsilon}\setminus B_{R-\varepsilon}}} (\overline{s}- \nabla\log h)\sqrt{G_{1-\varepsilon}(z_0,\cdot)} \Big\rVert_{L^2(B_{1-\varepsilon})}^2 \nonumber\\
        &\qquad+ \Big\lVert \underbrace{\underline{p}\circ s_{r^2}}_{\leq \one_{B_{\tilde R+\varepsilon}}} (\underline{s}- \nabla\log h)\sqrt{G_{1-\varepsilon}(z_0,\cdot)} \Big\rVert_{L^2(B_{1-\varepsilon})}^2 \nonumber\\
        &\lesssim 2^{-2l_1}\varepsilon^{-2} + \left\lVert (\overline{s}- \nabla\log h)\sqrt{G_{1-\varepsilon}(z_0,\cdot)} \right\rVert_{L^2(B_{1-2\varepsilon}\setminus B_{R+\varepsilon})}^2 + \lVert \underline{s}-\nabla \log h\rVert_{L^\infty(B_{\tilde R+\varepsilon})}^2
    \end{align}
    For $\varepsilon$ small enough, it still holds $R-\varepsilon>\lVert z_0\rVert$ and $\tilde R+\varepsilon<1-2\varepsilon$, whence we can replace $R-\varepsilon\to R$, $\tilde R+\varepsilon\to\tilde R$ and $2\varepsilon\to\varepsilon$ without loss of generality in the following.
    
    The last term in \eqref{final_approximation_error} is directly bounded by $N^{-2\alpha/(d-1)}$ using Proposítion~\ref{prop:approximation_BR}, while the first term is bounded by the same rate with the choice $l_1\coloneqq \lceil -\log\varepsilon + \alpha/(d-1)\log N\rceil$. The first term can be decomposed according to \eqref{L2_error_decomposition_spherical} and the subsequent discussion, yielding
    \begin{align*}
        \left\lVert (\overline{s}- \nabla\log h)\sqrt{G_{1-\varepsilon}(z_0,\cdot)} \right\rVert_{L^2(B_{1-\varepsilon}\setminus B_R)}^2 &\lesssim 2^{-2l_1}\varepsilon^{-2}(2^{-l_2}+\pi_{\mathrm{min}}^{-1})^2 + 2^{-2l_2}\log^2 \varepsilon^{-1} \\
        &\quad+ \pi_{\mathrm{min}}^{-1}\lVert\nabla h - \nabla h_N \rVert_{L^2(B_{1-\varepsilon})} + \pi_{\mathrm{min}}^{-2} \lVert\nabla h_N (h-h_N) \sqrt{G_{1-\varepsilon}(z_0,\cdot)}\rVert_{L^2(B_{1-\varepsilon}\setminus B_R)} \\
        &\quad+ \pi_{\mathrm{min}}^{-1}\lVert (\nabla h_N - s_{\nabla h_N})\sqrt{G_{1-\varepsilon}(z_0,\cdot)} \rVert_{L^2(B_{1-\varepsilon}\setminus B_R)} \\
        &\quad+ \pi_{\mathrm{min}}^{-2}\lVert s_{\nabla h_N}(h_N-s_{h_N}) \sqrt{G_{1-\varepsilon}(z_0,\cdot)}\rVert_{L^2(B_{1-\varepsilon}\setminus B_R)}.
    \end{align*}
    The previous choice for $l_1$ and $l_2\coloneqq \lceil\log\log\varepsilon^{-1}+ \alpha/(d-1)\log N\rceil$ yield the rate $\mathcal O(N^{-2\alpha/(d-1)})$ for the first two terms. The bounds on the four remaining terms are the result of Lemma~\ref{lem:approx_score_trunc} and Proposition~\ref{prop:approx_nn_trunc}, yielding the final claim
    \begin{equation*}
        \begin{split}
            \E^{z_0}\Big[\int_0^{\tau^h_{1-\varepsilon}} \lVert s(Z_t^h)-\nabla\log h(Z_t^h) \rVert^2\diff t\Big]&\,\lesssim N^{-2\alpha/(d-1)} + N^{-2\alpha/(d-1)} + N^{-2\alpha/(d-1)} \\
            &\quad +\log\varepsilon^{-1}(1-\varepsilon)^{2N+d-1} N^{-2\alpha} + (1-\varepsilon)^{2N+d} N^{-2\alpha+1} \\
            &\quad+ N^{-2\alpha/(d-1)}\log N\log\varepsilon^{-1} + N^{-2\alpha/(d-1)} \\
            &\lesssim \log\varepsilon^{-1}(1-\varepsilon)^{2N+d-1} N^{-2\alpha} + (1-\varepsilon)^{2N+d} N^{-2\alpha+1} \\
            &\quad+ N^{-2\alpha/(d-1)}\log N\log\varepsilon^{-1} \\
             &\lesssim N^{-2\alpha/(d-1)}\log N\log\varepsilon^{-1}.
        \end{split} 
    \end{equation*}
    For the last step we used that it holds
    \begin{equation*}
        N^{-2\alpha} = N^{-2\alpha/(d-1)}
    \end{equation*}
    and for the exponent of $N$ of the second summand
    \begin{align*}
        &\frac{2\alpha-1}{\alpha/(d-1)} = 2(d-1)-\frac{d-1}{\alpha} > 2d-4\geq 2, \\
        \implies& N^{-2\alpha+1}\leq N^{-2\alpha/(d-1)},
    \end{align*}
    due to $\alpha>(d-1)/2$ and $d\geq 3$. 

    The size of $s$ is straightforwardly evaluated with Proposition~\ref{prop:approx_nn_trunc}, the sizes of $s_{r^2}$, $\overline{p}$, $\underline{p}$ and Lemma~\ref{lem:multiplication_function}. The sizes of $\overline{p}\circ s_{r^2}$ and $\underline{p}\circ s_{r^2}$ evaluate as
    \begin{align*}
        &\mathsf{L}\lesssim \log\varepsilon^{-1}, && \mathsf{W}\lesssim 1, \\
        &\mathsf{S}\lesssim \log\varepsilon^{-1}, && \mathsf{B}\lesssim 1.
    \end{align*}
    Thus, the size of the parallelization $(\overline{p}\circ s_{r^2},\overline{s})$ is dominated by the size of $\overline{s}$. The size of $\underline{s}$ is given in Proposition~\ref{prop:approximation_BR}, such that the parallelization $(\underline{p}\circ s_{r^2},\underline{s})$ has size
    \begin{align*}
        &\mathsf{L}\lesssim \log N + \log\varepsilon^{-1}, && \mathsf{W}\lesssim N, \\
        &\mathsf{S}\lesssim N\log N + \log\varepsilon^{-1}, && \mathsf{B}\lesssim \mathrm{Poly}(N).
    \end{align*}
    By the choice of $l_1$, the size of $\phi_{\mathrm{mult}}$ is of the same order as the ones of $\overline{p}\circ s_{r^2}$ and $\underline{p}\circ s_{r^2}$, which are dominated by the size of $(\overline{p}\circ s_{r^2},\overline{s})$ and $(\underline{p}\circ s_{r^2},\underline{s})$. Thus, $\phi_{\mathrm{mult}}(\overline{p}\circ s_{r^2},\overline{s})$ has the same asymptotic size as $\overline{s}$ and $\phi_{\mathrm{mult}}(\underline{p}\circ s_{r^2},\underline{s})$ has the same size as $(\underline{p}\circ s_{r^2},\underline{s})$ derived above.  Finally, the sum of $\phi_{\mathrm{mult}}(\overline{p}\circ s_{r^2},\overline{s})$ and $\phi_{\mathrm{mult}}(\underline{p}\circ s_{r^2},\underline{s})$ giving the total network size evaluates as
    \begin{align*}
        \mathsf{L}&\lesssim (\log N\log\log N+\log^2\varepsilon^{-1})\lor (\log N + \log\varepsilon^{-1}) && \mathsf{W}\lesssim N\log^2 N + N \\
        &\lesssim \log N\log\log N+\log^2\varepsilon^{-1}, && \quad\lesssim N\log^2 N, \\
        \mathsf{S}&\lesssim N\log^3 N +\log^2\varepsilon^{-1} + N\log N + \log\varepsilon^{-1} && \mathsf{B}\lesssim (N^{1/(d-1)}\lor\varepsilon^{-4})\lor\mathrm{Poly}(N) \\
        &\lesssim N\log^3 N +\log^2\varepsilon^{-1}, && \quad\lesssim \mathrm{Poly}(N)\lor\varepsilon^{-4}.
    \end{align*}
\end{proof}

\section{Auxiliary lemmas}

\begin{lemma}\label{lem:bound_regular_Sobolev_norm}
    Let $\varphi_{\partial B,\pm}$ denote the stereographic coordinate map for the $(d-1)$-dimensional sphere defined in Section~\ref{app:approx} (either northern or southern hemisphere) and $k\in\N_0$. Then, for $u\in H^k(\partial B)$ and any $R>0$
    \begin{equation*}
        \lVert u\circ\varphi_{\partial B,\pm}^{-1} \rVert_{H^k(B^{(d-1)}_R)}\lesssim \lVert u\rVert_{H^k(\partial B)},
    \end{equation*}
    where the hidden constant diverges as $R\to\infty$.
\end{lemma}

\begin{proof}
    The main idea relies essentially on \autocite[Lemma~2.3.1]{Metsch_2023}, which uses the elliptic regularity estimate of solutions of the Dirichlet/Neumann problem on a manifold. Let $\varphi_{\partial B} \in \{\varphi_{\partial B,+}, \varphi_{\partial B,-}\}$ and denote by $g=(g^{ij})_{i,j=1,\dots,d-1}$ the metric tensor induced by $\varphi_{\partial B}$. Its determinant is given by (\cite{lee97} eq. (3.10))
    \begin{equation*}
        \lvert g(\theta)\rvert \coloneqq\lvert \det g(\theta)\rvert = \left(\frac{2}{1+\lVert \theta\rVert^2}\right)^{2d},\quad \theta\in \R^{d-1}.
    \end{equation*}
    We note that $\lvert g\rvert$ is that is bounded from above by $(2/(1-R))^d$) and from below by 1. Moreover, the (weak) Laplace--Beltrami operator can be expressed in coordinate space as
    \begin{align*}
        \Delta_{\partial B} = \frac{1}{\sqrt{\lvert g\rvert}} \sum_{i,j=1}^d\partial_j (g^{ij} \sqrt{\lvert g\rvert} \partial_i).
    \end{align*}

    Now, for any $v\in W_0^{1,2}(B^{(d-1)}_R)$ and some open set $U\supset \overline{B^{(d-1)}_R}$, we have the integration by parts formula
    \begin{align*}
        \int_U (-\Delta_{\partial B}) (u\circ\varphi^{-1}_{\partial B}) v \diff\lambda &= -\sum_{i,j=1}^d\int_U \partial_j (g^{ij}\sqrt{\lvert g\rvert} \partial_i (u\circ\varphi^{-1}_{\partial B})) \frac{v}{\sqrt{\lvert g\rvert}} \diff\lambda \\
        &= \sum_{i,j=1}^d \int_U g^{ij}\partial_i (u\circ\varphi^{-1}_{\partial B}) \partial_jv - g^{ij} \frac{\sqrt{\lvert g\rvert}}{2\sqrt{\lvert g\rvert^3}} \partial_j \lvert g\rvert \partial_i(u\circ\varphi^{-1}_{\partial B}) v \diff\lambda \\
        &= \sum_{i,j=1}^d\int_U g^{ij}\partial_i (u\circ\varphi^{-1}_{\partial B}) \partial_j v - g^{ij} \underbrace{\frac{1}{2\lvert g\rvert} \partial_j \lvert g\rvert}_{\eqqcolon \Gamma_{ij}} \partial_i(u\circ\varphi^{-1}_{\partial B}) v \diff\lambda,
    \end{align*}
    where the $\Gamma_{ij}$ are the so-called \emph{Christoffel symbols}. This is now a weak formulation for the PDE
    \begin{align*}
        \sum_{i,j=1}^d \partial_j(g^{ij}\partial_i - g^{ij}\Gamma_{ij}\partial_i)(u\circ\varphi^{-1}_{\partial B})=(-\Delta_{\partial B})u\circ\varphi^{-1}_{\partial B},
    \end{align*}
    with some boundary condition that is not important here. Since the metric tensor is always positive definite by definition, the differential operator in the above PDE is elliptic and we can apply the elliptic interior regularity estimate \autocite[Theorem~8.20]{Gilbarg_Trudinger2001}
    \begin{equation}\label{elliptic_regularity}
        \lVert u\circ\varphi^{-1}_{\partial B}\rVert_{H^k(B^{(d-1)}_R)} \lesssim \lVert (-\Delta_{\partial B}) u\circ\varphi^{-1}_{\partial B}\rVert_{H^{k-2}(B^{(d-1)}_R)} + \lVert u\circ\varphi^{-1}_{\partial B}\rVert_{L^2(B^{(d-1)}_R)}.
    \end{equation}
    And since $\Delta_{\partial B}u\circ\varphi^{-1}_{\partial B}\in H^{k-2}(B^{(d-1)}_R)$, we can apply the same argument iteratively to get
    \begin{align}
        &\lVert u\circ\varphi^{-1}_{\partial B}\rVert_{H^k(B^{(d-1)}_R)}\\ 
        &\,\lesssim \begin{cases}
            \sum_{l=0}^{k/2} \lVert (-\Delta_{\partial B})^l u\circ\varphi^{-1}_{\partial B}\rVert_{L^2(B^{(d-1)}_R)}, & k\text{ even} \\
            \sum_{l=0}^{(k-3)/2} \lVert (-\Delta_{\partial B})^l u\circ\varphi^{-1}_{\partial B}\rVert_{L^2(B^{(d-1)}_R)} + \lVert (-\Delta_{\partial B})^{(k-1)/2} u\circ\varphi^{-1}_{\partial B}\rVert_{H^1(B^{(d-1)}_R)}, & k\text{ odd}
        \end{cases} \nonumber\\
        \begin{split}\label{elliptic_regularity_case_by_case}
            &\,\lesssim 
            \begin{cases}
                \lVert (-\Delta_{\partial B})^{k/2} u\circ\varphi^{-1}_{\partial B}\rVert_{
                L^2(B^{(d-1)}_R)} + \lVert u\circ\varphi^{-1}_{\partial B}\rVert_{
                L^2(B^{(d-1)}_R)}, & k\text{ even} \\
                \lVert (-\Delta_{\partial B})^{(k-1)/2} u\circ\varphi^{-1}_{\partial B}\rVert_{H^1(B^{(d-1)}_R)} + \lVert (-\Delta_{\partial B})^{k/2} u\circ\varphi^{-1}_{\partial B}\rVert_{L^2(B^{(d-1)}_R)} + \lVert u\circ\varphi^{-1}_{\partial B}\rVert_{
                L^2(B^{(d-1)}_R)}, & k\text{ odd}.
            \end{cases}
        \end{split}
    \end{align}
    In the case of $k$ being even, it suffices to use that $1\leq \sqrt{\lvert g\rvert}$ on $B^{(d-1)}_R$ to get 
    \begin{align*}
        \lVert u\circ\varphi^{-1}_{\partial B}\rVert_{H^k(B^{(d-1)}_R)} &\lesssim \lVert (-\Delta_{\partial B})^{k/2} u\circ\varphi^{-1}_{\partial B}\,\lvert g\rvert^{1/4}\rVert_{
            L^2(B^{(d-1)}_R)} + \lVert u\circ\varphi^{-1}_{\partial B}\,\lvert g\rvert^{1/4}\rVert_{
            L^2(B^{(d-1)}_R} \\
            &\leq \lVert (-\Delta_{\partial B})^{k/2} u\circ\varphi^{-1}_{\partial B}\,\lvert g\rvert^{1/4}\rVert_{
            L^2(\R^{d-1})} + \lVert u\circ\varphi^{-1}_{\partial B}\,\lvert g\rvert^{1/4}\rVert_{
            L^2(\R^{d-1})} \\
            &\leq \lVert u\rVert_{H^k(\partial B)},
    \end{align*}
    which was the claim. For odd $k$, however, there is one derivative left to convert into half a Laplace--Beltrami operator. This follows essentially by $1\leq \sqrt{\lvert g\rvert}$ once again, the fact that $g$ is uniformly positive definite on $B^{(d-1)}_R$ and the integration by parts formula for manifolds:
    \begin{align*}
        \lVert (-\Delta_{\partial B})^{(k-1)/2} u\circ\varphi^{-1}_{\partial B}\rVert_{H^1(B^{(d-1)}_R)}^2 &= \int_{B^{(d-1)}_R} \lVert \nabla(-\Delta_{\partial B})^{(k-1)/2} u\circ\varphi^{-1}_{\partial B}\rVert^2 \diff\lambda \\
        &\lesssim \int_{\R^{d-1}} \lVert g\nabla(-\Delta_{\partial B})^{(k-1)/2} u\circ\varphi^{-1}_{\partial B}\rVert^2 \sqrt{\lvert g\rvert}\diff\lambda \\
        &=\int_{\partial B} \left(\nabla_{\partial B}(-\Delta_{\partial B})^{(k-1)/2} u\right)\left(\nabla_{\partial B}(-\Delta_{\partial B})^{(k-1)/2} u\right) \diff\sigma \\
        &= \int_{\partial B} (-\Delta_{\partial B})^{(k-1)/2} u(-\Delta_{\partial B})^{(k+1)/2} u \diff\sigma \\
        &= \lVert (-\Delta_{\partial B})^{k/2}u\rVert_{L^2(\partial B)}.
    \end{align*}
    Inserting the derived bound into \eqref{elliptic_regularity_case_by_case} and using one more time that $\sqrt{\lvert g\rvert}\geq 1$ to the remaining two terms yields the claim for odd $k$ as well.
\end{proof}

\begin{lemma}\label{lem:homogeneous_harmonic_polynomials}
    Let $l\in\N$, $m\in\N\cap[1,M_l]$ and let $p$ be a function of the form
    \begin{equation*}
        p_l(z)\coloneqq \lVert z\rVert^l \sum_{m=1}^{M_l}a_{lm}Y_{lm}(z/\lVert z\rVert),\qquad z\in \R^d.
    \end{equation*}
    Then,
    \begin{equation*}
        \langle \nabla p_l, \nabla p_k\rangle_{H^\alpha(\partial B)} = \delta_{lk}l(l+d-1)(1+((l-1)(l+d-3))^\alpha) \sum_{m=1}^{M_l} \lvert a_{lm}\rvert^2.
    \end{equation*}
\end{lemma}

\begin{proof}
    In the following, let $z\in\R^d$, $r=\lVert z\rVert$ and $x=z/\lVert z\rVert$.

    First, we note that
    \begin{equation*}
        \langle \nabla p_l, \nabla p_k\rangle_{H^\alpha(\partial B)} = \langle \nabla p_l, \nabla p_k\rangle_{L^2(\partial B)} + \langle \nabla p_l, \Delta_{\partial B}^\alpha\nabla p_k\rangle_{L^2(\partial B)}.
    \end{equation*}
    The first term can be evaluated straightforwardly, using $\nabla p_l(x)=\sum_{m=1}^{M_l}a_{lm}(lY_{lm}(x)x+\nabla_{\partial B}Y_{lm}(x))$:
    \begin{align}
        \langle \nabla p_l, \nabla p_k\rangle_{L^2(\partial B)} &= \int_{\partial B} \nabla p_l(y) \cdot\nabla p_k(y)\,\sigma(\diff y) \nonumber\\
        &= \sum_{m=1}^{M_l}\sum_{m'=1}^{M_k} a_{lm}a_{km'} \int_{\partial B} lkY_{lm}(y)Y_{km'}(y)\lVert y\rVert^2 + l \underbrace{\nabla_{\partial B}Y_{lm}(y)\cdot y}_{=0} + k\underbrace{\nabla_{\partial B}Y_{km'}(y)\cdot y}_{=0} \nonumber\\
        &\qquad\qquad\qquad + \nabla_{\partial B}Y_{lm}(y)\cdot\nabla_{\partial B}Y_{km'}(y)\,\sigma(\diff y) \nonumber\\
        &=\sum_{m=1}^{M_l} \sum_{m'=1}^{M_k} a_{lm}a_{km'} \left(l^2\delta_{lk}\delta_{mm'} - \int_{\partial B}Y_{lm}(y) \Delta_{\partial B}Y_{km'}(y)\,\sigma(\diff y) \right) \nonumber\\
        &=\delta_{lk}\sum_{m=1}^{M_l} \lvert a_{lm}\rvert^2 (l^2+l(l+d-2)). \label{p_l_p_k_orthogonal}
    \end{align}
    For the second term, the idea consists in relating the Laplace--Beltrami operator of $\nabla p_l$ on the sphere to its regular Laplacian on the ambient space, since the latter vanishes:
    \begin{align*}
        \Delta p_l(z) &= \partial_r^2p_l(r,x)+\frac{d-1}{r}\partial_r p_l(r,x)+\frac{1}{r^2}\Delta_{\partial B}p_l(r,x) \\
        &=\sum_{m=1}^{M_l} a_{lm}[l(l-1)r^{l-2}Y_{lm}(x) + (d-1)lr^{l-2}Y_{lm}(x) + r^{l-2}\Delta_{\partial B}Y_{lm}(x)] \\
        &=r^{l-2}[l(l+d-2)-l(l+d-2)]\sum_{m=1}^{M_l} a_{lm} Y_{lm}(x)=0,\\
        \implies& \Delta\nabla p_l = \nabla\Delta p_l=0.
    \end{align*}
    Writing $\Delta\nabla p_l$ in spherical coordinates again, we can solve for $\Delta_{\partial B}\nabla p_l$ to get
    \begin{align*}
        \Delta_{\partial B}\nabla p_l(x) &= -r^2 \left(\partial_r^2+\frac{d-1}{r}\partial_r \right)\nabla p_l(r,x)\Big|_{r=1} \\
        &=-r^2\left((l-1)(l-2)r^{l-3} + (d-1)(l-2)r^{l-3} \right)\sum_{m=1}^{M_l} a_{lm} \big[lY_{lm}(x)x + \nabla_{\partial B}Y_{lm}(x)\big]\Big|_{r=1} \\
        &=-(l-1)(l+d-3) r^{l-1}\sum_{m=1}^{M_l} a_{lm} \big[lY_{lm}(x)x+\nabla_{\partial B} Y_{lm}(x)\big]\Big|_{r=1} \\
        &=-(l-1)(l+d-3) \nabla p_l(x).
    \end{align*}
    In other words, the components of the regular gradient $\partial_{x_i}p_l$ are eigenfunctions of the Laplace--Beltrami operator for the eigenvalue $(l-1)(l+d-3)$. In particular, $\partial_{x_i}p_l$ and $\partial_{x_i}p_k$ for $k\neq l$ are orthogonal, which is consistent with \eqref{p_l_p_k_orthogonal}. Thus,
    \begin{align*}
        \langle \nabla p_l, \Delta_{\partial B}^\alpha\nabla p_k\rangle_{L^2(\partial B)}=((l-1)(l+d-3))^\alpha\langle \nabla p_l, \nabla p_k\rangle_{L^2(\partial B)}=((l-1)(l+d-3))^\alpha l(l+d-1)\delta_{lk},
    \end{align*}
    which concludes the proof.
\end{proof}

\begin{lemma}\label{lem:Green_function_bound}
    Let $R\in (0,1)$ be such that $R>\lVert z_0\rVert$ and $\varepsilon\leq 1- R$.
    Then there exists a constant $C>0$ independent of $\varepsilon$ such that for $ R\leq\lVert z\rVert\leq 1$
    \begin{equation*}
        G_{1-\varepsilon}(z,z_0)\leq C(1-\lVert z\rVert).
    \end{equation*}
\end{lemma}

\begin{proof}
    Since $d\geq 3$, the Green function $G_{1-\varepsilon}$ is given by \cite{Gilbarg_Trudinger2001}
    \begin{equation}\label{eq:green_kernel}
        G_{1-\varepsilon}(z,z_0) = \frac{1}{d(d-2)\sigma(\partial B)}\left(\frac{1}{\lVert z-z_0\rVert^{d-2}} - \frac{1}{\lVert z-z_0^\ast\rVert^{d-2}}\left(\frac{1-\varepsilon}{\lVert z_0\rVert} \right)^{d-2}\right), \quad z\in\R^d\setminus\{z_0,z_0^\ast\},
    \end{equation}
    with $z^\ast_0\coloneqq (1-\varepsilon)^2z_0/\lVert z_0\rVert^2$. Away from $z_0$ and $z^*_0$, the Green kernel is arbitrarily smooth, thence it is  bounded and differentiable with bounded derivative on $B\setminus B_R$. As a function of the radial coordinate $r=\lVert z\rVert$, it is differentiable as well with $\partial_r G_{1-\varepsilon}(z,z_0)=\nabla G_{1-\varepsilon}(z,z_0)\cdot z/\lVert z\rVert$. By the (one-dimensional) fundamental theorem of calculus, we therefore get for any $x\coloneqq z/\lVert z\rVert\in \partial B$
    \begin{align*}
        G_{1-\varepsilon}(z,z_0) &= \underbrace{G_{1-\varepsilon}((1-\varepsilon)x,z_0)}_{=0} - \int_{\lVert z\rVert}^{1-\varepsilon} \partial_r G_{1-\varepsilon}(rx,z_0)\diff r \\
        &\leq \sup_{y\in B\setminus B_R} \lVert\nabla G_{1-\varepsilon}(y,z_0) \rVert \, (1-\varepsilon-\lVert z\rVert) \\
        &\leq \underbrace{\sup_{y\in B\setminus B_R} \lVert\nabla G_{1-\varepsilon}(y,z_0)\rVert}_{\eqqcolon C_\varepsilon} \, (1-\lVert z\rVert).
    \end{align*}
    For $\varepsilon\leq 1-R$, the map $\varepsilon\mapsto\lVert\nabla G_{1-\varepsilon}(y,z_0)\rVert$ is uniformly bounded as well, meaning
    \begin{equation*}
        C\coloneq \sup_{\varepsilon \leq 1 - R}\sup_{y \in B\setminus B_R} \lVert \nabla G_{1-\varepsilon}(y,z_0) \rVert < \infty.
    \end{equation*}
\end{proof}

\section{Basic results in neural network approximation theory}\label{app:neural}
Here we summarise some important known results on approximation properties of neural networks from the class $\mathrm{NN}(\mathsf{L,W,S,B})$ introduced in \eqref{eq:neural_nets}.
\begin{lemma}[composition of neural networks]\label{lem:compsition_of_neural_networks}
    Let $\phi_1\in \mathrm{NN}(\mathsf{L_1,W_1,S_1,B_1})$ and $\phi_2\in \mathrm{NN}(\mathsf{L_2,W_2,S_2,B_2})$. Then its composition $\phi=\phi_1\circ\phi_2$ has size
    \begin{align*}
        \mathsf L\leq \mathsf L_1+\mathsf L_2+1, && \mathsf W= \mathsf W_1\lor \mathsf W_2, && \mathsf S=\mathsf S_1+\mathsf S_2, && \mathsf B=\mathsf B_1\lor \mathsf B_2.
    \end{align*}
\end{lemma}

\begin{lemma}[parallelisation of neural networks, \cite{oko23} Lemma~F.3] \label{lem:parallelisation_of_neural_networks}
    Let $\phi_1\in \mathrm{NN}(\mathsf{L_1,W_1,S_1,B_1})$ and $\phi_2\in \mathrm{NN}(\mathsf{L_2,W_2,S_2,B_2})$. Then the parallelisation $\phi=(\phi_1,\phi_2)$ has size
    \begin{align*}
        \mathsf L\leq \mathsf L_1\lor \mathsf L_2, && \mathsf W=2 (\mathsf W_1 + \mathsf W_2), && \mathsf S=2(\mathsf {S_1+L W_1 + S_2+LW_2}), && \mathsf B=\mathsf B_1\lor \mathsf B_2.
    \end{align*}
\end{lemma}

\begin{lemma}[sum of neural networks, \cite{oko23} Lemma~F.3] \label{lem:sum_of_neural_networks}
    Let $\phi_i\in \mathrm{NN}(\mathsf{L_i,W_i,S_i,B_i})$, $i=1,\dots,m$. Then the sum $\phi=\sum_{i=1}^m\phi_i$ is a neural network with size
    \begin{align*}
        \mathsf L\leq \mathsf \max_{i=1,\dots,m} L_i + 1, && \mathsf W=4 \sum_{i=1}^m\mathsf W_i, && \mathsf S=4 \sum_{i=1}^m(\mathsf {S_i+L W_i})+2m, && \mathsf B=\mathsf \max_{i=1,\dots,m}\mathsf B_i.
    \end{align*}
\end{lemma}

\begin{lemma}[minimum and maximum function] \label{lem:maximum_and_minimum_function}
    The minimum and maximum function can be expressed as ReLU networks in the class $\mathrm{NN}(1,3,7,1)$:
    \begin{align*}
        \max(x,y)&=y+\mathrm{ReLU}(x-y), && \min(x,y)=x-\mathrm{ReLU}(x-y), && x,y\in\R. \\
        &=\mathrm{ReLU}(y)-\mathrm{ReLU}(-y) \\
        &\qquad+\mathrm{ReLU}(x-y)
    \end{align*}
\end{lemma}

\begin{lemma}[approximation of multiplication function, \cite{Asbjorn_2025} Lemma~C.1] \label{lem:multiplication_function}
    For $l_1\in\N$ and $C\geq 1$, there exists a neural network $\phi_{\mathrm{mult}}\in \mathrm{NN}(l_1,1,l_1,C)$ s.t.
    \begin{align*}
        \lvert \phi^{\mathrm{mult}}_m(x,y)-xy \rvert \leq C2^{-l_1}, \qquad x\in[0,1],\quad y\in[-C,C].
    \end{align*}
\end{lemma}

\begin{lemma}[approximation of monomials, \cite{Elbraechter2021} Proposition~III.5]\label{lem:polynomial_function_approximation}
    For $\delta>0$ and $l\in\N_0$, there exists $\phi_{\mathrm{mon}}^{(l)}\in\mathrm{NN}(l(\log \delta^{-1}+\log l),9,l(\log \delta^{-1}+\log l),1)$ s.t.
    \begin{equation*}
        \lvert \phi_{\mathrm{mon}}^{(l)}(r)-r^l\rvert \leq \delta, \qquad r\in [0,1].
    \end{equation*}
\end{lemma}

\begin{lemma}[approximation of reciprocal function, \cite{Asbjorn_2025} Lemma~C.2]\label{lem:reciprocal_function}
    For $l_2,\underline{k},\overline{k}\in\N$, there exists a neural network $\phi_{\mathrm{rec}}\in\mathrm{NN}((k+l_2)\log(k+l_2),k,(k+l_2)\log(k+l_2),2^k)$ with $k=\underline{k}+\overline{k}$ such that
    \begin{equation*}
        \lvert \phi_{\mathrm{rec}}(x)-x^{-1}\rvert\leq 2^{l_2}, \qquad x\in[2^{-\underline{k}},2^{\overline{k}}].
    \end{equation*}
\end{lemma}

As an almost direct consequence of the previous lemma, one can get the following result, which will be needed at the end of the proof on the score approximation.

\begin{lemma}[growth control of neural network]\label{lem:growth_neural_network}
    For $m\in\N$, there exists a neural network $\phi^{\mathrm{cap}}\in\mathrm{NN}((m+1)\log(m+1),m+1,(m+1)\log(m+1),2^{m+1})$ such that
    \begin{align*}
        \frac{2(d+2)}{1-r}\leq\phi^{\mathrm{cap}}(r) \leq\frac{6(d+2)}{1-r},\qquad r\in[0,1-2^{-m}].
    \end{align*}
\end{lemma}

\begin{proof}
    The proof is a simpler version of \autocite[Lemma~3.12]{Asbjorn_2025}. Invoking Lemma~\ref{lem:reciprocal_function} with $l_2=1$, $\underline{k}=m$, and $\overline{k}=0$, there exists a neural network $\tilde\phi^{\mathrm{cap}}$ such that
    \begin{align*}
        &\qquad\lvert \phi_{\mathrm{rec}}(x)-x^{-1}\rvert\leq \frac{1}{2}\leq\frac{1}{2x}, \quad x\in[2^{-m},1] \\
        &\Rightarrow \frac{1}{2x} = \frac{1}{x} -\frac{1}{2x}\leq \tilde\phi^{\mathrm{cap}}(x)\leq \frac{1}{x}+\frac{1}{2x} = \frac{3}{2x}.
    \end{align*}
    Multiplying with $4(d+2)$, setting $r=1-x$, and $\phi^{\mathrm{cap}}(r)\coloneq 4(d+2)\tilde\phi^{\mathrm{cap}}(1-r)$ yields the claim.
\end{proof}

\printbibliography

\end{document}